\pdfoutput=1
\documentclass[10pt]{article}
\usepackage{amsmath,amsthm,amssymb} 
\usepackage{flexisym} 
\usepackage{breqn}
\usepackage{lmodern}
\usepackage{avant}
\usepackage{courier}

\usepackage[T1]{fontenc}
\usepackage[a4paper]{geometry}
\usepackage{mathtools}   
\usepackage{float}
\usepackage{appendix}
\usepackage{babel}
\usepackage{stmaryrd}
\usepackage{setspace}
\usepackage[numbers]{natbib}
\PassOptionsToPackage{normalem}{ulem}
\usepackage{ulem}
\usepackage[unicode=true,pdfusetitle,
 bookmarks=true,bookmarksnumbered=false,bookmarksopen=false,
 breaklinks=false,pdfborder={0 0 0},pdfborderstyle={},backref=false,colorlinks=true]
 {hyperref}
\hypersetup{
 pdfborderstyle={},pdfborderstyle={},linkcolor=blue,urlcolor=blue,citecolor=blue}
\usepackage{units}
\usepackage{bbold}
\usepackage{authblk} 
\allowdisplaybreaks
 
 \newcommand{\idd}{\mathbb{1}}

\newtheorem{thm}{Theorem}
\newtheorem{prop}{Proposition}%
\newtheorem{lem}{Lemma}%
\newtheorem{cor}{Corollary}%

\theoremstyle{remark}%
\newtheorem{rem}{Remark}%

\theoremstyle{definition}%
\newtheorem{defn}{Definition}%
\newtheorem{assumption}{Assumption}%
\newtheorem{example}{Example}%

\title{Combinatorial Reduction of Set Functions and Matroid Permutations through Minor 
Product Assignment}

\author[1,2]{Mario Angelelli}
\affil[1]{Department of Innovation Engineering, University of Salento, Complesso Ecotekne - Via per Monteroni, Lecce, 73100, Italy}
\affil[2]{INdAM (National Institute of Higher Mathematics) - GNSAGA, Italy}

\date{}

\begin{document}
\maketitle 

\begin{abstract}
We introduce an algebraic model, based on the determinantal expansion of the product of two matrices, to test combinatorial reductions of set functions. Each term of the determinantal expansion is deformed through a monomial factor in $d$ indeterminates, whose exponents define a $\mathbb{Z}^{d}$-valued set function. By combining the Grassmann-Pl\"{u}cker relations for the two matrices, we derive a family of sparse polynomials, whose factorisation properties in a Laurent polynomial ring are studied and related to information-theoretic notions. 

Under a given genericity condition, we prove the equivalence between combinatorial reductions and determinantal expansions with invertible minor products; specifically, a deformation returns a determinantal expansion if and only if it is induced by a diagonal matrix of units in $\mathbb{C}(\mathbf{t})$ acting as a kernel in the original determinant expression. This characterisation supports the definition of a new method for checking and recovering combinatorial reductions for matroid permutations. 
\end{abstract}

\section{\label{sec: Introduction} Introduction}

A natural question about set functions regards conditions 
that imply a form of complexity reduction: given sets $\mathcal{X}$ and $M$, a family $\mathfrak{X}$ of subsets of $\mathcal{X}$, and a map $\Psi:\,\mathfrak{X}\longrightarrow M$, we ask whether there is an underlying function $\psi:\,\mathcal{X}\longrightarrow M$ such that, for all $\mathcal{I}\in\mathfrak{X}$, we can recover $\Psi(\mathcal{I})$ from
$\{\psi(\alpha):\,\alpha\in\mathcal{I}\}$. The additivity axiom for countable probability distributions is a fundamental example of such a condition; but we can extend this investigation to more general cases where $M$ is an $R$-module for a ring $R$. In this way, we can reconstruct $\Psi$ from $\psi$ by taking advantage of algebraic relations. 

This work focuses on determinant expansions as a means to encode information about set functions, relying on maximal minors of two $(k\times n)$-dimensional full-rank complex matrices to specify the type of combinatorial reduction. In fact, Grassmann-Pl\"{u}cker relations among these minors are algebraic conditions entailing a reduction for $k$-forms on $\mathbb{C}^{n}$ \citep[Ch.3.1]{GKZ1994}: formally, they characterise totally de\-com\-po\-sa\-ble $k$-forms $v_{i_{1}}\wedge\dots \wedge v_{i_{k}}\in\bigwedge^{k}\mathbb{C}^{n}$, which are labelled by $k$-subsets of $\{i_{1},\dots,i_{k}\}\subseteq\{1\,\dots,n\}$ and are expressed as the wedge product of $1$-forms $v_{i_{1}},\dots,v_{i_{k}}$. 

The motivation behind the study of such determinantal relations is the wide range of applications of combinatorial, algebraic, and geometric properties of determinantal varieties and Grassmann-Pl\"{u}cker relations. To elucidate the main notions underlying this research, we start by recalling the explicit connection between determinantal varieties and one of these application areas. 
\subsection{Relations with previous work: determinantal expansions of Wronskian soliton solutions}
\label{subsec: determinantal expansions and soliton solutions}
%
The bilinear form of the Kadomtsev-Petviashvili (KP) II hierarchy, which is a family of PDEs expressing Grassmann-Pl\"{u}cker relations for an infinite-dimensional space \citep{Harnad2020},  
includes a special class of solutions ($\tau$-functions) with a rich combinatorial structure \citep{Kodama2017}, namely, Wronskian solutions of the type  
 \begin{equation}
\tau(\mathbf{x})=\det(\mathbf{A}\cdot{\ensuremath{\Theta}}(\mathbf{x})\cdot\mathbf{K})
\label{eq: Wronskian solutions}
 \end{equation}
where $\mathbf{A}$ is a matrix of constant coefficients, 
 \begin{equation}
\Theta:=\mathrm{diag}\left(\exp\left(\sum_{r=1}^{d}\kappa_{1}^{r}x_{r}\right),\dots,\exp\left(\sum_{r=1}^{d}\kappa_{n}^{r}x_{r}\right)\right)
\label{eq: diagonal matrix of exponents}
 \end{equation}
and $\mathbf{K}$ is the Vandermonde matrix associated with the $n$-tuple $(\kappa_{1},\dots,\kappa_{n})$. For these solutions of the KP II hierarchy, the determinant (\ref{eq: Wronskian solutions}) is converted into an exponential sum through the well-known Cauchy-Binet expansion for two $(k\times n)$-dimensional matrices $\mathbf{A},\mathbf{K}^{\mathtt{T}}$ over a ring ($k\leq n$) 
 \begin{equation}
\det(\mathbf{A}\cdot\mathbf{K})=\sum_{\mathcal{I}\in \wp_{k}[n]}\Delta_{\mathbf{A}}(\mathcal{I})\cdot\Delta_{\mathbf{K}}(\mathcal{I}).
\label{eq: Cauchy-Binet expansion}
 \end{equation} 
where $\wp_{k}[n]:=\{\mathcal{I}\subseteq \{1,\dots,n\}:\,\#\mathcal{I}=k\}$ and $\Delta_{\mathbf{A}}(\mathcal{I})$ (respectively, $\Delta_{\mathbf{K}}(\mathcal{I})$) is the maximal minor of $\mathbf{A}$ (respectively, $\mathbf{K}$) extracted from columns (respectively, rows) indexed by $\mathcal{I}\subseteq\{1,\dots,n\}$. From (\ref{eq: Cauchy-Binet expansion}), the solution (\ref{eq: Wronskian solutions}) is expressed as  
 \begin{equation}
\det(\mathbf{A}\cdot{\ensuremath{\Theta}}(\mathbf{x})\cdot\mathbf{K})=\sum_{\mathcal{I}\in \wp_{k}[n]}\Delta_{\mathbf{A}}(\mathcal{I})\cdot\Delta_{\mathbf{K}}(\mathcal{I})\cdot\mathrm{e}^{\sum_{\alpha\in\mathcal{I}}\sum_{r=1}^{d}\kappa_{\alpha}^{r}\cdot x_{r}} 
\label{eq: Cauchy-Binet expansion of Wronskian}
 \end{equation} 
so the $\tau$-function is a superposition of elementary waves (exponential functions) whose phases are set functions defined by the soliton parameters. 

The combinatorial properties derived from (\ref{eq: Cauchy-Binet expansion of Wronskian}) (see e.g. \cite{Kodama2017}) suggest broadening the exploration of the information provided by \emph{deformations} of the individual terms of such expansions that are compatible with the determinantal constraints, i.e. return another determinant expansion. The introduction of deformations is a way to investigate the rigidity of a given structure in terms of possible configurations that are consistent with algebraic constraints. In \citep{Angelelli2019}, we focused on the information-theoretic aspects of expressions (\ref{eq: Cauchy-Binet expansion of Wronskian}) preserving the KP II equation under discrete deformations 
 \begin{equation}
\Delta_{\mathbf{A}}(\mathcal{I})\mapsto\sigma_{\mathcal{I}}\cdot\Delta_{\mathbf{A}}(\mathcal{I}),\quad\sigma_{\mathcal{I}}\in\{1,-1\}\quad \mathcal{I}\in\wp_{k}[n].
\label{eq: sign of minor}
 \end{equation}
It should be noted that other determinantal expansions and their variations have been addressed in the literature, highlighting complexity aspects and computational advantages arising from the symmetries of the determinant, compared to other immanants such as the permanent \citep{Valiant1979}.

\subsection{Scope of this work: from algebraic to combinatorial conditions on set functions}
\label{subsec: scope of this work}

In this paper, we extend the previous investigation by focusing on algebraic and combinatorial properties entailed by deformations  
\begin{equation}
\Delta_{\mathbf{A}}(\mathcal{I})\cdot\Delta_{\mathbf{K}}(\mathcal{I}) \mapsto \Delta_{\mathbf{A}}(\mathcal{I})\cdot\Delta_{\mathbf{K}}(\mathcal{I}) \cdot c_{\mathcal{I}}\cdot \mathbf{t}^{\mathbf{e}(\mathcal{I})}, \quad \mathcal{I}\in\wp_{k}[n],\,\Delta_{\mathbf{A}}(\mathcal{I})\neq 0,\, c_{\mathcal{I}}\in\mathbb{C}\setminus\{0\}
\label{eq: toric deformations} 
\end{equation}
of each minor product in (\ref{eq: Cauchy-Binet expansion}) through a monomial in $d$ indeterminates, where we have adopted the notation 
\begin{equation}
    \mathbf{t}^{\mathbf{e}(\mathcal{I})}:=\prod_{u=1}^{d} t_{u}^{e_{u}(\mathcal{I})}.
    \label{eq: monomial notation}
\end{equation} 
Note that the deformations (\ref{eq: sign of minor}) studied in \cite{Angelelli2019} can be seen as a reduction of (\ref{eq: toric deformations}), choosing $d=1$ and specifying $t_{1}=-1$ and $c_{\mathcal{I}}=1$ for all the sets $\mathcal{I}\in\wp_{k}[n]$ that contribute to (\ref{eq: toric deformations}). Moreover, (\ref{eq: toric deformations}) generalises the functional form of the phases in (\ref{eq: Cauchy-Binet expansion of Wronskian}) for rational soliton parameters $\kappa_{1},\dots,\kappa_{n}$, as follows from the change of variables $\mathrm{e}^{x_{\alpha}/M}=:t_{\alpha}$ for all $\alpha\in\{1,\dots,d\}$ and an appropriate choice of $M\in\mathbb{N}$. 

We recall that, given a matrix $\mathbf{A}\in\mathbb{C}^{k\times n}$, the set 
 \begin{equation}
\mathfrak{G}(\mathbf{A}):=\left\{ \mathcal{I}\in\wp_{k}[n]:\,\Delta_{\mathbf{A}}(\mathcal{I})\neq0\right\}.
\label{eq: set of non-vanishing minors}
 \end{equation}
is a matroid, i.e. a non-empty set characterised by the following exchange relation \citep{Oxley2006} 
 \begin{equation}
\forall\,\mathcal{A},\mathcal{B}\in\mathfrak{G}(\mathbf{A}),\,\alpha\in\mathcal{A}\setminus\mathcal{B}:\,\exists\,\beta\in\mathcal{B}\setminus\mathcal{A}.\mathcal{A}_{\beta}^{\alpha}\in\mathfrak{G}(\mathbf{A}).
\label{eq: exchange relation}
 \end{equation}
Then, a deformation (\ref{eq: toric deformations}) can be described by a mapping that associates each subset $\mathcal{I}\in\mathfrak{G}(\mathbf{A})$ with an element $\Psi(\mathcal{I}):=(e_{1}(\mathcal{I}),\dots,e_{d}(\mathcal{I}))\in\mathbb{Z}^{d}$. From this, we can explore 
the existence of a ``potential''  $\psi:\,\{1,\dots,n\}\longrightarrow\mathbb{Z}^{d}$ such that $\Psi(\mathcal{I})$ can be uniquely recovered from values $\psi(\alpha)$ with $\alpha\in\mathcal{I}$. 

To relate $\Psi$ and $\psi$, we use the $\mathbb{Z}$-module structure of $\mathbb{Z}^{d}$ looking for additive and affine set functions. Additivity generalises (\ref{eq: Cauchy-Binet expansion of Wronskian}) and, in turn, extends to affine set functions since the set of configurations $\mathbf{e}$ in (\ref{eq: toric deformations}) that return the terms of a determinantal expansion is closed under translations $\mathbf{e}\mapsto \mathbf{e}+\mathbf{m}_{0}$ for any constant $\mathbf{m}_{0}\in\mathbb{Z}^{d}$. More importantly, affine set functions let us introduce consistency (another notion often related to integrability, see e.g. \citep{Bobenko2004}) in our combinatorial context: for all $\mathcal{H}\subseteq\{1,\dots,n\}$ with $\#\mathcal{H}=k-2$ and $\alpha_{1},\alpha_{2},\beta_{1},\beta_{2}\in\{1,\dots,n\}\setminus\mathcal{H}$, the object of interest is the quantity 
 \begin{equation}
\Psi\left(\mathcal{H}\cup\{\alpha_{1},\alpha_{2}\}\right)+\Psi\left(\mathcal{H}\cup\{\beta_{1},\beta_{2}\}\right)-\Psi\left(\mathcal{H}\cup\{\alpha_{1},\beta_{2}\}\right)-\Psi\left(\mathcal{H}\cup\{\beta_{1},\alpha_{2}\}\right)
\label{eq: consistency and curvature}
 \end{equation}
whenever the arguments lie in the domain of $\Psi$, that is, $\mathfrak{G}(\mathbf{A})$. Setting (\ref{eq: consistency and curvature}) equal to zero entails the consistency of the differences $\Psi(\mathcal{I}\setminus\{\alpha_{1}\}\cup\{\beta_{1}\})-\Psi(\mathcal{I})$ for two different choices of the basis, indexed by $\mathcal{I}:=\mathcal{H}\cup\{\alpha_{1},\alpha_{2}\}$ or $\mathcal{I}:=\mathcal{H}\cup\{\alpha_{1},\beta_{2}\}$, respectively. When, in addition, we have $\mathcal{H}\cup\{\alpha_{u},\beta_{u}\}\in\mathfrak{G}(\mathbf{A})$ for both $u\in\{1,2\}$, we also get consistency between two different sequences $(\alpha_{1},\alpha_{2})\rightarrow(\alpha_{u},\beta_{u})\rightarrow(\beta_{1},\beta_{2})$, defined by $u\in\{1,2\}$, of single-index exchanges connecting the bases labelled by $\mathcal{H}\cup\{\alpha_{1},\alpha_{2}\}$ and $\mathcal{H}\cup\{\beta_{1},\beta_{2}\}$, respectively.

\subsection{Overview of main results and implications for permutation testing}
\label{subsec: overview and main result}

The first contribution of this work is to characterise the set functions that are compatible with the determinantal expansion.  
\begin{thm}
\label{thm: rigidity theorem monomial}
Let $\mathbf{L}(\mathbf{t}),\mathbf{R}(\mathbf{t})^{\mathtt{T}}$ be two $(k\times n)$-dimensional matrices depending on $d$ indeterminates $\mathbf{t}$ such that $\max\{n-k,k\}\geq 5$, $\mathbf{R}(\mathbf{t})$ is generic (i.e. no maximal minor vanishes for a generic choice of $\mathbf{t}$), and $\mathbf{L}(\mathbf{1})\in\mathbb{C}^{k\times n}$ has two generic columns, that is, there exist $\mathcal{I}\in\wp_{k}[n]$ and $\alpha_{1},\alpha_{2}\in\{1,\dots,n\}\setminus\mathcal{I}$ such that
\begin{equation}
    \mathcal{J}\setminus\mathcal{I}\subseteq \{\alpha_{1},\alpha_{2}\} \Rightarrow \Delta_{\mathbf{L}(\mathbf{1})}(\mathcal{J})\neq0,\quad \mathcal{J}\in\wp_{k}[n].
\label{eq: two generic columns condition}
\end{equation}
If the terms in the Cauchy-Binet expansion of $\det(\mathbf{L}(\mathbf{t})\cdot\mathbf{R}(\mathbf{t}))$ satisfy the monomial condition 
 \begin{equation}
\Delta_{\mathbf{L}(\mathbf{t})}(\mathcal{I})\cdot\Delta_{\mathbf{R}(\mathbf{t})}(\mathcal{I})=g_{\mathcal{I}}\cdot \mathbf{t}^{\Psi(\mathcal{I})},\quad\mathcal{I}\in\mathfrak{G}(\mathbf{L}(\mathbf{1})),\,g_{\mathcal{I}}\in\mathbb{C},\,\Psi(\mathcal{I})\in\mathbb{Z}^{d}
\label{eq: monomial terms of Cauchy-Binet expansion}
 \end{equation}
then $\Psi$ is an affine set function, that is, there exist an element $\mathbf{m_{0}}\in\mathbb{Z}^{d}$ and a map $\psi:\,\{1,\dots,n\}\longrightarrow\mathbb{Z}^{d}$ such that
 \begin{equation}
\Delta_{\mathbf{L}(\mathbf{t})}(\mathcal{I})\cdot\Delta_{\mathbf{R}(\mathbf{t})}(\mathcal{I})=\mathbf{t}^{\mathbf{m}_{0}}\cdot\Delta_{\mathbf{L}(\mathbf{1})}(\mathcal{I})\cdot\Delta_{\mathbf{R}(\mathbf{1})}(\mathcal{I})\prod_{\alpha\in\mathcal{I}}\mathbf{t}^{\psi(\alpha)},\quad\mathcal{I}\in\wp_{k}[n].
\label{eq: set-to-element integrability}
 \end{equation} 
\end{thm}
Equivalently, the genericity condition on $\mathfrak{G}(\mathbf{L}(\mathbf{1}))$ in (\ref{eq: two generic columns condition}) guarantees that the compatibility with determinantal expansions for a generic $\mathbf{R}(\mathbf{t})$ is achieved if and only if we can reduce to the expansion of  
\begin{equation} 
\det\left(\mathbf{L}(\mathbf{1})\cdot \mathrm{diag}(\mathbf{t}^{\psi(\alpha)})_{\alpha\in\{1,\dots,n\}} \cdot \mathbf{R}(\mathbf{1})\right) 
\end{equation} 
apart from a common unit $\mathbf{t}^{\mathbf{m}_{0}}$ that is irrelevant in terms of Pl\"{u}cker coordinates. This factor can be absorbed into the determinant by translation $\psi\mapsto \psi +\mathbf{m}_{0}/k$ and change of variables $\mathbf{t}\mapsto \mathbf{s}^{k}$ to deal with units of $\mathbb{C}(\mathbf{t},\mathbf{s})/I$, where $I$ is the ideal generated by binomials $t_{u}-s_{u}^{k}$, $u\in\{1,\dots,d\}$. 

Theorem \ref{thm: rigidity theorem monomial} may fail when there are not enough non-vanishing minors of $\mathbf{L}(\mathbf{t})$: under the genericity conditions for $\mathbf{R}$ and for two columns of $\mathbf{L}$, the reduced number of minors may only come from the low dimensionality of the two matrices. When the genericity condition for $\mathbf{L}(\mathbf{1})$ is violated, we can find a counterexample to Theorem \ref{thm: rigidity theorem monomial}. 
\begin{example} 
\label{exa: principal minor counterexample}
Let us take $k=6$, $n=12$ (so $n-k\geq\max\{5,k\}$), and matrices 
\begin{equation}
    \mathbf{L}_{\text{c}}:= \left(
    \begin{smallmatrix}
 1 & 0 & 0 & 0 & 0 & 0 & 1 & 0 & 0 & 0 & 0 & 0 \\
 0 & 1 & 0 & 0 & 0 & 0 & 0 & 1 & 0 & 0 & 0 & 0 \\
 0 & 0 & 1 & 0 & 0 & 0 & 0 & 0 & 1 & 0 & 0 & 0 \\
 0 & 0 & 0 & 1 & 0 & 0 & 0 & 0 & 0 & 1 & 0 & 0 \\
 0 & 0 & 0 & 0 & 1 & 0 & 0 & 0 & 0 & 0 & 1 & 0 \\
 0 & 0 & 0 & 0 & 0 & 1 & 0 & 0 & 0 & 0 & 0 & 1 \\
    \end{smallmatrix}
    \right),
    \quad
    \mathbf{R}_{\text{c}}^{\mathtt{T}}:= \left(
    \begin{smallmatrix}
 1 & 0 & 0 & 0 & 0 & 0 & t & t+1 & t+1 & t+1 & t+1 & t+1 \\
 0 & 1 & 0 & 0 & 0 & 0 & t-1 & t & t-2 & t+1 & t-6 & t+2 \\
 0 & 0 & 1 & 0 & 0 & 0 & t-1 & t+2 & t & t+5 & t+3 & t+10 \\
 0 & 0 & 0 & 1 & 0 & 0 & t-1 & t-1 & t-5 & t & t-14 & t-7 \\
 0 & 0 & 0 & 0 & 1 & 0 & t-1 & t+6 & t-3 & t+14 & t & t+13 \\
 0 & 0 & 0 & 0 & 0 & 1 & t-1 & t-2 & t-10 & t+7 & t-13 & t \\
    \end{smallmatrix}
    \right)
\end{equation}
It is easily checked that $\Delta_{\mathbf{L}_{\text{c}}}(\mathcal{I})\neq0$ if and only if $\mathcal{I}\setminus\{1,2,3,4,5,6\}$ is obtained from $\{1,2,3,4,5,6\}\setminus\mathcal{I}$ by adding $6$ to all its elements. It follows that condition (\ref{eq: two generic columns condition}) is not verified. On the other hand, we can also check that all non-vanishing minors $\Delta_{\mathbf{L}_{\text{c}}}(\mathcal{I})\cdot \Delta_{\mathbf{R}_{\text{c}}}(\mathcal{I})\neq0$ are monomials in $t$ of degree $0$ or $1$. Evaluating $\mathbf{R}_{\text{c}}$ at $t=9$, all the maximal minors of $\mathbf{R}_{\text{c}}(9)$ are non-vanishing, where $\min_{\mathcal{I}\in\wp_{6}[12]}\mid\Delta_{\mathbf{R}_{\text{c}}(9)}(\mathcal{I})\mid = 1$ is achieved, for example, at $\mathcal{I}=\{1,2,3,4,5,6\}$; by the continuity of the determinant as a function of its arguments, there is a neighbourhood of $t=9$ where $\mathbf{R}_{\text{c}}(t)$ remains generic, so the genericity condition for $\mathbf{R}_{\text{c}}(t)$ is also verified. We see that 
\begin{equation}
    \Psi\left(\{1,2,3,4,5,6\}\right) + \Psi\left(\{3,4,5,6,7,8\}\right)
    - \Psi\left(\{1,3,4,5,6,8\}\right) - \Psi\left(\{2,3,4,5,6,7\}\right) = -2 
\end{equation}
in contradiction to (\ref{eq: consistency and curvature}), which is a necessary condition for (\ref{eq: set-to-element integrability}) to hold. Then, the combinatorial reduction does not take place in this case. 
\end{example}
More general independence structures associated with $\mathbf{L}(\mathbf{1})$ require an additional study of minimal conditions that guarantee the combinatorial reduction (\ref{eq: set-to-element integrability}), and we dedicate a separate work to this investigation.

The algebraic structure resulting from this model combines two quadratic contributions, namely the Grassmann-Pl\"{u}cker relations for two matrices in (\ref{eq: Cauchy-Binet expansion}). This results in a set of quartic equations that constrain the form of the minors, and hence the deformations preserving the determinantal form. Our second contribution delves into the relations between the set of allowed deformations and the factorisation properties of polynomials derived from this combination of Grassmann-Pl\"{u}cker relations, in particular their squarefree decomposition. In turn, we will find that this squarefree decomposition leads to a characterisation of Shannon's entropy for a Bernoulli random variable (see Remark \ref{rem: characterization of the binary entropy function}). These observations highlight an information-theoretic interpretation of the model: $\mathbf{L}(\mathbf{t})$ is a source of structural information represented by the matroid $\mathfrak{G}(\mathbf{L}(\mathbf{1}))$, and a generic matrix $\mathbf{R}(\mathbf{t})$ is selected to explore this structure using accessible information (\ref{eq: monomial terms of Cauchy-Binet expansion}).   

The assumptions of Theorem \ref{thm: rigidity theorem monomial} specify neither the form of $g_{\mathcal{I}}$ nor the exponents $\Psi(\mathcal{I})$, which makes the Cauchy-Binet expansion a \emph{sparse polynomial} \citep{Khovanskii1991} with a given upper bound for its sparsity. This argument does not rely upon the association $\mathcal{I}\mapsto(g_{\mathcal{I}},\Psi(\mathcal{I}))$ or potential reductions arising from the cancellation of terms in the expansion (\ref{eq: Cauchy-Binet expansion}): if we are able to find an assignment of non-vanishing terms $g_{\mathcal{I}}$ that is consistent with (\ref{eq: Cauchy-Binet expansion}) and 
the assumptions of Theorem \ref{thm: rigidity theorem monomial}, then the reduction (\ref{eq: set-to-element integrability}) follows. 

The third contribution of this work exploits Theorem \ref{thm: rigidity theorem monomial} and the previous observation, providing a verification protocol to check the existence and, if so, recover a given permutation of $\{1,\dots,n\}$ when available information is only provided by the unlabelled collection of terms (\ref{eq: monomial terms of Cauchy-Binet expansion}). 
When the set function $\Psi$ corresponds to a permutation $\widehat{\Psi}:\,\mathfrak{G}(\mathbf{L}(\mathbf{1}))\longrightarrow \mathfrak{G}(\mathbf{L}(\mathbf{1}))$, Theorem \ref{thm: rigidity theorem monomial} can be used to check if it is induced by a permutation $\widehat{\psi}:\,\{1,\dots,n\}\longrightarrow\{1,\dots,n\}$ even when we do not have explicit information about $\Psi$, but only through the terms of a candidate determinantal expansion. 
\begin{thm}
\label{thm: check and recover permutation reductions} 
Let $k,n\in\mathbb{N}$ with $k\leq n$ and $\max\{n-k,k\}\geq 5$, and $\widehat{\Psi}:\,\mathfrak{G}\longrightarrow\mathfrak{G}$ be a permutation of a matroid $\mathfrak{G}$ on $\{1,\dots,n\}$ with rank $k$. 
Assume that each $n$-tuple $\mathbf{t}\in\mathbb{C}^{n}$ returns a list, i.e. a tuple standing for an unlabelled multiset of $\#\mathfrak{G}$ complex numbers, which represent the candidate components of a determinantal expansion. Then, we can choose two $n$-tuples $\mathbf{t}_{u}$, $u\in\{1,2\}$, to verify the fulfilment of the assumptions of Theorem \ref{thm: rigidity theorem monomial}, the existence of a permutation $\psi$ of $\{1,\dots,n\}$ that induces $\Psi$ and, in the affirmative case, recover $\psi$ and matrices $\mathbf{a}$ and $\mathbf{q}$ that generate the terms of the expansion. 
\end{thm}
Additional \textit{a priori} knowledge of the domain of the entries of $\mathbf{a}$, e.g. $\mathbb{Z}$ or an algebraic number field, can be used to extend the previous argument, 
enabling specific methods to perform the verification based on less information. This is shown in Corollary \ref{cor: combinatorial reduction from two scalars when integer}, where only two numbers suffice to verify complexity reduction and retrieve two candidate matrices that define the determinantal expansion, if they exist.

This application provides a complementary view to distance-based permutation testing (see e.g. \cite{Fox2018}) and relates to the geometric interpretation of maximal minors, which serve as homogeneous coordinates, in the projectivization $\mathbb{P}(\bigwedge^{k}\mathbb{C}^{n})$, for $k$-subspaces of $\mathbb{C}^{n}$ associated with $(k\times n)$-dimensional full-rank matrices. Indeed, the expansion (\ref{eq: Cauchy-Binet expansion}) arises in the study of principal angles between subspaces of a vector space \citep{Miao1992} and consequent applications \citep{Hamm2008,Dong2022}. Theorem \ref{thm: check and recover permutation reductions} allows us to deal with the lack of explicit information on coordinate labelling, which is a phenomenon that arises in unlabelled sensing or shuffled data \cite{Tsakiris2020,Peng2021}; furthermore, it extracts sufficient information to decompose candidate minor products, identifying in this way two subspaces involved in the determinantal expansion. 

The assumption of monomial deformation of terms in the Cauchy-Binet expansion is equivalent to the invertibility in the ring $\mathbb{C}(\mathbf{t})$ of the non-vanishing principal minors of the composed matrix 
\begin{equation*}
\begin{pmatrix}\mathbf{0}_{k} & \mathbf{L}(\mathbf{t})\\
\mathbf{R}(\mathbf{t}) & \mathbf{0}_{n}
\end{pmatrix}
\end{equation*}
indexed by subsets $\{1,\dots,k\}\cup\mathcal{I}$ with $\mathcal{I}\subseteq\{k+1,\dots,k+n\}$ and $\#\mathcal{I}=k$ (here $\mathbf{0}_{k}$ denotes the $(k\times k)$-dimensional null matrix). This connection with non-constant, invertible minor assignment in a given ring is the basis for further studies.

\subsection{Organisation of the paper}
\label{subsec: organization of the paper}

After Section \ref{sec: Preliminaries}, where we fix the notation that will be used in the rest of the paper, we derive the polynomial equations that are the basis of our study (Section \ref{sec: From Plucker to quartic relations}) and the constraints on algebraic extensions (Section \ref{sec: Algebraic constraints on algebraic extensions}) resulting from these equations. In Section \ref{sec: All three-terms non-vanishing}, we explore the constraints for non-vanishing minors in the three-term Grassmann-Pl\"{u}cker relations and describe allowed configurations (Subsections \ref{subsec: allowed configuration for radical terms}-\ref{subsec: Recovering local from global data}). From these results, in Section \ref{sec: proof of theorem} we prove Theorem \ref{thm: rigidity theorem monomial}, also investigating the occurrence of multiple distinct configurations to assess the minimal dimensions that guarantee com\-bi\-na\-to\-rial reduction. We provide an application to permutations on matroids in Section \ref{sec: Reduction of multi-valued maps and applications to permutation coding}.

\section{\label{sec: Preliminaries} Preliminaries}

\subsection{\label{subsec: Notation} Notation }

Let $\wp[n]$ be the power set of $[n]:=\{1,\dots,n\}$ and
$\wp_{k}[n]:=\left\{\mathcal{I}\subseteq[n]:\right.$ $\left.\#\mathcal{I}=k\right\}$ $\subseteq\wp[n]$. We adopt the shortening expressions 
 \begin{equation}
\mathcal{I}_{\beta}^{\alpha}:=\mathcal{I}\backslash\{\alpha\}\cup\{\beta\}\quad\mathcal{I}\in\wp[n],\,\alpha\in\mathcal{I},\,\beta\notin\mathcal{I}\label{eq: 1-exchange}
 \end{equation}
and similarly, $\mathcal{I}_{\beta}:=\mathcal{I}\cup\{\beta\}$, $\mathcal{I}^{\alpha}:=\mathcal{I}\backslash\{\alpha\}$, $\mathcal{I}^{\alpha_{1}\alpha_{2}}:=\mathcal{I}\backslash\{\alpha_{1},\alpha_{2}\}$, \textit{etc}. The notation $\mathcal{I}_{\beta}^{\alpha}$ implicitly assumes that $\alpha\in\mathcal{I}$ and $\beta\notin\mathcal{I}$, unless $\alpha=\beta$ where $\mathcal{I}_{\alpha}^{\alpha}:=\mathcal{I}$. We denote by $\mathcal{I}^{\mathtt{C}}:=[n]\setminus\mathcal{I}$ the complement of $\mathcal{I}$ in $[n]$. 
We adopt the notation $\mathbf{A}_{\mathcal{A};\mathcal{B}}$
to denote the submatrix of $\mathbf{A}$ with rows extracted from $\mathcal{A}\subseteq[k]$ and columns extracted from $\mathcal{B}\subseteq[n]$.

Rather than working with exponential sums as in (\ref{eq: Cauchy-Binet expansion of Wronskian}), we adopt a different parameterisation, moving to a polynomial ring with indeterminates $\mathbf{t}:=(t_{u}:\,u\in [d])$. 
For each index $w\in [d]$, the symbol $\mathbf{t}_{\widehat{w}}$ denotes the $(d-1)$-tuple obtained from $\mathbf{t}$ by omission of the $w$-th component. Given the ring $\mathbb{C}(\mathbf{t}):=\mathbb{C}[\mathbf{t},\mathbf{t}^{-1}]$ of the Laurent polynomials in $\mathbf{t}$, let $\mathbb{F}$ be the field of fractions of $\mathbb{C}(\mathbf{t})$. Given two polynomials $P,Q$, the symbol $P\mid Q$ means that $P$ is a factor of $Q$. 

For a non-vanishing polynomial $P\in\mathbb{C}(\mathbf{t})$, we use the standard inner product $\langle,\rangle$ between polynomials to formulate the support and the exponent set of $P$, i.e.
\begin{eqnarray}
P\in\mathbb{C}(\mathbf{t}) & \mapsto & \mathrm{Supp}(P):=\left\{ \mathbf{m}\in\mathbb{C}(\mathbf{t}):\,\mathbf{m}^{-1}\in\mathbb{C}(\mathbf{t}),\,\langle P,\mathbf{m}\rangle=\langle\mathbf{m},\mathbf{m}\rangle\right\} 
\label{eq: monomial terms in a polynomial},\\
P\in\mathbb{C}(\mathbf{t}) & \mapsto & \Psi(P):=\left\{ \mathbf{e}\in\mathbb{Z}^{d}:\,\exists c_{\mathbf{e}}\in\mathbb{C}\setminus\{0\},\,c_{\mathbf{e}}\cdot\mathbf{t}^{\mathbf{e}}\in\mathrm{Supp}(P)\right\}.
\label{eq: phases of monomials in a polynomial}
\end{eqnarray}
When $P:=c\cdot\mathbf{t}^{\mathbf{e}}$ is invertible in $\mathbb{C}(\mathbf{t})$ and no ambiguity arises, we use the symbol $\Psi(P)$ to refer to the vector of exponents $\mathbf{e}\in\mathbb{Z}^{d}$, in line with the function $\Psi$ in (\ref{eq: monomial terms of Cauchy-Binet expansion}). 
When $\Psi(P)\subseteq\{0,1\}^{d}$, it reduces to the characteristic function of a finite subset of $\wp[d]$.  

\begin{rem}
\label{rem: invertible change of indeterminates}
Any unimodular matrix $\mathbf{V}\in\mathbb{Z}^{d\times d}$ defines an invertible transformation $s_{u}:=\mathbf{t}^{\mathbf{V}_{u}}$, where $\mathbf{V}_{u}$ is the $u$-th column of $\mathbf{V}$ and $u\in[d]$, so $t_{u}=\mathbf{s}^{\mathbf{V}_{u}^{-1}}$. Each $P\in\mathbb{C}(\mathbf{t})$ corresponds to a polynomial $P_{s}:=P(\mathbf{s}^{\mathbf{V}^{-1}})\in\mathbb{C}(\mathbf{s})$, which induces a bijection between $\mathrm{Supp}(P)$ and $\mathrm{Supp}(P_{s})$: distinct monomials $a_{1}\mathbf{m}_{1}$ and $a_{2}\mathbf{m}_{2}$ in $\mathrm{Supp}(P)$ are mapped into distinct monomials $a_{1}\mathbf{s}^{\mathbf{V}^{-1}\cdot \Psi(\mathbf{m}_{1})}$ and $a_{2}\mathbf{s}^{\mathbf{V}^{-1}\cdot \Psi(\mathbf{m}_{2})}$ in $\mathrm{Supp}(P_{s})$ due to $\ker\mathbf{V}^{-1}=\{0\}$, and vice versa. This bijection extends to a ring isomorphism between $\mathbb{C}(\mathbf{t})$ and $\mathbb{C}(\mathbf{s})$. 
\end{rem}

We will make use of monomial orders: let us recall that a monomial order $\preceq$ is a total order on the class of monic monomials that is compatible with multiplication:  
 \begin{equation}
\forall\mathbf{x},\mathbf{y},\mathbf{z}\text{ monic monomials}:\,\mathbf{x}\preceq\mathbf{y}\Rightarrow\mathbf{x}\cdot\mathbf{z}\preceq\mathbf{y}\cdot\mathbf{z}.
\label{eq: monomial order compatibility with multiplication}
 \end{equation}
We express the left-hand side of (\ref{eq: monomial terms of Cauchy-Binet expansion}) as  
\begin{equation}
h(\mathcal{I}) :=  \Delta_{\mathbf{L}(\mathbf{t})}(\mathcal{I})\cdot\Delta_{\mathbf{R}(\mathbf{t})}(\mathcal{I}),\quad\mathcal{I}\in\wp_{k}[n].
\label{eq: notation exponential terms}
\end{equation}
\begin{defn} 
\label{def: observable sets} 
The multiset 
 \begin{equation}
\chi(\mathcal{I}\mid_{\alpha\beta}^{ij}):=\left\{ h(\mathcal{I})\cdot h(\mathcal{I}_{\alpha\beta}^{ij}),h(\mathcal{I}_{\alpha}^{i})\cdot h(\mathcal{I}_{\beta}^{j}),h(\mathcal{I}_{\beta}^{i})\cdot h(\mathcal{I}_{\alpha}^{j})\right\} \label{eq: three-terms set}
 \end{equation}
is called \emph{observable} when $\chi(\mathcal{I}\mid_{\alpha\beta}^{ij})\neq\{0\}$ and \emph{integrable} when 
 \begin{equation}
\#\Psi\left(\chi(\mathcal{I}\mid_{\alpha\beta}^{ij})\setminus\{0\}\right)=1,\label{eq: integrable three element set}
 \end{equation}
which implies that $\chi(\mathcal{I}\mid_{\alpha\beta}^{ij})$ is also observable. We say that $\mathcal{I}\in\mathfrak{G}(\mathbf{L}(\mathbf{1}))$ is an \emph{integrable basis} when (\ref{eq: integrable three element set}) holds for all observable sets $\chi(\mathcal{I}\mid_{\alpha\beta}^{ij})$ with basis $\mathcal{I}$. We call $h(\mathcal{I})\cdot h(\mathcal{I}_{\alpha\beta}^{ij})$ the \emph{central term} of $\chi(\mathcal{I}\mid_{\alpha\beta}^{ij})$, while we refer to a term in $\chi(\mathcal{I}\mid_{\alpha\beta}^{ij})$
that is not proportional over $\mathbb C$ to any of the others as \emph{unique term}. We will refer to (\ref{eq: three-terms set}) as a set to ease the notation, bearing in mind that repeated elements are allowed.
\end{defn} 
Subsets $\mathcal{I}$ are considered unordered: signed minors are introduced as 
 \begin{equation}
\Delta_{\mathbf{A}}(\alpha,\beta\mid\mathcal{H}):=\Delta_{\mathbf{A}}(\mathcal{H}_{\alpha\beta})\cdot(-1)^{1+S(\alpha,\mathcal{H})+S(\beta,\mathcal{H})}\cdot\mathrm{sign}\left(\alpha-\beta\right),\quad\mathcal{H}\in\wp_{k-2}[n]\label{sign 2-2, b}
 \end{equation}
where 
 \begin{equation}
S(\alpha,\mathcal{H}):=\#\{\beta\in\mathcal{H}:\,\beta<\alpha\},\quad\alpha\in\mathcal{H}^{\mathtt{C}}
\label{eq: four splitting indices}
 \end{equation}
takes into account permutations of $[n]$ and makes us express the three-terms Pl\"{u}cker relations \citep{GKZ1994} as  
 \begin{eqnarray}
\Delta_{\mathbf{A}}(\delta_{1},\delta_{2}\mid\mathcal{H})\cdot\Delta_{\mathbf{A}}(\delta_{3},\delta_{4}\mid\mathcal{H}) & = & \Delta_{\mathbf{A}}(\delta_{1},\delta_{3}\mid\mathcal{H})\cdot\Delta_{\mathbf{A}}(\delta_{2},\delta_{4}\mid\mathcal{H})\nonumber \\ 
& & -\Delta_{\mathbf{A}}(\delta_{1},\delta_{4}\mid\mathcal{H})\cdot\Delta_{\mathbf{A}}(\delta_{2},\delta_{3}\mid\mathcal{H})
\label{eq: three-terms Plucker relations, bis}
 \end{eqnarray}
for any $\mathcal{H}\in\wp_{k-2}[n]$ and pairwise distinct elements $\delta_{a}\in\mathcal{H}^{\mathtt{C}}$, $a\in[4]$.

\section{\label{sec: From Plucker to quartic relations} Coupled Pl\"{u}cker relations}
Let us consider the Grassmann-Pl\"{u}cker relations in the form (\ref{eq: three-terms Plucker relations, bis}) for $\mathbf{L}(\mathbf{t})$, that is, 
 \begin{equation}
\Delta_{\mathbf{L}(\mathbf{t})}(\mathcal{H}_{ij})\cdot\Delta_{\mathbf{L}(\mathbf{t})}(\mathcal{H}_{\alpha\beta})=c_{1}\Delta_{\mathbf{L}(\mathbf{t})}(\mathcal{H}_{j\alpha})\cdot\Delta_{\mathbf{L}(\mathbf{t})}(\mathcal{H}_{i\beta})+c_{2}\Delta_{\mathbf{L}(\mathbf{t})}(\mathcal{H}_{j\beta})\cdot\Delta_{\mathbf{L}(\mathbf{t})}(\mathcal{H}_{i\alpha})
\label{eq: Grassmann-Plucker including permutations}
 \end{equation}
where $\mathcal{H}\in\wp_{k-2}[n]$, $i,j,\alpha,\beta\in\mathcal{H}^{\mathtt{C}}$, and signs 
\begin{eqnarray}
c_{1} & := & \mathrm{sign}\left[(i-j)\cdot(\alpha-\beta)\cdot(i-\beta)\cdot(\alpha-j)\right],\nonumber \\
c_{2} & := & \mathrm{sign}\left[(i-j)\cdot(\alpha-\beta)\cdot(i-\alpha)\cdot(j-\beta)\right]\label{eq: signs for permutation}
\end{eqnarray}
take into account the permutation of indices $i,j,\alpha,\beta$ with respect to a given order. The same relations hold for $\mathbf{R}(\mathbf{t})$. Multiplying term by term the resulting equations (\ref{eq: Grassmann-Plucker including permutations}) for $\mathbf{L}(\mathbf{t})$ and $\mathbf{R}(\mathbf{t})$, we obtain
\begin{eqnarray} 
h(\mathcal{I})\cdot h(\mathcal{I}_{\alpha\beta}^{ij}) & = & h(\mathcal{I}_{\alpha}^{i})\cdot h(\mathcal{I}_{\beta}^{j})+h(\mathcal{I}_{\beta}^{i})\cdot h(\mathcal{I}_{\alpha}^{j})\nonumber \\
&  & +c_{1}c_{2}\cdot \Delta_{\mathbf{R}(\mathbf{t})}(\mathcal{I}_{\alpha}^{i})\cdot\Delta_{\mathbf{R}(\mathbf{t})}(\mathcal{I}_{\beta}^{j})\cdot\Delta_{\mathbf{L}(\mathbf{t})}(\mathcal{I}_{\beta}^{i})\cdot\Delta_{\mathbf{L}(\mathbf{t})}(\mathcal{I}_{\alpha}^{j})\nonumber \\
&  & +c_{1}c_{2}\cdot \Delta_{\mathbf{R}(\mathbf{t})}(\mathcal{I}_{\beta}^{i})\cdot\Delta_{\mathbf{R}(\mathbf{t})}(\mathcal{I}_{\alpha}^{j})\cdot\Delta_{\mathbf{L}(\mathbf{t})}(\mathcal{I}_{\alpha}^{i})\cdot\Delta_{\mathbf{L}(\mathbf{t})}(\mathcal{I}_{\beta}^{j})
\label{eq: quartic from Grassmann-Plucker}
\end{eqnarray}
since $c_{1},c_{2}$ depend only on the indices $i,j,\alpha,\beta$. This expression can be written as 
\begin{eqnarray}
h(\mathcal{I})\cdot h(\mathcal{I}_{\alpha\beta}^{ij}) & = & h(\mathcal{I}_{\alpha}^{i})\cdot h(\mathcal{I}_{\beta}^{j})+h(\mathcal{I}_{\beta}^{i})\cdot h(\mathcal{I}_{\alpha}^{j})\nonumber \\
& & +c_{1}c_{2}\cdot\frac{\Delta_{\mathbf{R}(\mathbf{t})}(\mathcal{I}_{\alpha}^{i})}{\Delta_{\mathbf{R}(\mathbf{t})}(\mathcal{I}_{\beta}^{i})}\cdot h(\mathcal{I}_{\beta}^{i})\cdot\frac{\Delta_{\mathbf{R}(\mathbf{t})}(\mathcal{I}_{\beta}^{j})}{\Delta_{\mathbf{R}(\mathbf{t})}(\mathcal{I}_{\alpha}^{j})}\cdot h(\mathcal{I}_{\alpha}^{j})\nonumber \\
&  & +c_{1}c_{2}\cdot\frac{\Delta_{\mathbf{R}(\mathbf{t})}(\mathcal{I}_{\beta}^{i})}{\Delta_{\mathbf{R}(\mathbf{t})}(\mathcal{I}_{\alpha}^{i})}h(\mathcal{I}_{\beta}^{j})\cdot\frac{\Delta_{\mathbf{R}(\mathbf{t})}(\mathcal{I}_{\alpha}^{j})}{\Delta_{\mathbf{R}(\mathbf{t})}(\mathcal{I}_{\beta}^{j})}\cdot h(\mathcal{I}_{\alpha}^{i})
\label{eq: quartic from Grassmann-Plucker, b}
\end{eqnarray}
since we are assuming that all minors $\Delta_{\mathbf{R}(\mathbf{t})}(\mathcal{I})$, $\mathcal{I}\in\wp_{k}[n]$, are non-vanishing for a generic choice of $\mathbf{t}$. 
\begin{defn} 
\label{def: observable Y-terms} 
We define the \emph{$Y$-terms}, or cross-ratios, as
 \begin{equation}
Y(\mathcal{I})_{\alpha\beta}^{ij}:=c_{1}c_{2}\cdot\frac{\Delta_{\mathbf{R}(\mathbf{t})}(\mathcal{I}_{\alpha}^{i})}{\Delta_{\mathbf{R}(\mathbf{t})}(\mathcal{I}_{\beta}^{i})}\cdot\frac{\Delta_{\mathbf{R}(\mathbf{t})}(\mathcal{I}_{\beta}^{j})}{\Delta_{\mathbf{R}(\mathbf{t})}(\mathcal{I}_{\alpha}^{j})}
\label{eq: cross-section}
 \end{equation}
where (\ref{eq: signs for permutation}) gives 
\begin{equation}
c_{1}c_{2}=-\mathrm{sign}\left[(i-\alpha)\cdot(i-\beta)\cdot(j-\alpha)\cdot(j-\beta)\right]. 
\label{eq: c1c2 sign for Y}
\end{equation}
Furthermore, we introduce 
 \begin{equation}
Y(\mathcal{I}):=\left\{ Y(\mathcal{I})_{\alpha\beta}^{ij}:\,\chi(\mathcal{I}\mid_{\alpha\beta}^{ij})\text{ is observable}\right\}.
\label{eq: set of Y-terms of observables}
 \end{equation}
We extend the notation introduced in Definition \ref{def: observable sets} stating that $Y(\mathcal{I})_{\alpha\beta}^{ij}$ is \emph{observable} if $\chi(\mathcal{I}\mid_{\alpha\beta}^{ij})$ is an observable set. 
\end{defn}
The dependence of $Y_{\alpha\beta}^{ij}:=Y(\mathcal{I})_{\alpha\beta}^{ij}$ on the basis $\mathcal{I}$ will be implicit when no ambiguity arises. Being $Y(\mathcal{I})_{\beta\alpha}^{ij}=(Y(\mathcal{I})_{\alpha\beta}^{ij})^{-1}$, the genericity assumption on $\mathrm{R}(\mathbf{t})$ is equivalent to the existence of $Y(\mathcal{I})_{\alpha\beta}^{ij}$ for all choices of bases and indices. 
For each generalised permutation matrix $\mathbf{D}(\mathbf{t})$ dependent on $\mathbf{t}$, the $Y$-terms (\ref{eq: cross-section}) are invariant with respect to the action  
 \begin{equation}
\mathbf{L}(\mathbf{t})\mapsto\mathbf{L}(\mathbf{t})\cdot\mathbf{D}(\mathbf{t})^{-1},\quad\mathbf{R}(\mathbf{t})\mapsto\mathbf{D}(\mathbf{t})\cdot\mathbf{R}(\mathbf{t}).
\label{eq: generalized permutation matrix gauge}
 \end{equation}
Using this invariance, we can fix a form for the matrices $\mathbf{R}(\mathbf{t})$ and $\mathbf{L}(\mathbf{t})$ while preserving $Y$-terms.
In particular, we choose the reduced column echelon form of $\mathbf{R}$ and set 
 \begin{equation}
\mathbf{D}(\mathbf{t}):=\mathrm{diag}\left(\begin{array}{c
cccccc}
R_{k+1,1}(\mathbf{t}) 
& \dots & R_{k+1,k}(\mathbf{t}) & 1 & \frac{R_{k+1,1}(\mathbf{t})}{R_{k+2,1}(\mathbf{t})} & \dots & \frac{R_{k+1,1}(\mathbf{t})}{R_{n,1}(\mathbf{t})}\end{array}\right).
\label{eq: special gauge, diagonal, n}
 \end{equation}
In addition, we consider the right multiplication of $\mathbf{R}(\mathbf{t})$ by 
 \begin{equation}
\mathbf{d}(\mathbf{t}):=\text{rcef}(\mathbf{R}(\mathbf{t}))\cdot \mathrm{diag}\left(\begin{array}{cccc}
R_{k+1,1}(\mathbf{t})^{-1} & 
\dots & R_{k+1,k}(\mathbf{t})^{-1} \end{array}\right)\label{eq: special gauge, diagonal, k}
 \end{equation}
where the $k$-dimensional matrix $\text{rcef}(\mathbf{R}(\mathbf{t}))$ transforms $\mathbf{R}(\mathbf{t})$ into its reduced column echelon form. This preserves the determinantal expansion up to a common factor $\det\mathbf{d}(\mathbf{t})$ that is irrelevant in terms of Pl\"{u}cker coordinates; this additional factor can be absorbed through a left multiplication of $\mathbf{L}(\mathbf{t})$, e.g. by $\mathbf{d}(\mathbf{t})^{-1}$. In this way, $\mathbf{R}(\mathbf{t})$ takes the form 
 \begin{equation}
\mathbf{R}(\mathbf{t})\mapsto\mathbf{D}(\mathbf{t})\cdot\mathbf{R}(\mathbf{t})\cdot\mathbf{d}(\mathbf{t})=\left(\begin{array}{cccc}
1 & 0 & \dots & 0\\
0 & 1 & \dots & 0\\
\vdots & \vdots & \ddots & \vdots\\
0 & 0 & \dots & 1\\
1 & 1 & \dots & 1\\
1 & \frac{R_{k+1,1}R_{k+2,2}}{R_{k+1,2}R_{k+2,1}} & \dots & \frac{R_{k+1,1}R_{k+2,k}}{R_{k+1,k}R_{k+2,1}}\\
\vdots & \vdots & \ddots & \vdots\\
1 & \frac{R_{k+1,1}R_{n,2}}{R_{k+1,2}R_{n,1}} & \dots & \frac{R_{k+1,1}R_{n,k}}{R_{k+1,k}R_{n,1}}
\end{array}\right)
\label{eq: action of special gauge on R}
 \end{equation}
so that, for each $j\in\{2,\dots,k\}$ and $\alpha\in\{k+2,\dots,n\}$, we get ($\mathcal{V}=[k]$) 
 \begin{equation}
R_{\alpha j}(\mathbf{t})\mapsto\frac{\Delta_{\mathbf{R}(\mathbf{t})}(\mathcal{V}_{k+1}^{1})}{\Delta_{\mathbf{R}(\mathbf{t})}(\mathcal{V}_{k+1}^{j})}\cdot\frac{\Delta_{\mathbf{R}(\mathbf{t})}(\mathcal{V}_{\alpha}^{j})}{\Delta_{\mathbf{R}(\mathbf{t})}(\mathcal{V}_{\alpha}^{1})}=c_{1}c_{2}\cdot Y_{(k+1)\alpha}^{1j}.
\label{eq: entries of R in special gauge}
 \end{equation}

\begin{rem}
\label{rem: dual Grassmannian representation} 
The previous representation allows us to discuss the duality defined by a pair of matrices $(\mathbf{L}(\mathbf{t})^{\perp},\mathbf{R}(\mathbf{t})^{\perp})$ representing the orthogonal complements of the subspaces associated with $(\mathbf{L}(\mathbf{t}),\mathbf{R}(\mathbf{t}))$. Given any $\mathcal{I}\in\mathfrak{G}(\mathbf{L}(\mathbf{1}))$, an appropriate choice of a permutation matrix $\mathbf{P}$ lets us map $\mathcal{I}$ into $[k]$ through $(\mathbf{L}(\mathbf{t}),\mathbf{R}(\mathbf{t}))\mapsto(\mathbf{L}(\mathbf{t})\cdot \mathbf{P},\mathbf{P}^{-1}\cdot\mathbf{R}(\mathbf{t}))$ while preserving the expansion (\ref{eq: Cauchy-Binet expansion}).
Then we consider the reduced row echelon form $\mathbf{\widehat{L(\mathbf{t})}}=:(\idd_{k}\mid\mathbf{l})$ of $\mathbf{L}(\mathbf{t})\cdot \mathbf{P}$ and the reduced column echelon form $\mathbf{\widehat{R(\mathbf{t})}}=:(\idd_{k}\mid\mathbf{r})^{\mathtt{T}}$ of $\mathbf{P}^{-1}\cdot\mathbf{R}(\mathbf{t})$, where $\idd_{k}$ denotes the $k$-dimensional identity matrix. In this way, we can focus on the case $\mathcal{I} = [k]$ by considering such matrices $\widehat{\mathbf{L}(\mathbf{t})}$ and $\widehat{\mathbf{R}(\mathbf{t})}$ obtained through dimension relabelling via (\ref{eq: generalized permutation matrix gauge}). A well-known result (see e.g. \citep{Oxley2006,Karp2017} and references therein for more details) asserts that the row span of $\widehat{\mathbf{L}(\mathbf{t})}^{\perp}$ (respectively, the column span of $\widehat{\mathbf{R}(\mathbf{t})}^{\perp}$) generates the same Pl\"{u}cker coordinates as $\mathbf{\widehat{L(\mathbf{t})}}$ (respectively, $\mathbf{\widehat{R(\mathbf{t})}}$) up to a sign; specifically, there exist two non-vanishing coefficients $C_{\mathbf{L}}$, $C_{\mathbf{R}}$ such that  
\begin{equation*}
\Delta_{\mathrm{alt}(\mathbf{\widehat{L(\mathbf{t})}})}(\mathcal{J}^{\mathtt{C}}) = C_{\mathbf{L}}\cdot\Delta_{\widehat{\mathbf{L}(\mathbf{t})}}(\mathcal{J}),
\quad 
\Delta_{\mathrm{alt}(\mathbf{\widehat{R(\mathbf{t})}})}(\mathcal{J}^{\mathtt{C}}) = C_{\mathbf{R}}\cdot\Delta_{\widehat{\mathbf{R}(\mathbf{t})}}(\mathcal{J})
\end{equation*}
for all $\mathcal{J}\in\wp_{k}[n]$, where the row span of $\widehat{\mathbf{L}(\mathbf{t})}^{\perp}$ and the column span of $\widehat{\mathbf{R}(\mathbf{t})}^{\perp}$ are expressed by $(\mathbf{l}^{\mathtt{T}}\mid -\idd_{n-k})$ and $(\mathbf{r}^{\mathtt{T}}\mid -\idd_{n-k})^{\mathtt{T}}$, respectively, and 
\begin{equation*} 
\mathrm{alt}(\mathbf{\widehat{L(\mathbf{t})}}) := 
(\mathbf{l}^{\mathtt{T}}\mid -\idd_{n-k})\cdot \mathrm{diag}((-1)^{i+1})_{i\in [n]}, \quad 
\mathrm{alt}(\mathbf{\widehat{R(\mathbf{t})}}) := \mathrm{diag}((-1)^{i+1})_{i\in [n]}\cdot (\mathbf{r}^{\mathtt{T}}\mid -\idd_{n-k})^{\mathtt{T}}.
\end{equation*}
In particular, if $n<2\cdot k$ we can work with the matrices $\mathrm{alt}(\mathbf{\widehat{L(\mathbf{t})}})$ and $\mathrm{alt}(\mathbf{\widehat{R(\mathbf{t})}})$ obtaining the same minor products as in (\ref{eq: Cauchy-Binet expansion}) up to a common multiplicative factor. Maximal minors extracted from these matrices have order $n-k$ instead of $k$. Therefore, the condition $\max\{n-k,k\}\geq 5$ in the statement of Theorem \ref{thm: rigidity theorem monomial} becomes $n-k\geq 5$ hereinafter, without loss of generality. 
\end{rem}

It is also easy to check that 
\begin{eqnarray}
Y_{\alpha\beta}^{ij}Y_{\beta\gamma}^{ij} & = & -Y_{\alpha\gamma}^{ij},\label{eq: associativity, d}\\
Y_{\alpha\beta}^{im}\cdot Y_{\alpha\beta}^{mj} & = & -Y_{\alpha\beta}^{ij}.
\label{eq: associativity, u}
\end{eqnarray}
Iterating (\ref{eq: associativity, d}), we obtain an identity that will be used several times in the rest of this work, which we refer to as \emph{quadrilateral decomposition}:
\begin{equation}
    Y_{\alpha\beta}^{ij} = -Y_{\alpha\delta}^{ij}\cdot Y_{\delta\beta}^{ij} 
    = -Y_{\alpha\delta}^{im}\cdot Y_{\alpha\delta}^{mj}\cdot Y_{\delta\beta}^{im}\cdot Y_{\delta\beta}^{mj},\quad i,j,m\in\mathcal{I},\,\alpha,\beta,\delta\in\mathcal{I}^{\mathtt{C}}.
    \label{eq: quadrilateral decomposition}
\end{equation}
Then, (\ref{eq: quartic from Grassmann-Plucker}) is equivalent to 
 \begin{equation}
h(\mathcal{I})\cdot h(\mathcal{I}_{\alpha\beta}^{ij})=h(\mathcal{I}_{\alpha}^{i})\cdot h(\mathcal{I}_{\beta}^{j})+Y_{\alpha\beta}^{ij}\cdot h(\mathcal{I}_{\beta}^{i})\cdot h(\mathcal{I}_{\alpha}^{j})+\frac{1}{Y_{\alpha\beta}^{ij}}\cdot h(\mathcal{I}_{\alpha}^{i})\cdot h(\mathcal{I}_{\beta}^{j})+h(\mathcal{I}_{\beta}^{i})\cdot h(\mathcal{I}_{\alpha}^{j})\label{eq: quartic from Grassmann-Plucker, 2}
 \end{equation}
where the dependence on $\mathbf{t}$ is implicit. 
\begin{rem}
\label{rem: no Y=-1} Note that $Y_{\alpha\beta}^{ij}\neq-1$
for all $\mathcal{I}\in\wp_{k}[n]$, $i,j\in\mathcal{I}$, and $\alpha,\beta\in\mathcal{I}^{\mathtt{C}}$, unless $i=j$ or $\alpha=\beta$. Indeed, $Y_{\alpha\beta}^{ij}=-1$ is equivalent to 
\begin{equation*} 
\Delta_{\mathbf{R}(\mathbf{t})}(\mathcal{I}_{\alpha}^{i})\cdot\Delta_{\mathbf{R}(\mathbf{t})}(\mathcal{I}_{\beta}^{j})=-c_{1}c_{2}\cdot\Delta_{\mathbf{R}(\mathbf{t})}(\mathcal{I}_{\beta}^{i})\cdot\Delta_{\mathbf{R}(\mathbf{t})}(\mathcal{I}_{\alpha}^{j})
\end{equation*}
which implies $\Delta_{\mathbf{R}(\mathbf{t})}(\mathcal{I})\cdot\Delta_{\mathbf{R}(\mathbf{t})}(\mathcal{I}_{\alpha\beta}^{ij})=0$ by the Pl\"{u}cker relations (\ref{eq: Grassmann-Plucker including permutations}), and this contradicts the assumption $\Delta_{\mathbf{R}(\mathbf{t})}(\mathcal{I}_{\alpha\beta}^{ij})\neq0$. 
\end{rem}
We also note that (\ref{eq: associativity, d}) and (\ref{eq: associativity, u}) are special instances of (\ref{eq: quadrilateral decomposition}): these correspond to the degenerate cases where the lower, respectively, upper indices in some $Y$-terms coincide, leading to the only allowed case of $Y$-terms equal to $-1$, according to the definition (\ref{eq: cross-section}) and Remark \ref{rem: no Y=-1}.

Whether $\{0\}\subsetneq\chi(\mathcal{I}\mid_{\alpha\beta}^{ij})$, (\ref{eq: quartic from Grassmann-Plucker, 2}) is equivalent to $P(Y_{\alpha\beta}^{ij})=0$ where $P$ is a monic polynomial over $\mathbb{C}(\mathbf{t})$, since $h(\mathcal{I}_{\beta}^{i})\cdot h(\mathcal{I}_{\alpha}^{j})\neq0$ is invertible in $\mathbb{C}(\mathbf{t})$ by assumption; therefore, $Y_{\alpha\beta}^{ij}$ is integral over $\mathbb{C}(\mathbf{t})$, and we introduce the notation 
\begin{eqnarray}
A_{\alpha\beta}^{ij} & := & h(\mathcal{I})\cdot h(\mathcal{I}_{\alpha\beta}^{ij})-h(\mathcal{I}_{\alpha}^{i})\cdot h(\mathcal{I}_{\beta}^{j})-h(\mathcal{I}_{\beta}^{i})\cdot h(\mathcal{I}_{\alpha}^{j}),
\label{eq: expression cross-section quadratic case, a}\\
B_{\alpha\beta}^{ij} & := & \left(A_{\alpha\beta}^{ij}\right)^{2} -4\cdot h(\mathcal{I}_{\alpha}^{i})\cdot h(\mathcal{I}_{\beta}^{j})\cdot h(\mathcal{I}_{\beta}^{i})\cdot h(\mathcal{I}_{\alpha}^{j})
\label{eq: expression cross-section quadratic case, b}
\end{eqnarray}
to express 
 \begin{equation}
Y_{\alpha\beta}^{ij}=\frac{\varepsilon_{\alpha\beta}^{ij}\cdot\sqrt{B_{\alpha\beta}^{ij}}+A_{\alpha\beta}^{ij}}{2\cdot h(\mathcal{I}_{\beta}^{i})\cdot h(\mathcal{I}_{\alpha}^{j})}\label{eq: radical expression}
 \end{equation}
as a root of 
\begin{equation}
F_{\alpha\beta}^{ij}(X):=\left(h(\mathcal{I}_{\beta}^{i})\cdot h(\mathcal{I}_{\alpha}^{j})\right)\cdot X^{2}-A_{\alpha\beta}^{ij}\cdot X+h(\mathcal{I}_{\alpha}^{i})\cdot h(\mathcal{I}_{\beta}^{j})
\label{eq: polynomial h to Y}
\end{equation}
where $\varepsilon_{\alpha\beta}^{ij}\in\{+1,-1\}$ and $\sqrt{B_{\alpha\beta}^{ij}}$ is a fixed square root of $B_{\alpha\beta}^{ij}$. Even for the \textit{$A$-terms} (\ref{eq: expression cross-section quadratic case, a})
and the \textit{$B$-terms} (\ref{eq: expression cross-section quadratic case, b}) the dependence on $\mathcal{I}$ will be implicit where no ambiguity arises. 
For future convenience, we also introduce the following function based on (\ref{eq: expression cross-section quadratic case, b}) 
 \begin{equation}
B(x,y,z):=(x-y-z)^{2}-4yz.
\label{eq: B-function}
 \end{equation}
\begin{defn}
\label{def: radical triple}
We call $Y_{\alpha\beta}^{ij}$ a \emph{radical term} if $0\notin\chi(\mathcal{I}\mid_{\alpha\beta}^{ij})$ and $Y_{\alpha\beta}^{ij}\notin\mathbb{F}$. By extension, we say that the associated set $\chi(\mathcal{I}\mid_{\alpha\beta}^{ij})$ is a \emph{radical set}. 
We refer to $(Y_{\alpha\beta}^{ij},Y_{\beta\gamma}^{ij},Y_{\alpha\gamma}^{ij})$ as a \emph{radical triple} if all its components are radical $Y$-terms.
\end{defn}

The exchange $\alpha\leftrightarrows\beta$ affects $Y_{\alpha\beta}^{ij}\mapsto Y_{\beta\alpha}^{ij}=(Y_{\alpha\beta}^{ij})^{-1}$ through the change of sign $\pm\mapsto\mp$ in (\ref{eq: radical expression}), 
while $c_{1}\cdot c_{2}$, $A_{\alpha\beta}^{ij}$, and $B_{\alpha\beta}^{ij}$ are preserved. 
From the Pl\"{u}cker relations (\ref{eq: Grassmann-Plucker including permutations}), we also find 
\begin{eqnarray}
Y(\mathcal{I}_{\alpha}^{i})_{i\beta}^{\alpha j} & = & -c_{2}\frac{\Delta_{\mathbf{R}(\mathbf{t})}(\mathcal{I})\cdot\Delta_{\mathbf{R}(\mathbf{t})}(\mathcal{I}_{\alpha\beta}^{ij})}{\Delta_{\mathbf{R}(\mathbf{t})}(\mathcal{I}_{\beta}^{i})\cdot\Delta_{\mathbf{R}(\mathbf{t})}(\mathcal{I}_{\alpha}^{j})}=-Y(\mathcal{I})_{\alpha\beta}^{ij}-1,\label{eq: Grassmann-Plucker translation exchange, vertical}\\
Y(\mathcal{I}_{\beta}^{i})_{\alpha i}^{\beta j} & = & -c_{1}\cdot\frac{\Delta_{\mathbf{R}(\mathbf{t})}(\mathcal{I}_{\alpha}^{i})}{\Delta_{\mathbf{R}(\mathbf{t})}(\mathcal{I})}\cdot\frac{\Delta_{\mathbf{R}(\mathbf{t})}(\mathcal{I}_{\beta}^{j})}{\Delta_{\mathbf{R}(\mathbf{t})}(\mathcal{I}_{\alpha\beta}^{ij})}=-\frac{1}{1+\left(Y(\mathcal{I})_{\alpha\beta}^{ij}\right)^{-1}}.
\label{eq: Grassmann-Plucker translation exchange, diagonal}
\end{eqnarray}
Together with the exchange $\alpha\leftrightarrows\beta$, (\ref{eq: Grassmann-Plucker translation exchange, vertical}) and (\ref{eq: Grassmann-Plucker translation exchange, diagonal}) define a set of \emph{local transformations} (with respect to the basis $\mathcal{I}$ and the indices $i,j$, $\alpha,\beta$) that will be useful in the following. 
\begin{lem}
\label{lem: observable sets with 0 are algebraic}
Any radical term $Y_{\alpha\beta}^{ij}$ forces $0\notin\chi(\mathcal{I}\mid_{\alpha\gamma}^{ij})$ for all $\gamma\in\mathcal{I}^{\mathtt{C}}$, $\gamma\neq \alpha$. 
\end{lem}
\begin{proof}
All observable $Y$-terms $Y(\mathcal{I})_{\gamma\delta}^{ml}$ with $0\in\chi(\mathcal{I}\mid_{\gamma\delta}^{ml})$ lie in $\mathbb{F}$, as can easily be seen from (\ref{eq: quartic from Grassmann-Plucker, 2}). 
It follows that any radical $Y$-term $Y_{\alpha\beta}^{ij}$ satisfies 
 \begin{equation}
h(\mathcal{I})\cdot h(\mathcal{I}_{\alpha\beta}^{ij})\cdot h(\mathcal{I}_{\alpha}^{i})\cdot h(\mathcal{I}_{\beta}^{i})\cdot h(\mathcal{I}_{\alpha}^{j})\cdot h(\mathcal{I}_{\beta}^{j})\neq0
\label{eq: all 3 terms non-vanishing}
 \end{equation}
so $\mathcal{I}_{\gamma}^{m}\in\mathfrak{G}(\mathbf{L}(\mathbf{1}))$ for all $\gamma\in\{\alpha,\beta\}$ and $m\in\{i,j\}$. The term $Y_{\alpha\beta}^{ij}$ remains radical under each change of basis $\mathcal{I}\mapsto\mathcal{I}_{\gamma}^{m}$, so if $0\in\chi(\mathcal{I}\mid_{\alpha\omega}^{ij})$ we can assume $h(\mathcal{I}_{\omega}^{i})=0$, since the other cases are obtained from changes of bases generated by local transformations. Due to the lack of null columns, we can find $p\in\mathcal{I}$ such that $h(\mathcal{I}_{\omega}^{p})\neq 0$, and we choose $p=j$ if $h(\mathcal{I}_{\omega}^{j})\neq 0$. So $\chi(\mathcal{I}\mid_{\gamma\omega}^{pm})$ is observable and contains $0$ for all $\gamma\in\{\alpha,\beta\}$ and $m\in\{i,j\}$, then $Y_{\gamma\omega}^{pm}\in\mathbb{F}$ and, from (\ref{eq: quadrilateral decomposition}), $Y_{\alpha\beta}^{ij}\in\mathbb{F}$, i.e. a contradiction. 
\end{proof}
\begin{rem}
\label{rem: transformations and permutation group}
There are three basic, non-equivalent actions that can be carried out on the indices of $\chi(\mathcal{I}\mid_{\alpha\beta}^{ij})$: the identity, the exchange of upper indices $i\leftrightarrows j$ (equivalently, lower indices $\alpha\leftrightarrows \beta$), and the change of basis $\mathcal{I}\mapsto\mathcal{I}_{\alpha}^{i}$ (equivalently, $\mathcal{I}\mapsto\mathcal{I}_{\beta}^{j}$). Other operations can be obtained from the composition of the previous ones: for instance,  $\mathcal{I}\mapsto\mathcal{I}_{\alpha}^{j}$ is obtained from the composition of $i\leftrightarrows j$, the above-mentioned exchange of basis, and a second exchange $i\leftrightarrows \alpha$. 

Looking at the action of these transformations on the $Y$-terms, they are obtained from the composition of the inversion $f_{h}(Y):=Y\mapsto Y^{-1}$ associated with $i\leftrightarrows j$, and the transformation $f_{v}(Y):=-1-Y$ associated with $i\leftrightarrows \alpha$ as in (\ref{eq: Grassmann-Plucker translation exchange, vertical}). These two functions are involutions and can be matched with transpositions in the permutation group $\mathcal{S}_{3}$: the identity is assigned to the trivial permutation, $f_{h}$ is assigned to the transposition $(1\,2)$, and $f_{v}$ is assigned to $(2\,3)$. In this way, a combination of these two functions corresponds to the product of the associated transpositions, and the different combinations are recovered by decomposing the elements in $\mathcal{S}_{3}$ in terms of $(1\,2)$ and $(2\,3)$. In particular, (\ref{eq: Grassmann-Plucker translation exchange, diagonal}) comes from the composition $f_{h}\circ f_{v} \circ f_{h}$, which is mapped to the product of transpositions $(1\,2)(2\,3)(1\,2)=(1\,3)$. 
\end{rem}

\section{\label{sec: Algebraic constraints on algebraic extensions}Constraints on algebraic extensions}

For a given unit $G_{\alpha\beta}^{ij}\in\mathbb{C}(\mathbf{t})$, we introduce 
 \begin{equation}
\hat{\chi}(\mathcal{I}\mid_{\alpha\beta}^{ij}):=\left\{ (G_{\alpha\beta}^{ij})^{-1}\cdot X,\,X\in\chi(\mathcal{I}\mid_{\alpha\beta}^{ij})\right\}.
\label{eq: normalised three-term set}
 \end{equation}
We anticipate that the scaling (\ref{eq: normalised three-term set}) will be used later in this work with different definitions of the monomial $G_{\alpha\beta}^{ij}$. In this subsection, we consider the ``ground'' monic monomial in $\mathbb{C}(\mathbf{t})$ 
\begin{equation}
G_{\alpha\beta}^{ij} := 
\prod_{c=1}^{d}t_{c}^{\min\left\{ \Psi\left(h(\mathcal{I})\cdot h(\mathcal{I}_{\alpha\beta}^{ij})\right)_{c},\Psi\left(h(\mathcal{I}_{\alpha}^{i})\cdot h(\mathcal{I}_{\beta}^{j})\right)_{c},\Psi\left(h(\mathcal{I}_{\beta}^{i})\cdot h(\mathcal{I}_{\alpha}^{j})\right)_{c}\right\} }.
\label{eq: greatest common divisor}
\end{equation}
Note that $B_{\alpha\beta}^{ij}$ and $G_{\alpha\beta}^{ij}$ are symmetric under permutations of elements of $\chi(\mathcal{I}\mid_{\alpha\beta}^{ij})$, so these quantities are preserved by the simultaneous application of a bijection $\pi:\,\{i,j,\alpha,\beta\}\longrightarrow\{i,j,\alpha,\beta\}$ and the change of basis induced by $\pi$. 

We start from the following lemma, which contains a very basic result of practical relevance for the discussion, as it will be repeatedly recalled in the rest of the paper.  
\begin{lem}
\label{lem: case B^(ij)_(ab) square} If (\ref{eq: all 3 terms non-vanishing}) is verified and $B_{\alpha\beta}^{ij}$ is a perfect square in $\mathbb{C}(\mathbf{t})$, then $Y_{\alpha\beta}^{ij}\in\mathbb{C}$. 
\end{lem}
\begin{proof}
Let $B_{\alpha\beta}^{ij}=P^{2}$ for some $P\in\mathbb{C}(\mathbf{t})$.
From the definitions (\ref{eq: expression cross-section quadratic case, a})-(\ref{eq: expression cross-section quadratic case, b}), we get 
 \begin{equation}
\left(A_{\alpha\beta}^{ij}-P\right)\cdot\left(A_{\alpha\beta}^{ij}+P\right)=4\cdot h(\mathcal{I}_{\alpha}^{i})\cdot h(\mathcal{I}_{\beta}^{i})\cdot h(\mathcal{I}_{\alpha}^{j})\cdot h(\mathcal{I}_{\beta}^{j}).
\label{eq: factorization if B perfect square}
 \end{equation}
Therefore, both $A_{\alpha\beta}^{ij}-P$ and $A_{\alpha\beta}^{ij}+P$ are invertible in $\mathbb{C}(\mathbf{t})$, and  
their sum $2A_{\alpha\beta}^{ij}$ has sparsity at most $2$ (as well as their difference $2P$). From (\ref{eq: expression cross-section quadratic case, a}), this means that at least two of the elements in $\chi(\mathcal{I}\mid_{\alpha\beta}^{ij})$ are proportional over $\mathbb{C}$. Changing the basis $\mathcal{I}\mapsto\mathcal{J}$ through a local transformation $\mathcal{J}\in\{\mathcal{I},\mathcal{I}_{\alpha}^{i},\mathcal{I}_{\beta}^{i}\}$, we can get $h(\mathcal{J})\cdot h(\mathcal{J}_{\alpha\beta}^{ij})=c\cdot h(\mathcal{J}_{\alpha}^{i})\cdot h(\mathcal{J}_{\beta}^{j})$,
$c\in\mathbb{C}\setminus\{0\}$, while preserving the factors of $B_{\alpha\beta}^{ij}$ in $\mathbb{C}(\mathbf{t})$. We label the unit $\tau:=h(\mathcal{J}_{\alpha}^{i})\cdot h(\mathcal{J}_{\beta}^{j})\cdot h(\mathcal{J}_{\beta}^{i})^{-1}\cdot h(\mathcal{J}_{\alpha}^{j})^{-1}$ to express 
 \begin{equation}
\frac{B_{\alpha\beta}^{ij}}{h(\mathcal{J}_{\beta}^{i})^{2}\cdot h(\mathcal{J}_{\alpha}^{j})^{2}}=(1-c)^{2}\cdot\tau^{2}-2(c+1)\cdot\tau+1.
\label{eq: B with two proportional terms}
 \end{equation}
Being $c\neq0$, $B_{\alpha\beta}^{ij}$ is a square in $\mathbb{C}(\mathbf{t})$ only if $\tau\in\mathbb{C}$: by definition, this means that $\chi(\mathcal{I}\mid_{\alpha\beta}^{ij})$ is integrable and hence $Y_{\alpha\beta}^{ij}\in\mathbb{C}$.
\end{proof}

\begin{lem}
\label{lem: common squarefree} All observable $Y$-terms with basis $\mathcal{I}$ lie in the same quadratic extension of $\mathbb{F}$. 
\end{lem}
\begin{proof}
We can assume the existence of a radical term $Y_{\alpha\beta}^{ij}$, otherwise all the observable $Y$-terms lie in $\mathbb{F}$ and the thesis follows. 
By Lemma \ref{lem: observable sets with 0 are algebraic}, we can focus on observable sets that satisfy (\ref{eq: all 3 terms non-vanishing}) for the rest of the proof: being $\mathbb{C}(\mathbf{t})$ a unique factorisation domain, we can express 
 \begin{equation}
B_{\alpha\beta}^{ij}:=\left(Q_{\alpha\beta}^{ij}\right)^{2}\cdot D\label{eq: factorization square and squarefree}
 \end{equation}
where $Q_{\alpha\beta}^{ij},D\in\mathbb{C}(\mathbf{t})$ and $D$ is a squarefree polynomial, that is, there exists no non-invertible $P\in\mathbb{C}(\mathbf{t})$ such that $P^{2}\mid D$. Then, we prove   
 \begin{equation}
Y_{\gamma\delta}^{ml}\in\mathbb{F}(Y_{\alpha\beta}^{ij})=\mathbb{F}(\sqrt{D})
\label{eq: same algebraic extension}
 \end{equation}
for observable terms $Y_{\gamma\delta}^{ml}$, where $\mathbb{F}(Y_{\alpha\beta}^{ij})$ denotes the algebraic extension of $\mathbb{F}$ by $Y_{\alpha\beta}^{ij}$.

We start from the instance $\{m,l\}=\{i,j\}$ and $\delta\in\{\alpha,\beta\}$, noting that we only have to consider radical triples $(Y_{\alpha\beta}^{ij},Y_{\beta\gamma}^{ij},Y_{\alpha\gamma}^{ij})$: being (\ref{eq: all 3 terms non-vanishing}) verified by $\chi(\mathcal{I}|_{\alpha\beta}^{ij})$ and $\chi(\mathcal{I}|_{\gamma\delta}^{ij})$, all the components of this triple are observable, and from (\ref{eq: associativity, d}) we get the thesis (\ref{eq: same algebraic extension}) when $Y_{\gamma\alpha}^{ij}\in\mathbb{F}$ or $Y_{\gamma\beta}^{ij}\in\mathbb{F}$. Therefore, taking into account (\ref{eq: radical expression}) for 
the $Y$-terms in the expression $Y_{\alpha\beta}^{ij}=-Y_{\alpha\gamma}^{ij}\cdot Y_{\gamma\beta}^{ij}$, we get  
 \begin{equation}
Y_{\alpha\gamma}^{ij}\in\mathbb{F}(\sqrt{B_{\alpha\beta}^{ij}})\Leftrightarrow Y_{\beta\gamma}^{ij}\in\mathbb{F}(\sqrt{B_{\alpha\beta}^{ij}}).
\label{eq: conditionals same extension}
 \end{equation}
Using again the assumption $Y_{\alpha\beta}^{ij}$, $Y_{\alpha\gamma}^{ij}$, $Y_{\gamma\beta}^{ij}\notin\mathbb{F}$, the product $Y_{\alpha\gamma}^{ij}\cdot Y_{\gamma\beta}^{ij}$ lies in a quadratic extension of $\mathbb{F}$ only if (\ref{eq: same algebraic extension}) holds. 
This argument can be adapted to triples $(Y_{\alpha\beta}^{ij},Y_{\alpha\beta}^{im},Y_{\alpha\beta}^{jm})$ by transposition of the upper and lower indices. Consequently, for all $\gamma_{1},\gamma_{2}\in\mathcal{I}^{\mathtt{C}}$ and $m_{1},m_{2}\in\mathcal{I}$, such that $Y_{\gamma_{1}\gamma_{2}}^{ij}$ and $Y_{\alpha\beta}^{m_{1}m_{2}}$ are observable, $Y_{\gamma_{w}\alpha}^{ij}$ and $Y_{\alpha\beta}^{m_{u}i}$ are also observable for all $u,w\in[2]$, and we have 
\begin{equation}
    Y_{\gamma_{1}\gamma_{2}}^{ij}=-Y_{\gamma_{1}\alpha}^{ij}\cdot Y_{\alpha\gamma_{2}}^{ij}\in \mathbb{F}(\sqrt{D}),\quad  Y_{\alpha\beta}^{m_{1}m_{2}}=-Y_{\alpha\beta}^{m_{1}i}\cdot Y_{\alpha\beta}^{im_{2}}\in \mathbb{F}(\sqrt{D}).
\label{eq: same quadratic extension, only upper or lower}
\end{equation}

Then, we move to radical terms $Y_{\gamma_{1}\gamma_{2}}^{m_{1}m_{2}}$ with $\{\alpha,\beta\}\neq\{\gamma_{1},\gamma_{2}\}$ and $\{i,j\}\neq\{m_{1},m_{2}\}$: for all $a,b,u,w\in[2]$, $i_{a}\in\{i,j\}$, and $\alpha_{b}\in\{\alpha,\beta\}$, the set $\chi(\mathcal{I}|_{\gamma_{w}\alpha_{b}}^{m_{u}i_{a}})$ is observable since $h(\mathcal{I}_{\alpha_{b}}^{i_{a}})\cdot h(\mathcal{I}_{\gamma_{w}}^{m_{u}})\neq0$. For any $u,w\in[2]$, we use such sets and (\ref{eq: quadrilateral decomposition}) to express  
\begin{equation} 
Y_{\alpha\beta}^{ij}=-Y_{\alpha\gamma_{w}}^{im_{u}}Y_{\alpha\gamma_{w}}^{m_{u}j}Y_{\gamma_{w}\beta}^{im_{u}}Y_{\gamma_{w}\beta}^{m_{u}j}\notin\mathbb{F}
\end{equation} 
which implies that there exists at least one pair in $\{i,j\}\times\{\alpha,\beta\}$, say $(i,\alpha)$ with proper labelling, such that $Y_{\alpha\gamma_{w}}^{im_{u}}$ is radical. This condition forces $h(\mathcal{I}_{\gamma_{w}}^{j})\cdot h(\mathcal{I}_{\beta}^{m_{u}})\neq0$ too: indeed, at $h(\mathcal{I}_{\gamma_{w}}^{j})=0$ we would have $Y_{\alpha\gamma_{w}}^{jm_{u}},Y_{\alpha\gamma_{w}}^{ij}\in\mathbb{F}$, since they derive from observable sets containing $0$, so their product would return $Y_{\alpha\gamma_{w}}^{im_{u}}\in\mathbb{F}$ by (\ref{eq: associativity, u}). A similar argument gives $h(\mathcal{I}_{\beta}^{m_{u}})\neq0$. In this way, we get 
\begin{equation}
h(\mathcal{I}_{\gamma_{w}}^{i_{a}})\cdot h(\mathcal{I}_{\alpha_{b}}^{m_{u}})\neq0,
\quad a,b,u,w\in[2],\,i_{a}\in\{i,j\},\,\alpha_{b}\in\{\alpha,\beta\}.
\label{eq: non-central non-vanishing}
\end{equation}

When $\chi(\mathcal{I}\mid_{\alpha\beta}^{m_{1}m_{2}})$ is radical, we apply twice (\ref{eq: same quadratic extension, only upper or lower}) to obtain 
\begin{equation} 
Y_{\alpha\beta}^{ij}\in\mathbb{F}(\sqrt{D})\setminus \mathbb{F} \Rightarrow Y_{\alpha\beta}^{m_{1}m_{2}}\in\mathbb{F}(\sqrt{D})\setminus\mathbb{F} \Rightarrow Y_{\gamma_{1}\gamma_{2}}^{m_{1}m_{2}}\in\mathbb{F}(\sqrt{D})\setminus\mathbb{F}
\end{equation} 
and a similar implication holds for radical $\chi(\mathcal{I}\mid_{\gamma_{1}\gamma_{2}}^{ij})$. When both $\chi(\mathcal{I}\mid_{\alpha\beta}^{m_{1}m_{2}})$ and $\chi(\mathcal{I}\mid_{\gamma_{1}\gamma_{2}}^{ij})$ contain $0$, (\ref{eq: non-central non-vanishing}) entails $h(\mathcal{I}_{\alpha\beta}^{m_{1}m_{2}})=h(\mathcal{I}_{\gamma_{1}\gamma_{2}}^{ij})=0$; being $Y_{\alpha\beta}^{ij},Y_{\gamma_{1}\gamma_{2}}^{m_{1}m_{2}}\notin\mathbb{F}$, this means that we can find $\overline{i}\in\{i,j\}$ and $\overline{m}\in\{m_{1},m_{2}\}$ such that $Y_{\alpha\beta}^{\overline{i}m_{u}}$ and $Y_{\gamma_{1}\gamma_{2}}^{\overline{m}x}$ are radical for all $u\in[2]$ and $x\in\{i,j\}$. In particular, both $Y_{\alpha\beta}^{\overline{i}\overline{m}}$ and $Y_{\gamma_{1}\gamma_{2}}^{\overline{i}\overline{m}}$ are radical: so, we apply (\ref{eq: same quadratic extension, only upper or lower}) three times to get 
\begin{eqnarray}
    Y_{\alpha\beta}^{ij}\in\mathbb{F}(\sqrt{D})\setminus \mathbb{F} & \Rightarrow & 
    Y_{\alpha\beta}^{\overline{i}\overline{m}}\in\mathbb{F}(\sqrt{D})\setminus\mathbb{F} \nonumber \\ 
    & \Rightarrow &  Y_{\gamma_{1}\gamma_{2}}^{\overline{i}\overline{m}}\in\mathbb{F}(\sqrt{D})\setminus\mathbb{F} \Rightarrow Y_{\gamma_{1}\gamma_{2}}^{m_{1}m_{2}}\in\mathbb{F}(\sqrt{D})\setminus\mathbb{F}.
\end{eqnarray}
\end{proof}
The previous lemma states that all $B$-terms based on $\mathcal{I}\in\mathfrak{G}(\mathbf{L}(\mathbf{1}))$ are Laurent polynomials in $\mathbf{t}$ that are not perfect squares in $\mathbb{C}(\mathbf{t})$ have the same squarefree part $D$. Now, we extend this result to all bases. 
\begin{prop}
\label{prop: same radical for all starting sets} The quantity $D$
in (\ref{eq: factorization square and squarefree}) is the same for each choice of basis $\mathcal{J}\in\mathfrak{G}(\mathbf{L}(\mathbf{t}))$. 
\end{prop}
\begin{proof}
Here, we make explicit the dependence of $Y(\mathcal{I})_{\alpha\beta}^{ij}$
on the basis $\mathcal{I}$. For any $\mathcal{J}\in\mathfrak{G}(\mathbf{L}(\mathbf{t}))$, we denote by $D(\mathcal{J})$ the squarefree part of any $B$-term associated with a radical $Y$-term.
In particular, $D(\mathcal{J})\in\mathbb{C}$ indicates that there is no radical $Y$-term with basis $\mathcal{J}$. Let us label $\mathcal{I}\setminus\mathcal{J}=:\{m_{1},\dots,m_{r}\}$, $r\leq k$. As remarked in \cite[Lem.6]{Angelelli2019}, the exchange property (\ref{eq: exchange relation}) for the matroid $\mathfrak{G}(\mathbf{L}(\mathbf{1}))$ implies that there exists a labelling $\mathcal{J}\setminus\mathcal{I}=:\{\delta_{1},\dots,\delta_{r}\}$ such that 
 \begin{equation}
\mathcal{L}_{u}:=\mathcal{I}\backslash\{m_{1},\dots,m_{u}\}\cup\{\delta_{1},\dots,\delta_{u}\}\in\mathfrak{G}(\mathbf{L}(\mathbf{1})),\quad u\in[r].
\label{eq: chain of non-vanishing}
 \end{equation}
Note that $\delta_{u}\neq\delta_{t}$ for all $t<u$, since $\delta_{t}\in\mathcal{L}_{u-1}$ and $\delta_{u}\notin\mathcal{L}_{u-1}$. We set $\mathcal{L}_{0}:=\mathcal{I}$ to unify the notation. 

Now consider any $p\in[r]$: the thesis holds for $\mathcal{L}_{p-1}$ and $\mathcal{L}_{p}$ when $D(\mathcal{L}_{p-1})$, $D(\mathcal{L}_{p})\in\mathbb{C}$, so let us assume the contrary, say $D(\mathcal{L}_{p-1})\notin\mathbb{C}$ with an appropriate labelling of $\{\mathcal{I},\mathcal{J}\}$. This means that there exists a radical term $Y(\mathcal{L}_{p-1})_{\alpha\beta}^{ij}$ with $i,j\in\mathcal{L}_{p-1}$ and $\alpha,\beta\in\mathcal{L}_{p-1}^{\mathtt{C}}$. We invoke the decompositions (\ref{eq: quadrilateral decomposition}) 
\begin{equation} 
Y(\mathcal{L}_{p-1})_{\alpha\beta}^{ij} = - Y(\mathcal{L}_{p-1})_{\alpha\delta_{p}}^{im_{p}} \cdot Y(\mathcal{L}_{p-1})_{\alpha\delta_{p}}^{m_{p}j} \cdot  Y(\mathcal{L}_{p-1})_{\delta_{p}\beta}^{im_{p}} \cdot Y(\mathcal{L}_{p-1})_{\delta_{p}\beta}^{m_{p}j}. 
\label{eq: propagation algebraic extension through bases}
\end{equation} 
Since $\chi(\mathcal{L}_{p-1}\mid_{\alpha\beta}^{ij})$ satisfies (\ref{eq: all 3 terms non-vanishing}) by assumption and $h((\mathcal{L}_{p-1})_{\delta_{p}}^{m_{p}})\neq 0$ by construction (\ref{eq: chain of non-vanishing}), all factors on the right-hand side of (\ref{eq: propagation algebraic extension through bases}) derive from observable sets. By Lemma \ref{lem: common squarefree}, they lie in the same quadratic extension of $\mathbb{F}$; furthermore, by $Y(\mathcal{L}_{p-1})_{\alpha\beta}^{ij}\notin\mathbb{F}$, at least one of these $Y$-terms does not lie in $\mathbb{F}$. These two properties are preserved under the transformation rules (\ref{eq: Grassmann-Plucker translation exchange, vertical})-(\ref{eq: Grassmann-Plucker translation exchange, diagonal}), which let us move from $\mathcal{L}_{p-1}$ to $\mathcal{L}_{p}$, finding $D(\mathcal{L}_{p-1})=D(\mathcal{L}_{p})$.
Concatenating these equalities for all $p\in[r]$, we get $D(\mathcal{I})=D(\mathcal{J})$. 
\end{proof}
A consequence of the previous proposition is that the set $\Psi(D)$ of exponents of monomials appearing in $D$ is uniquely characterised by any radical $Y$-term, modulo multiplication by invertible elements in $\mathbb{C}(\mathbf{t})$, independently of the choice of $\mathcal{I}\in\mathfrak{G}(\mathbf{L}(\mathbf{t}))$. 

\section{\label{sec: All three-terms non-vanishing} Reduction of set functions: quadratic case}
As mentioned in the Introduction, in this paper we concentrate on the following:
\begin{assumption}
\label{claim: generic matrices} All the maximal minors of $\mathbf{R}(\mathbf{1})$ are non-vanishing, and there is $\mathcal{I}\in\mathfrak{G}(\mathbf{L}(\mathbf{1}))$, $\alpha_{1},\alpha_{2}\in\mathcal{I}^{\mathtt{C}}$ such that (\ref{eq: two generic columns condition}) is verified, equivalently,  $\chi(\mathcal{I}\mid_{\alpha_{1}\alpha_{2}}^{m_{1}m_{2}})$ satisfies (\ref{eq: all 3 terms non-vanishing}) for all distinct $m_{1},m_{2}\in\mathcal{I}$. 
\end{assumption} 
Before deepening the investigation of algebraic conditions started in the previous section, we provide another example (in addition to Example \ref{exa: principal minor counterexample}) to show that a constraint on $\mathfrak{G}(\mathbf{L}(\mathbf{1}))$ is required to guarantee combinatorial reductions. 
\begin{example}
\label{exa: non-projectively extensional, no bound} Consider $\mathbf{a}\in\mathbb{C}^{k}\setminus\{\mathbf{0}_{k}\}$ and a $(k\times s)$-dimensional matrix $\mathbf{e}$ of units of $\mathbb{C}(\mathbf{t})$ such that each minor $\det(\mathbf{e}_{\mathcal{A};\mathcal{B}})$, where $\mathcal{A}\in\wp_{z}[k]$ and $\mathcal{B}\in\wp_{z}[s]$ for some $z\leq \min\{k,s\}$, does not vanish for a generic choice of $\mathbf{t}$. Then, we introduce
\begin{equation*}
\mathbf{L_{s}}(\mathbf{t}) :=  \left(\idd_{k}\mid\mathbf{a}\cdot \mathbf{1}_{s}^{\mathtt{T}}\right)\in\mathbb{C}^{k\times(k+s)},
\quad
\mathbf{R_{s}}(\mathbf{t}) := \left(\idd_{k}\mid\mathbf{e}(\mathbf{t})\right)^{\mathtt{T}}\in\mathbb{C}(\mathbf{t})^{(k+s)\times k}.
\end{equation*}
In other words, $\mathbf{L_{s}}$ is constructed by appending $s$ replicas of the column $\mathbf{a}$ to $\idd_{k}$. Each non-vanishing minor of $\mathbf{L_{s}}(\mathbf{t})$, apart from the one indexed by $[k]$, is indexed by $[k]_{\alpha}^{i}$ for some $(i,\alpha)\in[k]\times[k+1;k+s]$, and the associated term $\Delta_{\mathbf{L_{s}}(\mathbf{t})}([k]_{\alpha}^{i})\cdot\Delta_{\mathbf{R_{s}}(\mathbf{t})}([k]_{\alpha}^{i})$ is proportional to $e_{i,\alpha-k}$ and hence invertible in $\mathbb{C}(\mathbf{t})$. However, the condition (\ref{eq: consistency and curvature}) is not verified in general, in which case the reduction (\ref{eq: set-to-element integrability}) does not occur. 
\end{example}
\subsection{\label{subsec: allowed configuration for radical terms} Allowed configurations and squarefree decomposition}
\begin{prop}
\label{prop: factor for square and GCD} 
If $Q_{\alpha\beta}^{ij}$ in (\ref{eq: factorization square and squarefree}) is not invertible in $\mathbb{C}(\mathbf{t})$, then it is a binomial. 
\end{prop}
\begin{proof}
Taking into account the definition (\ref{eq: greatest common divisor}), we introduce 
 \begin{equation}
\hat{A}_{\alpha\beta}^{ij}:=\frac{A_{\alpha\beta}^{ij}}{G_{\alpha\beta}^{ij}},\quad\hat{B}_{\alpha\beta}^{ij}:=\frac{B_{\alpha\beta}^{ij}}{(G_{\alpha\beta}^{ij})^{2}},\quad\hat{Q}_{\alpha\beta}^{ij}:=\frac{Q_{\alpha\beta}^{ij}}{G_{\alpha\beta}^{ij}}.
\label{eq: rationalized A, B and Q w.r.t. GCD}
 \end{equation}
Whenever $D\in\mathbb{C}$, we find $(\hat{Q}_{\alpha\beta}^{ij})^{2}=\hat{B}_{\alpha\beta}^{ij}$, which means $\hat{Q}_{\alpha\beta}^{ij},\hat{B}_{\alpha\beta}^{ij}\in\mathbb{C}$ by Lemma \ref{lem: case B^(ij)_(ab) square}, while $Q_{\alpha\beta}^{ij}$ is invertible in $\mathbb{C}(\mathbf{t})$ at $D=\hat{B}_{\alpha\beta}^{ij}$. Focusing on the remaining cases, 
%
we take $p\in[d]$ satisfying $\partial_{t_{p}}\hat{Q}_{\alpha\beta}^{ij}\neq0$, where $\partial_{t_{p}}$ denotes the partial derivative with respect
to $t_{p}$, so $\hat{Q}_{\alpha\beta}^{ij}$ is a common factor for $\hat{B}_{\alpha\beta}^{ij}$ and $\partial_{t_{p}}\hat{B}_{\alpha\beta}^{ij}$. By (\ref{eq: rationalized A, B and Q w.r.t. GCD}) there is a term in $\hat{\chi}(\mathcal{I}\mid_{\alpha\beta}^{ij})=:\{E_{1},E_{2},E_{3}\}$ that is constant with respect to $t_{p}$, therefore we can write
\begin{eqnarray}
\hat{B}_{\alpha\beta}^{ij}
& = & E_{1}^{2}+E_{2}^{2}+E_{3}^{2}-2E_{1}E_{2}-2E_{1}E_{3}-2E_{2}E_{3}
\label{eq: univariate B, E}\\
& = & c_{1}^{2}t_{p}^{2d_{1}}+c_{2}^{2}t_{p}^{2d_{2}}+c_{3}^{2}-2c_{1}c_{2}t_{p}^{d_{1}+d_{2}}-2c_{1}c_{3}t_{p}^{d_{1}}-2c_{2}c_{3}t_{p}^{d_{2}}
\label{eq: univariate B}
\end{eqnarray}
with $c_{u}\in\mathbb{C}[\mathbf{t}_{\hat{p}}]$, $E_{u}=c_{u}t_{p}^{d_{u}}$, $u\in[3]$, and $d_{1}\geq d_{2}\geq d_{3}=0$. Each common factor of $\hat{B}_{\alpha\beta}^{ij}$ and $\partial_{t_{p}}\hat{B}_{\alpha\beta}^{ij}$ also divides the combination 
 \begin{equation}
\frac{t_{p}}{d_{1}}\cdot \partial_{t_{p}}\hat{B}_{\alpha\beta}^{ij}-\hat{B}_{\alpha\beta}^{ij} =  \left(c_{1}t_{p}^{d_{1}}-c_{2}t_{p}^{d_{2}}+c_{3}\right)\cdot \left(c_{1}t_{p}^{d_{1}}+c_{2}(1-2d_{1}^{-1}d_{2}) t_{p}^{d_{2}}-c_{3}\right).
\label{eq: factorizable combination}
 \end{equation}
However, we also find 
\begin{eqnarray}
\hat{B}_{\alpha\beta}^{ij}-(c_{1}t_{p}^{d_{1}}-c_{2}t_{p}^{d_{2}}+c_{3})^{2} & = & -4c_{1}c_{3}t_{p}^{d_{1}},
\label{eq: combination to power d1}\\
\hat{B}_{\alpha\beta}^{ij}-(c_{1}t_{p}^{d_{1}}-c_{2}t_{p}^{d_{2}}-c_{3})^{2} & = & -4c_{2}c_{3}t_{p}^{d_{2}}.
\label{eq: combination to power d2}
\end{eqnarray}
From (\ref{eq: combination to power d1}) and $c_3\neq 0$ we see that $\hat{B}_{\alpha\beta}^{ij}$
and $c_{1}t_{p}^{d_{1}}-c_{2}t_{p}^{d_{2}}+c_{3}$ are coprime in $\mathbb{C}[\mathbf{t}]$, so from (\ref{eq: factorizable combination}) any common factor between
$\hat{B}_{\alpha\beta}^{ij}$ and $\partial_{t_{p}}\hat{B}_{\alpha\beta}^{ij}$ also divides 
\begin{equation*}
w_{1}:=c_{1}t_{p}^{d_{1}}+c_{2}(1-2d_{1}^{-1}d_{2})\cdot t_{p}^{d_{2}}-c_{3}.
\end{equation*}
If $d_{1}=d_{2}$ we find the reductions 
\begin{eqnarray}
w_{1,d_{1}=d_{2}} & := & (c_{1}-c_{2})t_{p}^{d_{1}}-c_{3},
\label{eq: w_1, case d_1=d_2}\\
(\hat{B}_{\alpha\beta}^{ij})_{d_{1}=d_{2}} & := & (c_{1}-c_{2})^{2}t_{p}^{2d_{1}}+c_{3}^{2}-2(c_{1}+c_{2})c_{3}t_{p}^{d_{1}}.
\label{eq: B, case d_1=d_2}
\end{eqnarray}
Taking into account their combination 
 \begin{equation}
    - (\hat{B}_{\alpha\beta}^{ij})_{d_{1}=d_{2}} + w_{1,d_{1}=d_{2}}^2 = 4 c_{2} c_{3} \cdot t_{p}^{d_{1}}
\label{eq: linear combination of B and w_1, case d_1=d_2}
 \end{equation}
we find that $w_{1,d_{1}=d_{2}}$ and $(\hat{B}_{\alpha\beta}^{ij})_{d_{1}=d_{2}}$ are coprime. Thus, we get $d_{1}\neq d_{2}$ and look at the combination 
\begin{eqnarray}
&  & -t_{p}\cdot \partial_{t_{p}}\hat{B}_{\alpha\beta}^{ij}-\frac{2c_{3}d_{2}-2(d_{1}-d_{2})\cdot(c_{1}t_{p}^{d_{1}}-c_{2}t_{p}^{d_{2}}-c_{3})}{d_{1}-d_{2}}\cdot d_{1}\cdot w_{1}\nonumber \\
& = & 2(d_{1}-d_{2})\left(c_{1}t_{p}^{d_{1}}-c_{2}t_{p}^{d_{2}}-c_{3}\right)\left(c_{2}\cdot t_{p}^{d_{2}}-c_{3}d_{1}^{2}(d_{1}-d_{2})^{-2}\right).\label{eq: third common}
\end{eqnarray}
Noting that $c_{1}t_{p}^{d_{1}}-c_{2}t_{p}^{d_{2}}-c_{3}$ and $\hat{B}_{\alpha\beta}^{ij}$ are coprime by (\ref{eq: combination to power d2}), any common factor has to divide 
 \begin{equation}
w_{2}:=c_{2}\cdot t_{p}^{d_{2}}-c_{3}d_{1}^{2}(d_{1}-d_{2})^{-2}
\label{eq: d2 roots of units}
 \end{equation}
which also entails $d_{2}>0$ so as to have $\partial_{t_{p}}\hat{Q}_{\alpha\beta}^{ij}\neq0$. Finally, each common divisor of $\hat{B}_{\alpha\beta}^{ij}$ and $\partial_{t_{p}}\hat{B}_{\alpha\beta}^{ij}$ also divides
\begin{eqnarray}
w_{3} & := & w_{1}-(1-2d_{1}^{-1}d_{2})\cdot w_{2}\nonumber \\
& = & c_{1}t_{p}^{d_{1}}-(d_{1}-d_{2})^{-2}c_{3}d_{2}^{2}.
\label{eq: d1 roots of units}
\end{eqnarray}
The two factors (\ref{eq: d2 roots of units}) and (\ref{eq: d1 roots of units}) have a non-trivial common factor only if the condition 
 \begin{equation}
c_{1}^{-d_{2}}(d_{1}-d_{2})^{-2d_{2}}c_{3}^{d_{2}}d_{2}^{2d_{2}}=c_{2}^{-d_{1}}c_{3}^{d_{1}}d_{1}^{2d_{1}}(d_{1}-d_{2})^{-2d_{1}}\label{eq: condition monomials resonance in B, 0}
 \end{equation}
is satisfied, equivalently, only if 
 \begin{equation}
\frac{E_{2}^{d_{1}}}{d_{1}^{2d_{1}}}=\frac{E_{3}^{d_{1}-d_{2}}}{(d_{1}-d_{2})^{2(d_{1}-d_{2})}}\cdot\frac{E_{1}^{d_{2}}}{d_{2}^{2d_{2}}},\quad d_{1}>d_{2}>0.\label{eq: condition monomials resonance in B}
 \end{equation}
If this relation holds, then we can introduce  
 \begin{equation}
\varrho:=d_{2}/d_{1},\quad\varrho_{0}:=\gcd(d_{1},d_{2})/d_{1}
\label{eq: defining exponents for resonance}
 \end{equation}
and identify a unit $\mathrm{r}\in\mathbb{C}(\mathbf{t})$ such that 
\begin{equation}
\mathrm{r}^{1/\varrho_{0}}:=\left(\varrho^{-1}-1\right)^{2}E_{3}^{-1}E_{1}. 
\label{eq: homogeneous with respect to clustering}
 \end{equation}
In more detail, from (\ref{eq: condition monomials resonance in B}) we can write $\Psi(E_{u}\cdot E_{3}^{-1})=:\gcd(\Psi(E_{u}\cdot E_{3}^{-1}))\cdot\mathbf{f}$ where $\mathbf{f}$ is a primitive vector independent of $u$, which can be completed to a unimodular matrix $\mathbf{V}\in\mathbb{Z}^{d\times d}$ \cite{Newman1985}. Following Remark \ref{rem: invertible change of indeterminates}, we get an associated ring isomorphism 
that maps $w_{2}$ and $w_{3}$ into univariate binomials in $\mathbb{C}(s_{1})$. Then, 
any non-trivial common divisor of $w_{2}$ and $w_{3}$ corresponds to a common factor for the images of these two binomials in $\mathbb{C}(s_{1})$, so it divides $\mathrm{r}-1$ where $\mathrm{r}$ satisfies (\ref{eq: homogeneous with respect to clustering}). Using (\ref{eq: homogeneous with respect to clustering}) and (\ref{eq: condition monomials resonance in B}) to write $\hat{B}_{\alpha\beta}^{ij}$ in terms of $\mathrm{r}$ up to units of $\mathbb{C}(\mathbf{t})$, it is easily checked that it has a double root at $\mathrm{r}=1$. Thus, we get 
 \begin{equation}
\hat{Q}_{\alpha\beta}^{ij}=E_{1}^{\varrho_{0}}-\left(\varrho^{-1}-1\right)^{-2\cdot\varrho_{0}}E_{3}^{\varrho_{0}}
\label{eq: quadratic factor in resonance}
 \end{equation}
where the complex phases of the two summands are determined by $\mathrm{r}$ up to a common invertible factor.  
\end{proof} 
\begin{rem}
\label{rem: homogeneous coordinate}
Situations where $\hat{Q}_{\alpha\beta}^{ij}$ is not invertible in $\mathbb{C}(\mathbf{t})$ are associated with two disjoint tuples of variables $\mathbf{t}_{1},\mathbf{t}_{3}$ extracted from $\mathbf{t}$, which satisfy $E_{u}=c_{u}\mathbf{t}_{u}^{\mathbf{e}_{u}/\varrho_{0}}\in\hat{\chi}(\mathcal{I}\mid_{\alpha\beta}^{ij})$ with $c_{u}\in\mathbb{C}\setminus\{0\}$ and 
$\mathbf{e}_{u}\in\mathbb{N}^{\#\mathbf{t}_{u}}$ for both $u\in\{1,3\}$. The label switch $\mathbf{t}_{1}\leftrightarrows\mathbf{t}_{3}$ induces $\varrho\leftrightarrows1-\varrho$ consistently with (\ref{eq: condition monomials resonance in B}) and preserves the factors of $\hat{Q}_{\alpha\beta}^{ij}$, as it only contributes as a constant factor $(\varrho^{-1}-1)^{2\cdot\varrho_{0}}$ multiplying
$\hat{Q}_{\alpha\beta}^{ij}$ in (\ref{eq: quadratic factor in resonance}). In terms of the univariate notation (\ref{eq: homogeneous with respect to clustering}), this exchange leads to the change of variable $\mathrm{r}\mapsto\mathrm{r}^{-1}$.
\end{rem}
The previous proposition leads us to introduce the following notation. 
\begin{defn}
\label{def: coherent and balanced sets}
Given $\mathbf{f}_{1},\mathbf{f}_{3}\in\mathbb{Z}^{d}$, 
we say that $\mathbf{M}$ is  \emph{$(\mathbf{f}_{1},\mathbf{f}_{3})$-homogeneous} if, 
for all $\mathbf{x}\in\mathbf{M}$, $\Psi(\mathbf{x}^{-1}\cdot \mathbf{M})$ lies in a unidimensional submodule of the span of $\mathbf{f}_{1}$ and $\mathbf{f}_{3}$ in $\mathbb{Z}^{d}$ (therefore, $\Psi(\mathbf{M})$ lies in a unidimensional affine submodule of $\mathbb{Z}^{d}$). 
\end{defn}
The following corollary guarantees that even $D$ can be expressed as a polynomial in the variable $\mathrm{r}$ satisfying (\ref{eq: homogeneous with respect to clustering}) when 
$\hat{Q}_{\alpha\beta}^{ij}$ is not invertible. 
\begin{cor}
\label{cor: class-2 D, variables and independence} When $Q_{\alpha\beta}^{ij}$ is not invertible in $\mathbb{C}(\mathbf{t})$, $\mathrm{Supp}(D)$ is $(\mathbf{e}_{1},\mathbf{e}_{3})$-homogeneous where $\{\mathbf{e}_{1},\mathbf{e}_{3}\}$ $=\Psi(\hat{Q}_{\alpha\beta}^{ij})$. 
\end{cor}
\begin{proof}
Consider the two sub-tuples $\mathbf{t}_{1},\mathbf{t}_{3}$
of $\mathbf{t}$ introduced in Remark \ref{rem: homogeneous coordinate}. Both $\mathrm{Supp}(\hat{B}_{\alpha\beta}^{ij})$ and $\mathrm{Supp}(\hat{Q}_{\alpha\beta}^{ij})$ are $(\mathbf{e}_{1},\mathbf{e}_{3})$-homogeneous, as they can be expressed in terms of $\mathrm{r}$ through (\ref{eq: homogeneous with respect to clustering}) and (\ref{eq: condition monomials resonance in B}). For any monomial order $\preceq$ on $\mathrm{Supp}(D)\cup\mathrm{Supp}(\hat{Q}_{\alpha\beta}^{ij})$, if there exists a set $\mathbf{F}\subseteq \mathrm{Supp}(D)$ that makes $\mathrm{Supp}(D)$ $(\mathbf{e}_{1},\mathbf{e}_{3})$-inhomogeneous, then we can look at the minimum of such monomials with respect to $\preceq$. Specifically, the product $(\min\mathrm{Supp}(\hat{Q}_{\alpha\beta}^{ij})^{2})\cdot(\min\mathbf{F})$ appears in $\mathrm{Supp}(\hat{B}_{\alpha\beta}^{ij})$ and makes this set $(\mathbf{e}_{1},\mathbf{e}_{3})$-inhomogeneous too, i.e. a contradiction. Therefore, $\mathrm{Supp}(D)$ is $(\mathbf{e}_{1},\mathbf{e}_{3})$-homogeneous, which implies that $\Psi(D)$ lies in the $\mathbb{Z}$-submodule generated by $\{\mathbf{e}_{1},\mathbf{e}_{3}\}=\Psi(\hat{Q}_{\alpha\beta}^{ij})$. 
\end{proof}
\begin{rem}[Characterization of the binary entropy function]
\label{rem: characterization of the binary entropy function}
We can write (\ref{eq: condition monomials resonance in B}) as
 \begin{equation}
\frac{E_{2}^{d_{1}}}{E_{3}^{d_{1}-d_{2}}E_{1}^{d_{2}}}=\left(\frac{d_{1}-d_{2}}{d_{1}}\right)^{-2(d_{1}-d_{2})}\cdot\left(\frac{d_{2}}{d_{1}}\right)^{-2d_{2}}\in\mathbb{Q}_{+}\label{eq: condition monomials resonance in B, rational factor}
 \end{equation}
which leads to  
 \begin{equation}
\frac{1}{2d_{1}}\log\left(\frac{E_{2}^{d_{1}}}{E_{3}^{d_{1}-d_{2}}E_{1}^{d_{2}}}\right)=H\left(\frac{d_{2}}{d_{1}}\right)
\label{eq: condition resonance entropy}
 \end{equation}
where $H(p)$ is the entropy associated with a Bernoulli random variable with parameter $p\in(0,1)$. 

The Bernoulli distribution is the basic law for dealing with random binary outcomes, whose functional form here is characterised by an algebraic condition: the discriminant $B_{\alpha\beta}^{ij}$ is not squarefree only if (\ref{eq: condition resonance entropy}) holds, which entails a normalisation condition for exponents $\Psi(\hat{\chi}(\mathcal{I}\mid_{\alpha\beta}^{ij}))$ (equivalently, the homogeneity of $\hat{B}_{\alpha\beta}^{ij}$, $\hat{Q}_{\alpha\beta}^{ij}$, and $D$ discussed in Corollary \ref{cor: class-2 D, variables and independence}) and determines the form of $H$ as a function of $d_{2}/d_{1}=\varrho$. As a basis for future research, this suggests a deeper study of the connections between the factorisation properties of sparse polynomials obtained from the deformation of different determinantal expansions and entropy functions. 
\end{rem}
\begin{example}
\label{exa: class-2 D first occurrence} 
The first case with $\hat{Q}_{\alpha\beta}^{ij}\notin\mathbb{C}$ is found at $d_{1}=2d_{2}$ in (\ref{eq: condition monomials resonance in B}), leading to the ratios $\mathrm{gcd}(d_{1},d_{2})/d_{1}=\frac{1}{2}$ and $d_{2}/(d_{1}-d_{2})=1$. Thus, we get 
 \begin{equation}
D=\varepsilon_{1}^{2}-6\varepsilon_{1}\varepsilon_{3}+\varepsilon_{3}^{2}
\label{eq: class-2 D first occurrence}
 \end{equation}
where $\varepsilon_{u}^{2}=E_{u}$, $u\in\{1,3\}$, and the factorisation
(\ref{eq: factorization square and squarefree}) is given by 
\begin{equation} 
B_{\alpha\beta}^{ij}=(\varepsilon_{1}-\varepsilon_{3})^{2}\cdot(\varepsilon_{1}^{2}-6\varepsilon_{1}\varepsilon_{3}+\varepsilon_{3}^{2}).
\label{eq: class-2 B first occurrence}
\end{equation}

This example also shows that ambiguity can arise when reconstructing
$B_{\alpha\beta}^{ij}$ from $D$: indeed, the case $D=\varepsilon_{1}^{2}-6\varepsilon_{1}\varepsilon_{3}+\varepsilon_{3}^{2}$ can be associated with both the configurations $\chi_{\mathrm{I},a}(\mathcal{I}\mid_{\alpha\beta}^{ij})=\{\varepsilon_{1},2\varepsilon_{1},\varepsilon_{3}\}$
and $\chi_{\mathrm{I},b}(\mathcal{I}\mid_{\alpha\beta}^{ij})=\{\varepsilon_{1},2\varepsilon_{3},\varepsilon_{3}\}$
with $\hat{Q}_{\alpha\beta}^{ij}=1$, as well as $\chi_{\mathrm{II}}(\mathcal{I}\mid_{\alpha\beta}^{ij})=\{\varepsilon_{1}^{2},4\varepsilon_{1}\varepsilon_{3},\varepsilon_{3}^{2}\}$ with $\hat{Q}_{\alpha\beta}^{ij}=\varepsilon_{1}-\varepsilon_{3}$. We highlight that the sets $\chi_{\mathrm{I},a}(\mathcal{I}\mid_{\alpha\beta}^{ij})$ and $\chi_{\mathrm{I},b}(\mathcal{I}\mid_{\alpha\beta}^{ij})$ are $(\Psi(\varepsilon_{1}),\Psi(\varepsilon_{3}))$-homogeneous
as well, in line with the proof of Corollary \ref{cor: class-2 D, variables and independence}. 
\end{example}

\subsection{\label{subsec: Recovering local from global data} Recovering local from global information}

Now we address the reverse problem of recovering local data that generate $Y$-terms, i.e. the sets $\chi(\mathcal{I}\mid_{\alpha\beta}^{ij})$, starting from the global information provided by $D$. In the present context, the term global means independent of the choice of the basis and the indices of the $Y$-term. 
%
\begin{lem}
\label{lem: binomial D} If $\#\mathrm{Supp}(D)=2$, then $B_{\alpha\beta}^{ij}$ is squarefree in $\mathbb{C}(\mathbf{t})$ and there exist two distinct terms $E_{u},E_{w}\in\chi(\mathcal{I}\mid_{\alpha\beta}^{ij})$ and $c_{S}\in\{1,-1\}$ such that $E_{u}=c_{S}\cdot E_{w}$.
\end{lem}
\begin{proof}
From (\ref{eq: quadratic factor in resonance}), a non-invertible square factor $\hat{Q}_{\alpha\beta}^{ij}\in\mathbb{C}(\mathbf{t})$ of $\hat{B}_{\alpha\beta}^{ij}$ when $\#\mathrm{Supp}(D)=2$ entails the existence of $\mathbf{e}_{1},\mathbf{e}_{3}\in\mathbb{Z}^{d}$ such that 
 \begin{equation}
\Psi(D)=\{2\cdot(\varrho_{0}^{-1}-1)\cdot\mathbf{e}_{1},2\cdot(\varrho_{0}^{-1}-1)\cdot\mathbf{e}_{3}\}\label{eq: D when binomial D}
 \end{equation}
We can order the terms in $\hat{B}_{\alpha\beta}^{ij}$, $\hat{Q}_{\alpha\beta}^{ij}$, and $D$ based on their fractional degree $\nu_{1}(\cdot)$ with respect to a non-empty tuple between those involved in Remark \ref{rem: homogeneous coordinate} and Definition \ref{def: coherent and balanced sets}, say $\mathbf{t}_{1}$: we normalise $\nu_{1}(E_{1})=1$, while $\nu_{1}(E_{3})=0$ and, from (\ref{eq: condition monomials resonance in B}), $\nu_{1}(E_{2})=\varrho$. The induced order is total since $\hat{B}_{\alpha\beta}^{ij}$, $\hat{Q}_{\alpha\beta}^{ij}$, and $D$ are $(\mathbf{e}_{1},\mathbf{e}_{3})$-homogeneous.

This allows comparing the exponents in $\nu_{1}(\mathrm{Supp}(\hat{B}_{\alpha\beta}^{ij}))$ with the corresponding ones in the expansion
$\nu_{1}(\mathrm{Supp}(D\cdot(\hat{Q}_{\alpha\beta}^{ij})^{2}))$: the minimum is $0$ for all sets, while the least non-vanishing weights are
\begin{equation*}
\min_{(0,1)}\nu_{1}(\mathrm{Supp}((\hat{Q}_{\alpha\beta}^{ij})^{2})) = \varrho_{0}, \quad 
\min_{(0,+\infty)}\nu_{1}(\mathrm{Supp}(D)) =  2\cdot(1-\varrho_{0}), \quad 
\min_{(0,1)} \nu_{1}(\mathrm{Supp}(\hat{B}_{\alpha\beta}^{ij})) = \varrho
\end{equation*}
where the last equality follows from $0<d_{2}<d_{1}$, as stated in (\ref{eq: condition monomials resonance in B}). The same condition, combined with (\ref{eq: defining exponents for resonance}), entails $ 0\leq\varrho_{0}\leq \frac{1}{2}$ and $\varrho_{0}\leq \varrho < 1$, which can be summarised by
 \begin{equation}
   0 < \varrho_{0} \leq \varrho < 1\leq 2\cdot (1-\varrho_{0}).
   \label{eq: constraints on ratios}
 \end{equation} 
To satisfy (\ref{eq: factorization square and squarefree}), two elements of $\{2\cdot(1-\varrho_{0}),\varrho_{0},\varrho\}$ must coincide, and (\ref{eq: constraints on ratios}) forces $\varrho_{0}=\varrho$, that is, $d_{2}\mid d_{1}$. We can repeat this argument considering the dual of the previous order, which coincides with the order induced by the valuation $\nu_{3}(\cdot)$ with respect to $\mathbf{t}_{3}$ when it is non-empty, or equivalently considering maxima instead of minima. This corresponds to the transformations $\mathbf{e}_{1}\leftrightarrows\mathbf{e}_{3}$ and $\varrho\leftrightarrows1-\varrho$, according to (\ref{eq: condition monomials resonance in B}). Thus, we also get $\varrho_{0}=1-\varrho$, which means $\varrho=\varrho_{0}=\frac{1}{2}$. This leads to (\ref{eq: class-2 D first occurrence}) and does not satisfy $\#\mathrm{Supp}(D)=2$. Therefore, $D=\hat{B}_{\alpha\beta}^{ij}$ and to obtain $\#\mathrm{Supp}(D)=2$, we find two linearly dependent elements of $\hat{\chi}(\mathcal{I}\mid_{\alpha\beta}^{ij})$, say $E_{2}=c_{S}\cdot E_{3}$, where $c_{S}\in\mathbb{C}$ satisfies $D=(c_{S}-1)^{2}E_{3}^{2}-2(c_{S}+1)E_{3}E_{1}+E_{1}^{2}$. This is a binomial if and only if $c_{S}\in\{1,-1\}$. 
\end{proof}
A possible issue arising from this process has been highlighted in
Example \ref{exa: class-2 D first occurrence}. Before analysing it in more detail, the results of Proposition \ref{prop: factor for square and GCD} and Lemma \ref{lem: binomial D} suggest the following definition. 
\begin{defn}
\label{def: type and class of radical terms}[\textit{Type and class
of radical terms}] 
We say that a configuration $\chi(\mathcal{I}\mid_{\alpha\beta}^{ij})$
(or the associated expressions $B_{\alpha\beta}^{ij}$ and $Y_{\alpha\beta}^{ij}$)
is of \textit{$G$-type} (generic type) if $\#\Psi(D)>2$ and is of
\textit{$S$-type} (singular type) if $\#\Psi(D)=2$. A $G$-type
configuration $\chi(\mathcal{I}\mid_{\alpha\beta}^{ij})$ is \textit{class-$\mathrm{I}$} if $Q_{\alpha\beta}^{ij}$ is a monomial, and it is \textit{class-$\mathrm{II}$ }if $\hat{Q}_{\alpha\beta}^{ij}$
is a binomial (\ref{eq: quadratic factor in resonance}). 
\end{defn}
We stress that the type of configuration is independent of the set $\chi(\mathcal{I}\mid_{\alpha\beta}^{ij})$, since it depends only on $D$, while this may not hold for the class, as shown in Example \ref{exa: class-2 D first occurrence}.
\begin{prop}
\label{prop: class-2, recover hat-B from D} For $G$-type configurations, we can reconstruct $\hat{B}_{\alpha\beta}^{ij}$ from $D$ and the knowledge of the class of $\hat{B}_{\alpha\beta}^{ij}$. 
\end{prop}
\begin{proof}
We focus on class-$\mathrm{II}$ configurations; otherwise, $\hat{B}_{\alpha\beta}^{ij}=D$. 
Let $\Omega:=\#\mathrm{Supp}(D)>2$ and proceed as follows: 
choose any variable $t_{p}$ such that $\partial_{t_{p}}D\neq 0$ and order the elements in $\mathrm{Supp}(D)$ according to their degree with respect to $t_{p}$. This order is total since all the elements in $\Psi(D)$ belong to a $\mathbb{Z}$-submodule of $\mathbb{Z}^{d}$ generated by two vectors $\mathbf{f}_{1},\mathbf{f}_{\Omega}$, and $\mathrm{Supp}(D)$ is $(\mathbf{f}_{1},\mathbf{f}_{\Omega})$-homogeneous by Corollary \ref{cor: class-2 D, variables and independence}. We can identify the maximum $\mathbf{d}_{1}$ and minimum $\mathbf{d}_{\Omega}$ in $\mathrm{Supp}(D)$, so the same order is induced by the valuation $\nu_{\mathbf{d}_{1}}(\mathbf{d}_{u}):=\deg_{t_{p}}(\mathbf{d}_{u})/\deg_{t_{p}}(\mathbf{d}_{1})$, $u\in[\Omega]$. 
This function is related to the valuation $\nu_{E_{1}}$ with respect to the monomial $E_{1}$ introduced in (\ref{eq: univariate B, E}) via 
\begin{equation} 
\nu_{\mathbf{d}_{1}}(E_{1})^{-1} =\nu_{E_{1}}(\mathbf{d}_{1})=\max(\nu_{E_{1}})(\hat{B}_{\alpha\beta}^{ij})-\max\nu_{E_{1}}((\hat{Q}_{\alpha\beta}^{ij})^2)=2\cdot (1-\varrho_{0}).
\label{eq: relation between normalizations}
\end{equation}

Denoting $q_{u}:=\nu_{\mathbf{d}_{1}}(\mathbf{d}_{u})$, $u\in[\Omega]$, we recover $\varrho_{0}:=\gcd(d_{1},d_{2})/d_{1}$ by looking at the minimal gap  
 \begin{equation} 
    \mu:=\min_{u}\{q_{u}-q_{u+1}\}.
    \label{eq: minimal gap in D}
 \end{equation}
Indeed, we can evaluate $\nu_{\mathbf{d}_{1}}$ at monomials in (\ref{eq: univariate B, E}), using the expansion of (\ref{eq: factorization square and squarefree}) to relate them to monomials in $\mathrm{Supp}(D)$: we have 
 \begin{eqnarray}
q_{1}-q_{2} > \nu_{\mathbf{d}_{1}}(E_{1}) \varrho_{0}  & \Rightarrow & \varrho_{0} =  \frac{\nu_{\mathbf{d}_{1}}(E_{1})-\nu_{\mathbf{d}_{1}}(E_{2})}{\nu_{\mathbf{d}_{1}}(E_{1})} = 1-\varrho
\label{eq: coincidence of left-gaps} \\  q_{\Omega-1}-q_{\Omega} > \nu_{\mathbf{d}_{1}}(E_{1}) \varrho_{0} & \Rightarrow & \varrho_{0} = \frac{\nu_{\mathbf{d}_{1}}(E_{2})-\nu_{\mathbf{d}_{1}}(E_{3})}{\nu_{\mathbf{d}_{1}}(E_{1})} = \varrho
\label{eq: coincidence of right-gaps}
 \end{eqnarray}
so the combination of the two previous assumptions entails $\varrho=\varrho_{0}=\frac{1}{2}$ too; from (\ref{eq: relation between normalizations}), in this case we can conclude $\nu_{\mathbf{d}_{1}}(E_{1})(\varrho_{0})=\frac{1}{2}$, which is an upper bound for $\mu$. Even when (\ref{eq: coincidence of left-gaps}) and (\ref{eq: coincidence of right-gaps}) do not hold simultaneously, by contraposition, we infer 
\begin{equation*} 
    \mu \leq \min\{q_{1}-q_{2},q_{\Omega-1}-q_{\Omega}\}\leq\nu_{\mathbf{d}_{1}}(E_{1})\cdot\varrho_{0}.
\end{equation*}
On the other hand, from Corollary \ref{cor: class-2 D, variables and independence} we know that $\Psi(D)$ lies in the $\mathbb{Z}$-submodule generated by $\Psi(\hat{Q}_{\alpha\beta}^{ij})$, so $\mu$ is bounded from below by $\max\nu_{\mathbf{d}_{1}}(\mathrm{Supp}(\hat{Q}_{\alpha\beta}^{ij}))=\nu_{\mathbf{d}_{1}}(E_{1})\varrho_{0}$. Therefore, $\mu=\nu_{\mathbf{d}_{1}}(E_{1})\cdot\varrho_{0}$
and we can recover $\varrho_{0}$ from $\mu$ using (\ref{eq: relation between normalizations}). Furthermore, from the knowledge of $\varrho_{0}$ we can consider an appropriate scaling for elements $q_{u}$, that is, $s_{u}:=\nu_{E_{1}}(\mathbf{d}_{u})=2\cdot (1-\varrho_{0})\cdot q_{u}$, $u\in [\Omega]$. 

When $\varrho_{0}=\frac{1}{2}$ we infer $\varrho=\varrho_{0}$ and recover the decomposition shown in Example \ref{exa: class-2 D first occurrence}, so we focus on $\varrho_{0}<\frac{1}{2}$. Looking at the expansion of $\hat{B}_{\alpha\beta}^{ij}=(\hat{Q}_{\alpha\beta}^{ij})^{2}\cdot D$, a necessary and sufficient condition for $\varrho_{0}<\frac{1}{2}$ is that a cancelation occurs and one between 
\begin{eqnarray} 
0 & = & E_{1}^{2\varrho_{0}}\mathbf{d}_{2}-2\left(\varrho^{-1}-1\right)^{-2\cdot\varrho_{0}}E_{1}^{\varrho_{0}}E_{3}^{\varrho_{0}}\mathbf{d}_{1}, 
\label{eq: constraint from simplification, left-hand-side}\\
0 & = & \left(\varrho^{-1}-1\right)^{-4\cdot\varrho_{0}}E_{3}{}^{2\varrho_{0}}\mathbf{d}_{\Omega-1}-2\left(\varrho^{-1}-1\right)^{-2\cdot\varrho_{0}}E_{1}^{\varrho_{0}}E_{3}^{\varrho_{0}}\mathbf{d}_{\Omega}
\label{eq: constraint from simplification, right-hand-side} 
\end{eqnarray}
holds. We can discriminate the compatibility of the previous two relations by looking at $\Xi:=4\mathbf{d}_{\Omega}\mathbf{d}_{1}\mathbf{d}_{2}^{-1}\mathbf{d}_{\Omega-1}^{-1}$. At $\Xi=1$, (\ref{eq: constraint from simplification, left-hand-side}) and (\ref{eq: constraint from simplification, right-hand-side}) are consistent and define a unique value $\mathrm{r}=E_{1}^{\varrho_{0}}(\varrho^{-1}-1)^{2\cdot\varrho_{0}}E_{3}^{-\varrho_{0}}=2\mathbf{d}_{2}^{-1}\mathbf{d}_{1}$. 
On the other hand, $\Xi\neq1$ means that exactly one of these relations holds, i.e. (\ref{eq: constraint from simplification, left-hand-side}) at $\varrho=\varrho_{0}$ and (\ref{eq: constraint from simplification, right-hand-side}) at $\varrho=1-\varrho_{0}$. The relation associated with a choice of $\varrho$ determines the corresponding unit $\mathrm{r}_{\varrho}$ introduced in (\ref{eq: homogeneous with respect to clustering}), and hence the $B$-term, which we denote by $\mathrm{B}_{\varrho}(\mathrm{r}_{\varrho})$. Given $\mathrm{B}_{\varrho_{0}}(\mathrm{r}_{\varrho_{0}})$, the polynomial $\mathrm{r}^{2/\varrho_{0}}\cdot\mathrm{B}_{1-\varrho_{0}}(\mathrm{r}_{\varrho_{0}}^{-1})$ has the same factors in $\mathbb{C}(\mathbf{t})$ as noted in Remark \ref{rem: homogeneous coordinate}, thus they are the unique $B$-terms compatible with $D$ at $\varrho_{0}$ and $1-\varrho_{0}$, respectively. 

In this way, we obtain information that uniquely defines $(Q_{\alpha\beta}^{ij})^{2}$ and, from (\ref{eq: factorization square and squarefree}), $B_{\alpha\beta}^{ij}$, up to units. 
\end{proof}
\begin{rem}
The last step of the previous proof can also be checked by expanding $(\mathrm{r}-k_{\varrho})^{2}\cdot D$ with $k_{\varrho}\in\mathbb{C}$, verifying the compatibility of $D$ with the two choices $\varrho\in\{\varrho_{0},1-\varrho_{0}\}$. Fix the unit $\mathrm{r}$ satisfying (\ref{eq: homogeneous with respect to clustering}) at $\varrho=\varrho_{0}$, so that $k_{\varrho_{0}}=1$. Looking at the $(2/\varrho_{0})$-degree term in the expansion, we get $\mathrm{r}^{2}\mathbf{d}_{1}=E_{1}^{2}$; as for the $(2/\varrho_{0}-1)$-degree term, we find $\mathrm{r}^{2}\mathbf{d}_{2}-2k_{\varrho_{0}}\mathrm{r}\mathbf{d}_{1}=0$ and $\mathrm{r}^{2}\mathbf{d}_{2}-2k_{1-\varrho_{0}}\mathrm{r}\cdot\mathbf{d}_{1}=-2E_{1}E_{2}$, where $E_{1},E_{2}$ refer to the expression (\ref{eq: univariate B, E}) at $\varrho=1-\varrho_{0}$. Solving for $\mathrm{r}$ and $E_{2}$, we obtain 
\begin{equation}
    \mathrm{r}=2\mathbf{d}_{1}\mathbf{d}_{2}^{-1},\quad E_{2}=-2^{-1}\mathbf{d}_{1}^{-1}\mathbf{d}_{2}\cdot E_{1}\cdot(1-k_{1-\varrho_{0}}).
\label{eq: ambiguity reconstruction, level 1}
\end{equation}
If $\varrho_{0}<\frac{1}{3}$, looking at the $(2/\varrho_{0}-2)$-degree term, we also get $\mathrm{r}^{2}\mathbf{d}_{3}-2k_{\varrho_{0}}\mathrm{r}\mathbf{d}_{2}+\mathbf{d}_{1}=0$ and $\mathrm{r}^{2}\mathbf{d}_{3}-2k_{1-\varrho_{0}}\mathrm{r}\mathbf{d}_{2}+k_{1-\varrho_{0}}^{2}\mathbf{d}_{1}=E_{2}^{2}$, which can be combined with (\ref{eq: ambiguity reconstruction, level 1}) to infer $\mathbf{d}_{3}=\frac{1}{4}(4k_{\varrho_{0}}-1)\cdot\mathbf{d}_{1}^{-1}\mathbf{d}_{2}^{2}$ and $\mathbf{d}_{3}=\frac{1}{4}(2k_{1-\varrho_{0}}+1)\cdot\mathbf{d}_{1}^{-1}\mathbf{d}_{2}^{2}$. Together with $k_{\varrho_{0}}=1$ and $\mathbf{d}_{1}^{-1}\mathbf{d}_{2}^{2}\neq0$, the previous equations are compatible only if $k_{1-\varrho_{0}}=1$ too, so the two choices of $\varrho$ return the same terms $Q_{\alpha\beta}^{ij}$ and $B_{\alpha\beta}^{ij}$ up to units in $\mathbb{C}(\mathbf{t})$. The same result is easily checked by direct verification at $\varrho_{0}=\frac{1}{3}$. 
\end{rem}
\begin{thm}
\label{thm: distinguishability terms from Y} From any radical $Y_{\alpha\beta}^{ij}$
we can recover a finite number $N_{D}$ of monomial configurations $(E_{1},E_{2},E_{3})\in\mathbb{C}(\mathbf{t})^{3}$ defining radical $Y$-terms. 
In particular, there are no radical terms if $n-k > 2 N_{D}+2$. 
\end{thm}
\begin{proof}
All $B$-terms are homogeneous quadratic expressions of the elements of $\mathfrak{\chi}(\mathcal{I}\mid_{\alpha\beta}^{ij})$, so they can be recovered from $Y_{\alpha\beta}^{ij}$ up to $G_{\alpha\beta}^{ij}$.
We focus on $\hat{\chi}(\mathcal{I}\mid_{\alpha\beta}^{ij})$, consider a given monomial order on $\mathrm{Supp}(D)$, and label its elements accordingly, that is, $u>v$ implies $\mathbf{d}_{u}\prec\mathbf{d}_{v}$,
$u,v\in[\Omega]$ with $\Omega:=\#\mathrm{Supp}(D)$. 

Starting from class-$\mathrm{II}$ configurations, Proposition \ref{prop: class-2, recover hat-B from D} entails that $\varrho_{0}$, $\hat{Q}_{\alpha\beta}^{ij}$, and the set $\{\varrho,1-\varrho\}$ are uniquely defined by $D$; for each individual choice of the value $\varrho$ or $1-\varrho$, we can reconstruct $\hat{\chi}(\mathcal{I}\mid_{\alpha\beta}^{ij})$ from $\mathrm{Supp}(\hat{Q}_{\alpha\beta}^{ij})$ with this information. 
According to Remark \ref{rem: homogeneous coordinate}, the two choices are related by $\varrho\leftrightarrows1-\varrho$ and induce a permutation of the roots of $B_{\alpha\beta}^{ij}$. 

Now we move to class-$\mathrm{I}$ configurations and look at $\mathbf{d}_{1}=\max\mathrm{Supp}(D)$ and $\mathbf{d}_{\Omega}=\min\mathrm{Supp}(D)$, which coincide if and only if $Y_{\alpha\beta}^{ij}\in\mathbb{C}$.
We fully recover $\hat{\chi}(\mathcal{I}\mid_{\alpha\beta}^{ij})=\{E_{1},E_{2},E_{3}\}$
up to common factors when $\Omega\geq4$, since we can fix a square root of $E_{1}^2:=\mathbf{d}_{1}$ and obtain $E_{2}=-\frac{1}{2}E_{1}\cdot\mathbf{d}_{2}\mathbf{d}_{1}^{-1}$
and $E_{3}:=-2E_{2}\cdot\mathbf{d}_{\Omega}\mathbf{d}_{\Omega-1}^{-1}$.
The occurrence $\Omega=3$ means that two terms in $\hat{\chi}(\mathcal{I}\mid_{\alpha\beta}^{ij})$ are linearly proportional. To unify the notation, we still denote by $\varrho=0$, respectively $\varrho=1$, the choice $\Psi(E_{2})=\Psi(E_{3})$, respectively $\Psi(E_{2})=\Psi(E_{1})$. Each of the two choices $u\in\{1,3\}$ allows us to uniquely recover the constant $c$ from $B(E_{1},c\cdot E_{u},E_{3})=D$. 

Possible ambiguity between the two classes may arise, as shown in Example
\ref{exa: class-2 D first occurrence}. Thus, when $\#\mathrm{Supp}(D)>2$, the possible configurations follow from: i. the knowledge of the class; ii. the choice in $\{\varrho,1-\varrho\}$ defining $E_{2}$ ($\varrho\in[0,1]$); iii. the choice of the square root of $B_{\alpha\beta}^{ij}$; iv. the ordering of $\hat{\chi}(\mathcal{I}\mid_{\alpha\beta}^{ij})$ defining $(E_{1},E_{2},E_{3})$. 

When $\#\mathrm{Supp}(D)=2$, the possible configurations for $Y_{\alpha\beta}^{ij}$ generated by $D$ come from the following choices: i. the choice of the square root of the $B$-term 
ii. the constant $c\in\{1,-1\}$ mentioned in Lemma \ref{lem: binomial D}; iii. with reference to the same lemma, 
the choice of $\mathbf{e}\in\mathrm{Supp}(D)$ such that $\mathbf{e}=E_{u}=c\cdot E_{w}$, which extends the definition of $\varrho$ already specified at $\#\mathrm{Supp}D>0$; iv. the ordering $(E_{1},E_{2},E_{3})$. There are two choices for both the root of the $B$-term and the identification of the independent term in $\mathrm{Supp}(D)$; at $c=+1$ there are only three distinct configurations for $(E_{1},E_{2},E_{3})$ associated with the cyclic subgroup of $\mathcal{S}_{3}$, while there are six configurations at $c=-1$, namely the full $\mathcal{S}_{3}$ group. 

Finally, let us denote by $N_{D}$ the number of distinct forms 
for radical terms. We suppose that there exists a radical term $Y_{\alpha\beta}^{ij}$. From Lemma \ref{lem: observable sets with 0 are algebraic}, all sets $\chi(\mathcal{I}\mid_{\alpha\gamma}^{ij})$ and $\chi(\mathcal{I}\mid_{\beta\gamma}^{ij})$ satisfy (\ref{eq: all 3 terms non-vanishing}). Therefore, for each $\gamma\in\mathcal{I}^{\mathtt{C}}$, at least one between $Y(\mathcal{I})_{\alpha\gamma}^{ij}\notin \mathbb{F}$ and $Y(\mathcal{I})_{\beta\gamma}^{ij}\notin \mathbb{F}$ holds. If $n-k > 2 N_{D}+2$, Dirichlet's box principle would imply that one between $\alpha$ and $\beta$, say $\alpha$ with appropriate labelling, satisfies $Y_{\alpha\gamma_{u}}^{ij}\notin \mathbb{F}$ for $N_{D}+1$ distinct indices $\gamma_{u}\in \mathcal{I}^{\mathtt{C}}$. From the previous argument, a second application of Dirichlet's box principle lets us infer that there exist two indices, say $\gamma_{p}$ and $\gamma_{q}$, such that $Y_{\alpha\gamma_{p}}^{ij}=Y_{\alpha\gamma_{q}}^{ij}$. The identity (\ref{eq: associativity, d}) implies $Y_{\gamma_{p}\gamma_{q}}^{ij}=-1$, which contradicts Remark \ref{rem: no Y=-1}, so $n-k\leq 2 N_{D}+2$. 
\end{proof}
The proof of Theorem \ref{thm: distinguishability terms from Y} highlights the multiplicity sources for sets $\hat{\chi}(\mathcal{I}\mid_{\alpha\beta}^{ij})$ that are compatible with a given $D$, which add up to the ordering of monomials that defines the triple
 \begin{equation}
\overrightarrow{\chi}(\mathcal{I}\mid_{\alpha\beta}^{ij}):=(G_{\alpha\beta}^{ij})^{-1}\cdot\left(h(\mathcal{I})\cdot h(\mathcal{I}_{\alpha\beta}^{ij}),h(\mathcal{I}_{\alpha}^{i})\cdot h(\mathcal{I}_{\beta}^{j}),h(\mathcal{I}_{\beta}^{i})\cdot h(\mathcal{I}_{\alpha}^{j})\right).\label{eq: three-set triple}
 \end{equation}
We briefly discuss the situation where $\Psi(D)$ does not span a $3$-dimensional sublattice of $\mathbb{Z}^{d}$ and  
$\Psi(\hat{B}_{\alpha\beta}^{ij})$ 
can be expressed using the univariate notation (\ref{eq: homogeneous with respect to clustering}), which also applies when two terms in $\overrightarrow{\chi}(\mathcal{I}\mid_{\alpha\beta}^{ij})$ are proportional
over $\mathbb{C}$. 
\begin{rem}
\label{rem: if proportional, then unique constant ratio } 
For a $G$-type configuration $(E_1,E_2,E_3)=(\mathrm{r},c_{1}\cdot\mathrm{r},q_{1})$ with 
$c_{1}\in\mathbb{C}$, this constant can only assume values from a set $\{\kappa,\kappa^{-1}\}$ that is uniquely defined knowing $\hat{B}_{\alpha\beta}^{ij}$. The two choices correspond to different labels for the two proportional terms $E_1,E_2$. Different configurations for $\overrightarrow{\chi}(\mathcal{I}\mid_{\alpha\beta}^{ij})$ are described by permutations of two triples $(\mathrm{r},c_{1}\cdot\mathrm{r},q_{1})$ and $(\mathrm{r},c_{2},q_{2})$ such that $B(\mathrm{r},c_{1}\cdot\mathrm{r},q_{1})$ is equal to $B(\mathrm{r},c_{2},q_{2})$, up to units. In turn, both $(\mathrm{r},c_{2},q_{2})$ and $q_{1}\mathrm{r}^{-1}\cdot(\mathrm{r},c_{2},q_{2})=(q_{1},c_{2}q_{1}\mathrm{r}^{-1},q_{2}q_{1}\mathrm{r}^{-1})$ return the same $B$-term, up to units. Therefore, the ratio $c_{2}/q_{2}=(c_{2}q_{1}\mathrm{r}^{-1})/(q_{2}q_{1}\mathrm{r}^{-1})$ in the triple $(\mathrm{r},c_{2},q_{2})$ belongs to $\{\kappa,\kappa^{-1}\}$ as well. The explicit sets returning the same $B$-term are 
 \begin{equation}
\mathrm{C}_{1}:=\left\{ \mathrm{r},\kappa\cdot g,g\right\} ,\quad\mathrm{C}_{2}:=\left\{ (1-\kappa)\cdot g,\frac{\kappa\mathrm{r}}{1-\kappa},\frac{\mathrm{r}}{1-\kappa}\right\} .\label{eq: two configurations for proportional}
 \end{equation}
We also note that the set $\mathrm{C}_{2}$ can be obtained from $\mathrm{C}_{1}$ starting from 
(\ref{eq: associativity, d}): two $Y$-terms generated by the same set $\mathrm{C}_{1}$ but different roots of $\mathrm{B}$ in (\ref{eq: radical expression}) produce a term derived from $\mathrm{C}_{2}$, that is,  
 \begin{equation}
    \frac{\kappa g-g-\mathrm{r}+\sqrt{\mathrm{B}}}{2g}\cdot\frac{g-\kappa g-\mathrm{r}-\sqrt{\mathrm{B}}}{2\mathrm{r}}=\frac{\kappa-1}{2\mathrm{r}}\cdot\left((1-\kappa) g-\frac{\mathrm{r}}{1-\kappa}-\frac{\kappa\mathrm{r}}{1-\kappa}-\sqrt{\mathrm{B}}\right)
    \label{eq: intertwined configurations}
 \end{equation}
For $S$-type configurations, Proposition \ref{prop: factor for square and GCD} and Corollary \ref{cor: class-2 D, variables and independence} imply that $B_{\alpha\beta}^{ij}$ is squarefree in $\mathbb{C}(\mathbf{t})$. We can write $D=e_{1}^{2}+e_{2}^{2}$ with linearly independent monomials $e_{1},e_{2}$: as in the proof of Theorem \ref{thm: distinguishability terms from Y}, the components of $\overrightarrow{\chi}(\mathcal{I}\mid_{\alpha\beta}^{ij})=:(E_{1},E_{2},E_{3})$ are defined by the constant $c_{S}\in\{1,-1\}$ in Lemma \ref{lem: binomial D} 
and the choice of $p\in [2]$ in the relation $\frac{3+c_{S}}{2}\cdot \Psi(e_{p})=\Psi(E_{u})=\Psi(E_{w})$ for two distinct $u,w\in[3]$. This corresponds to the choices shown in Table \ref{tab: type-S possible configuration sets}. 
\begin{table}
\begin{centering}
\begin{tabular}{|c|c|c|}
\hline 
$\hat{\chi}(\mathcal{I}\mid_{\alpha\beta}^{ij})$ & $c_{S}=-1$ & $c_{S}=1$\tabularnewline
\hline 
\hline 
$p=1$ & $\left\{ \frac{1}{2}e_{1},-\frac{1}{2}e_{1},e_{2}\right\} $ & $\left\{ \frac{1}{4}e_{1}^{2},\frac{1}{4}e_{1}^{2},-e_{2}^{2}\right\} $\tabularnewline
\hline 
$p=2$ & $\left\{ e_{1},\frac{1}{2}e_{2},-\frac{1}{2}e_{2}\right\} $ & $\left\{ -e_{1}^{2},\frac{1}{4}e_{2}^{2},\frac{1}{4}e_{2}^{2}\right\} $\tabularnewline
\hline 
\end{tabular}
\par\end{centering}
\centering{}\caption{\label{tab: type-S possible configuration sets}Possible sets $\hat{\chi}(\mathcal{I}\mid_{\alpha\beta}^{ij})$
that can be derived by $D=e_{1}^{2}+e_{2}^{2}$.}
\end{table}
In conclusion, we can consider the variables $\mathrm{r}:=e_{2}^{-1}e_{1}$ or $\mathrm{r}^{-1}$ in analogy to (\ref{eq: homogeneous with respect to clustering}). In this way, we can unify $G$-type and $S$-type factors at $c_{S}=-1$ noting that the latter configuration is a special case of the former, where $g=\frac{1}{2}$ and $\kappa=-1$ in (\ref{eq: two configurations for proportional}). 
\end{rem}

\section{\label{sec: proof of theorem} Proof of the main result and dimensional bounds} 


This section explores the constrained form of radical terms derived in Theorem \ref{thm: distinguishability terms from Y} together with the relation (\ref{eq: associativity, d}) and the condition in Remark \ref{rem: no Y=-1}. The details of this section improve the qualitative results of Theorem \ref{thm: distinguishability terms from Y} with respect to the dimensional constraints in the presence of radical terms. A deeper analysis of the implications of (\ref{eq: associativity, d}) provides quantitative bounds and leads to a better understanding of the possible monomial combinations. 

\begin{lem}
\label{lem: constant Y-terms in radical triples} 
Let $Y_{\alpha_{1}\alpha_{2}}^{ij}$ and $Y_{\alpha_{1}\alpha_{3}}^{ij}$ be radical, while $Y_{\alpha_{2}\alpha_{3}}^{ij}\in\mathbb{F}$. Then $Y_{\alpha_{2}\alpha_{3}}^{ij}\in\mathbb{C}$ and, for both $w\in\{2,3\}$, $\overrightarrow{\chi}(\mathcal{I}\mid_{\alpha_{1}\alpha_{w}}^{ij})$ contains at least two proportional, but not equal, terms.
\end{lem}
\begin{proof}
Say $\kappa:=-Y_{\alpha_{2}\alpha_{3}}^{ij}\in\mathbb{F}$. Lemma \ref{lem: observable sets with 0 are algebraic} implies that $\chi(\mathcal{I}\mid_{\alpha_{u}\alpha_{w}}^{ij})$ satisfies (\ref{eq: all 3 terms non-vanishing}) for all $1\leq u < w \leq 3$.
The proportionality of $Y_{\alpha_{1}\alpha_{2}}^{ij}$ and $Y_{\alpha_{1}\alpha_{3}}^{ij}$ 
means that the corresponding $B$-terms have the same factors in $\mathbb{C}(\mathbf{t})$, and an appropriate choice of the scaling (\ref{eq: normalised three-term set}) for the elements in $\chi(\mathcal{I}\mid_{\alpha_{1}\alpha_{2}}^{ij})$ lets us fix $\hat{B}_{\alpha_{1}\alpha_{2}}^{ij}=\hat{B}_{\alpha_{1}\alpha_{3}}^{ij}=:\mathrm{B}$ without affecting the associated $Y$-terms. When $\hat{\chi}(\mathcal{I}\mid_{\alpha_{1}\alpha_{2}}^{ij})$ contains three pairwise independent elements, they can be recovered uniquely from $\mathrm{B}$ as discussed in Theorem \ref{thm: distinguishability terms from Y} and Remark \ref{rem: if proportional, then unique constant ratio }. This forces $Y_{\alpha_{1}\alpha_{2}}^{ij}=Y_{\alpha_{1}\alpha_{3}}^{ij}$, in contradiction to Remark \ref{rem: no Y=-1}. Therefore, each of the sets $\chi(\mathcal{I}\mid_{\alpha_{1}\alpha_{2}}^{ij})$ and $\chi(\mathcal{I}\mid_{\alpha_{1}\alpha_{3}}^{ij})$ contains two proportional terms.
A similar argument holds for $S$-type configurations: $A_{\alpha_{1}\alpha_{2}}^{ij}$ and $A_{\alpha_{1}\alpha_{3}}^{ij}$ must have the same factors even under the changes of basis $\mathcal{I}\mapsto\mathcal{I}_{\alpha_{2}}^{i}$ and $\mathcal{I}\mapsto\mathcal{I}_{\alpha_{2}}^{j}$, and they are not invertible in $\mathbb{C}(\mathbf{t})$ for one such choice of basis; from Table \ref{tab: type-S possible configuration sets}, we see that a configuration with $c_{S}=1$ is compatible with these conditions only if $Y_{\alpha_{1}\alpha_{2}}^{ij}=Y_{\alpha_{1}\alpha_{3}}^{ij}$, at odds with Remark \ref{rem: no Y=-1}. 
\end{proof}
\begin{rem}
\label{rem: permutation group representation and scalars}
From Remark \ref{rem: transformations and permutation group}, we get six transformations acting on $\hat{\chi}(\mathcal{I}\mid_{\alpha_{2}\alpha_{3}}^{ij})$, which form a group under composition that is isomorphic to $\mathcal{S}_{3}$. Their action on $\kappa$ is defined by (\ref{eq: Grassmann-Plucker translation exchange, vertical}) and (\ref{eq: Grassmann-Plucker translation exchange, diagonal}), which act as $-\kappa\mapsto\kappa-1$ and $-\kappa\mapsto (\kappa^{-1}-1)^{-1}$, respectively, together with their reciprocals. So we get a finite set of allowed values for $Y_{\alpha_{2}\alpha_{3}}^{ij}$ that is closed under the transformation rules (\ref{eq: Grassmann-Plucker translation exchange, vertical}) and (\ref{eq: Grassmann-Plucker translation exchange, diagonal}), that is, under the action of functions $f_{h}$ and $f_{v}$ introduced in Remark \ref{rem: transformations and permutation group}. When $D$ is of $S$-type, this reduces the possible values for $Y_{\alpha_{2}\alpha_{3}}^{ij}$ to the set $\{-\frac{1}{2},1,-2\}$, in line with allowed expressions in Table \ref{tab: type-S possible configuration sets}. 
\end{rem}

\begin{lem}
\label{lem: concurrence of two classes} The only form for $D$ that is compatible with both classes $\mathrm{I}$ and $\mathrm{II}$ is (\ref{eq: class-2 D first occurrence}). In particular, all radical terms have the same class when $D$ does not have this form. 
\end{lem}
\begin{proof}
Let $D$ be a squarefree polynomial in $\mathbb{C}(\mathbf{t})$ such that there exist three units $e_{u}\in\mathbb{C}(\mathbf{t})$, $u\in[3]$, satisfying $B(e_{1},e_{2},e_{3})=D$. This is a class-$\mathrm{I}$ configuration, and we look for a non-invertible polynomial $Q_{\alpha\beta}^{ij}\in\mathbb{C}(\mathbf{t})$ such that $D\cdot(Q_{\alpha\beta}^{ij})^{2}$ is an allowed class-$\mathrm{II}$ $B$-term. 
We exploit the homogeneity property discussed in Corollary \ref{cor: class-2 D, variables and independence} and the proof of Proposition \ref{prop: class-2, recover hat-B from D} to infer the existence of a unit $\mathrm{r}\in\mathbb{C}(\mathbf{t})$ such that $Q_{\alpha\beta}^{ij}=\mathrm{r}-1$ and $D$ is a univariate polynomial in $\mathrm{r}$, up to units. With appropriate labelling, we can set $\Psi(e_{3})=\mathbf{0}$, $\Psi(e_{2})=q\cdot\Psi(\mathrm{r})$, and $\Psi(e_{1})=p\cdot\Psi(\mathrm{r})$ where $p,q\in\mathbb{N}$ and $q< p$. 

We focus on the cases where (\ref{eq: coincidence of left-gaps}) and (\ref{eq: coincidence of right-gaps}) do not simultaneously hold, otherwise we recover Example \ref{exa: class-2 D first occurrence} as in the proof of Proposition \ref{prop: class-2, recover hat-B from D}. So, $q\in\{1,p-1\}$, say $q=1$ without loss of generality, then we prove that 
$(\mathrm{r}-1)^{2}\cdot D$, where $D=B(\mathrm{r}^{\omega},c_{1}\mathrm{r},c_{2})$ for some $\omega\in\mathbb{N}$, has a sparsity strictly greater than $6$ at $\omega\geq 3$, so it cannot be represented as a $B$-term. This claim is easily verified at $\omega\in\{3,4\}$, where (\ref{eq: coincidence of left-gaps}) or (\ref{eq: coincidence of right-gaps}) are violated only if $c_{1}=-c_{2}$, hence we get 
\begin{equation*}
(\mathrm{r}-1)^{2}\cdot B(\mathrm{r}^{3},c_{1}\mathrm{r},-c_{1})
= c_{1}^{2}-2c_{1}^{2}\mathrm{r}^{2}+2c_{1}\mathrm{r}^{3}+c_{1}(c_{1}-6)\mathrm{r}^{4}+6c_{1}\mathrm{r}^{5}+(1-2c_{1})\mathrm{r}^{6}-2\mathrm{r}^{7}+\mathrm{r}^{8}
\end{equation*}
and 
\begin{equation*} 
(\mathrm{r}-1)^{2}\cdot B(\mathrm{r}^{4},c_{1}\mathrm{r},-c_{1})
= c_{1}^{2}-2c_{1}^{2}\mathrm{r}^{2}+c_{1}(c_{1}+2)\mathrm{r}^{4}-6c_{1}\mathrm{r}^{5}+6c_{1}\mathrm{r}^{6}-2c_{1}\mathrm{r}^{7}+\mathrm{r}^{8}-2\mathrm{r}^{9}+\mathrm{r}^{10}.
\end{equation*}
At most one coefficient can vanish in these expressions, so the sparsity is at least $7$. At $\omega > 4$, we get 
\begin{eqnarray*}
& & (\mathrm{r}-1)^{2}\cdot B(\mathrm{r}^{\omega},c_{1}\mathrm{r},c_{2}) \nonumber \\ 
& = & c_{2}^{2}-2c_{2}(c_{1}+c_{2})\mathrm{r}+(c_{1}^{2}+4c_{1}c_{2}+c_{2}^{2})\mathrm{r}^{2}-2c_{1}(c_{1}+c_{2})\mathrm{r}^{3}+c_{1}^{2}\mathrm{r}^{4}-2c_{2}\mathrm{r}^{\omega}\\
& + & 2(2c_{2}-c_{1})\cdot\mathrm{r}^{\omega+1}+2(2c_{1}-c_{2})\cdot\mathrm{r}^{\omega+2}-2c_{1}\mathrm{r}^{\omega+3}+\mathrm{r}^{2\omega}-2\mathrm{r}^{2\omega+1}+\mathrm{r}^{2\omega+2}
\end{eqnarray*}
which has a sparsity at least $7$. The remaining exponent $\omega=2$ returns $(\mathrm{r}-1)^{2}\cdot B(\mathrm{r}^{2},c_{1}\mathrm{r},c_{2})$, which does not reproduce a $B$-term as is easily checked. Therefore, when $D$ is not of the form (\ref{eq: class-2 D first occurrence}) we can recover the class of any radical term, and all radical terms have the same class. 
\end{proof}
We can now analyse the additional constraints on $Y_{\alpha\beta}^{ij}\in\mathbb{C}\setminus\{0\}$ generated by the existence of radical terms $Y_{\beta\gamma}^{ij}$, $\gamma \in \mathcal{I}^{\mathtt{C}}$. Prior to that, we introduce the notation that will be used in the rest of this section. 
\begin{defn}
\label{def: proportionality position and symbols}
When two components of $\overrightarrow{\chi}(\mathcal{I}\mid_{\alpha\beta}^{ij})$ are proportional, we denote by $\Lambda_{\alpha\beta}^{ij}\in[3]$ the position of the unique independent component in the triple $\overrightarrow{\chi}(\mathcal{I}\mid_{\alpha\beta}^{ij})$ when $Y_{\alpha\beta}^{ij}\notin\mathbb{F}$; we extend this definition setting $\Lambda_{\alpha\beta}^{ij}=0$ to indicate the condition $Y_{\alpha\beta}^{ij}\in\mathbb{C}$. Furthermore, we introduce the symbols
\begin{equation}
 h(\mathcal{I}_{\delta_{0}}^{i})\cdot h(\mathcal{I}_{\delta_{1}}^{j})=:c_{\delta_{0}\delta_{1}}\cdot h(\mathcal{I}_{\delta_{1}}^{i})\cdot h(\mathcal{I}_{\delta_{0}}^{j}),\quad \mathcal{I}_{\delta_{w}}^{m}\in\mathfrak{G}(\mathbf{L}(\mathbf{1})),\,w\in\{0,1\},\,m\in\{i,j\}.
\label{eq: associativity for common positions}
\end{equation}
The consistency with this definition implies $c_{\delta_{1}\delta_{1}}=1$ and $c_{\delta_{1}\delta_{2}}=c_{\delta_{1}\delta_{0}}c_{\delta_{0}\delta_{2}}$ for all indices $\delta_{0},\delta_{1},\delta_{2}$. 
\end{defn}
\begin{prop}
\label{prop: existence proportional but non-equal terms} 
When $n-k\geq5$, for all $\mathcal{I}\in\mathfrak{G}(\mathbf{L}(\mathbf{t}))$, $i,j\in\mathcal{I}$, and $\alpha,\beta\in\mathcal{I}^{\mathtt{C}}$, the triple  $\overrightarrow{\chi}(\mathcal{I}\mid_{\alpha\beta}^{ij})$ returning a radical term $Y(\mathcal{I})_{\alpha\beta}^{ij}$ contains two proportional but not equal components. 
\end{prop}
\begin{proof}
Let $\overrightarrow{\chi}(\mathcal{I}\mid_{\alpha_{1}\alpha_{2}}^{ij})=:(E_{1},E_{2},E_{3})$ violate the thesis, that is, it contains pairwise independent monomials, or two components coincide. Being $Y_{\alpha_{1}\alpha_{2}}^{ij}\notin\mathbb{F}$, this is summarised in the following condition  
 \begin{equation}
\#\Psi(\hat{\chi}(\mathcal{I}\mid_{\alpha_{1}\alpha_{2}}^{ij}))=\#\mathrm{Supp}(\hat{\chi}(\mathcal{I}\mid_{\alpha_{1}\alpha_{2}}^{ij}))>1
\label{eq: concordance phases and minors cardinalities}
 \end{equation}
where, in this case, $\mathrm{Supp}(\hat{\chi}(\mathcal{I}\mid_{\alpha_{1}\alpha_{2}}^{ij}))$ denotes the support, i.e. the underlying set of the multiset $\hat{\chi}(\mathcal{I}\mid_{\alpha_{1}\alpha_{2}}^{ij})$. By contraposition of Lemma \ref{lem: constant Y-terms in radical triples}, (\ref{eq: concordance phases and minors cardinalities}) implies $Y_{\alpha_{1}\alpha_{3}}^{ij},Y_{\alpha_{3}\alpha_{2}}^{ij}\notin\mathbb{C}$ for all $\alpha_{3}\in \mathcal{I}^{\mathtt{C}}$. Together with the assumption $Y_{\alpha_{1}\alpha_{2}}^{ij}\notin\mathbb{F}$, we also infer $Y_{\alpha_{1}\alpha_{3}}^{ij},Y_{\alpha_{3}\alpha_{2}}^{ij}\notin\mathbb{F}$, by Lemma \ref{lem: observable sets with 0 are algebraic}. 

First, we look at configurations where $\hat{\chi}(\mathcal{I}\mid_{\alpha_{u}\alpha_{w}}^{ij})$ is the same for all $u<w$, which implies $\hat{B}_{\alpha_{u}\alpha_{w}}^{ij}=\hat{B}_{\alpha_{1}\alpha_{2}}^{ij}=:\mathrm{B}$. The only configuration that returns $Y_{\alpha_{1}\alpha_{2}}^{ij} \cdot Y_{\alpha_{2}\alpha_{3}}^{ij}\cdot Y_{\alpha_{3}\alpha_{1}}^{ij}\in\mathbb{F}$,
as required by (\ref{eq: associativity, d}), is 
 \begin{equation}
 \frac{E_{1}-E_{2}-E_{3}+\sqrt{\mathrm{B}}}{2E_{x}}\cdot\frac{E_{2}-E_{1}-E_{3} +\sqrt{\mathrm{B}}}{2E_{y}}\cdot\frac{E_{3}-E_{1}-E_{2}+\sqrt{\mathrm{B}}}{2E_{z}} 
= \frac{E_{1}E_{2}E_{3}}{E_{x}E_{y}E_{z}}.
\label{eq: cyclic product}
 \end{equation}
Indeed, at least two $Y$-terms of the form  $Y_{\alpha_{u}\alpha_{w}}^{ij}$, say  $Y_{\alpha_{1}\alpha_{2}}^{ij}$ and $Y_{\alpha_{2}\alpha_{3}}^{ij}$, involve the same square root of $\mathrm{B}$, by Dirichlet's box principle; on the other hand, $A_{\alpha_{1}\alpha_{2}}^{ij}$ and $A_{\alpha_{2}\alpha_{3}}^{ij}$ cannot have the same factors in $\mathbb{C}(\mathbf{t})$, unless they are invertible, otherwise $B_{\alpha_{1}\alpha_{3}}^{ij}$ would not coincide with $\mathrm{B}$ up to units. Under the hypothesis (\ref{eq: concordance phases and minors cardinalities}), this forces the form (\ref{eq: cyclic product}), possibly with $E_{u}=E_{w}$ for some $u,w\in[3]$. The conditions $E_{x}\in\{E_{2},E_{3}\}$, $E_{y}\in\{E_{1},E_{3}\}$,
and $E_{z}\in\{E_{1},E_{2}\}$ imply $\frac{E_{1}E_{2}E_{3}}{E_{x}E_{y}E_{z}}=E_{s}E_{t}^{-1}$ for some $s,t\in[3]$, so this ratio never equals $-1$ when (\ref{eq: concordance phases and minors cardinalities}) holds. 

From the previous argument, we infer that multiple configurations for $\hat{\chi}$-sets
are involved, which allows us to move to univariate polynomials in $\mathbb{C}(\mathrm{r})$, according to Remarks \ref{rem: homogeneous coordinate} and \ref{rem: if proportional, then unique constant ratio }, with an appropriate choice of the unit $\mathrm{r}$. For $G$-type configurations satisfying (\ref{eq: concordance phases and minors cardinalities}), multiple $\hat{\chi}$-sets can arise when both classes $\mathrm{I}$ and $\mathrm{II}$ appear: from Lemma \ref{lem: concurrence of two classes}, this can only be achieved by (\ref{eq: class-2 D first occurrence}), so we choose the form $\overrightarrow{\chi}(\mathcal{I}\mid_{\alpha_{1}\alpha_{2}}^{ij})=(\mathrm{r}^{2},4\mathrm{r},1)$ and get $c_{\alpha_{1}\alpha_{u}}\cdot c_{\alpha_{u}\alpha_{2}}=c_{\alpha_{1}\alpha_{2}}=4\cdot\mathrm{r}$. We infer $\{c_{\alpha_{1}\alpha_{u}},c_{\alpha_{u}\alpha_{2}}\}=\{2\cdot\mathrm{r},2\}$, 
since the alternative $\{\mathrm{r}^{2},4\cdot\mathrm{r}^{-1}\}$ returns (\ref{eq: cyclic product}), and this set can only be generated by 
\begin{equation} 
    \left\{\overrightarrow{\chi}(\mathcal{I}\mid_{\alpha_{1}\alpha_{3}}^{ij}),\overrightarrow{\chi}(\mathcal{I}\mid_{\alpha_{3}\alpha_{2}}^{ij})\right\}=\left\{(\mathrm{r},2\cdot\mathrm{r},1),(\mathrm{r},2,1)\right\}. 
    \label{eq: G-type extremal combination}
\end{equation}
since $\left\{(\mathrm{r},2\cdot\mathrm{r},1),(1,2\cdot \mathrm{r},\mathrm{r})\right\}$ returns (\ref{eq: cyclic product}) as well. 
For $S$-type configurations, we choose $\overrightarrow{\chi}(\mathcal{I}\mid_{\alpha_{1}\alpha_{2}}^{ij})$ $=(-\frac{1}{4},\mathrm{r}^{2},-\frac{1}{4})$, so we have $c_{\alpha_{1}\alpha_{2}}=-4\cdot\mathrm{r}^{2}$: excluding $\{c_{\alpha_{1}\alpha_{3}},c_{\alpha_{3}\alpha_{2}}\}=\{-4\mathrm{r^{2}},1\}$, which returns (\ref{eq: cyclic product}), we infer $\{c_{\alpha_{1}\alpha_{3}},c_{\alpha_{3}\alpha_{2}}\}=\{2\cdot\mathrm{r},-2\cdot\mathrm{r}\}$. These values are generated by 
\begin{equation} 
    \left\{\overrightarrow{\chi}(\mathcal{I}\mid_{\alpha_{1}\alpha_{3}}^{ij}),\overrightarrow{\chi}(\mathcal{I}\mid_{\alpha_{3}\alpha_{2}}^{ij})\right\}=\left\{\left(\frac{1}{2},\mathrm{r},-\frac{1}{2}\right),\left(-\frac{1}{2},\mathrm{r},\frac{1}{2}\right)\right\}.
    \label{eq: S-type extremal combination}
\end{equation}
In both cases, at $n-k\geq 5$ we can invoke Dirichlet's box principle to identify two indices $\gamma_{1},\gamma_{2}\in (\mathcal{I}_{\alpha_{1}\alpha_{2}})^{\mathtt{C}}$ such that $\overrightarrow{\chi}(\mathcal{I}\mid_{\alpha_{1}\gamma_{1}}^{ij})=\overrightarrow{\chi}(\mathcal{I}\mid_{\alpha_{1}\gamma_{2}}^{ij})$; the only possibility to make them different, according to Remark \ref{rem: no Y=-1}, is that $h(\mathcal{I}_{\gamma_{2}}^{i})\cdot Y_{\alpha_{1}\gamma_{1}}^{ij}$ and $h(\mathcal{I}_{\gamma_{2}}^{i})\cdot Y_{\alpha_{1}\gamma_{2}}^{ij}$ are conjugate over $\mathbb{F}$, hence 
\begin{equation*} 
Y_{\gamma_{1}\gamma_{2}}^{ij} = -Y_{\gamma_{1}\alpha_{1}}^{ij}\cdot Y_{\alpha_{1}\gamma_{2}}^{ij} = -\frac{h(\mathcal{I}_{\gamma_{2}}^{i})\cdot h(\mathcal{I}_{\alpha_{1}}^{j})}{h(\mathcal{I}_{\gamma_{2}}^{j})\cdot h(\mathcal{I}_{\alpha_{1}}^{i})} \cdot (Y_{\alpha_{1}\gamma_{2}}^{ij})^2.
\end{equation*} 
This expression for $Y_{\gamma_{1}\gamma_{2}}^{ij}$ is not compatible with $Y_{\alpha_{1}\gamma_{2}}^{ij}$ as derived from one of the triples in (\ref{eq: G-type extremal combination}) or (\ref{eq: S-type extremal combination}), since $B_{\gamma_{1}\gamma_{2}}^{ij}$ has neither the same factors as $D$, nor it is consistent with (\ref{eq: class-2 B first occurrence}). Thus, $n-k\leq 4$. 
\end{proof}
Configurations allowing both class-$\mathrm{I}$ and class-$\mathrm{II}$ terms exist at $n-k\leq4$, e.g. 
\begin{equation*}
    \frac{-\mathrm{r}-1+\sqrt{D}}{2}\cdot\frac{\mathrm{r}-3-\sqrt{D}}{2}=-\frac{\mathrm{r}^{2}-4\mathrm{r}-1-(\mathrm{r}-1)\cdot\sqrt{D}}{2}.
\end{equation*}
\begin{cor}
\label{cor: only one constant factor per Y-term} For each radical $Y_{\alpha\beta}^{ij}$, there is at most one index $\gamma$ with $Y_{\beta\gamma}^{ij}\in\mathbb{C}$. 
\end{cor}
\begin{proof}
Let $Y_{\alpha_{1}\alpha_{2}}^{ij}$ be radical and take $\alpha_{3}$ such that $\kappa:=-Y_{\alpha_{2}\alpha_{3}}^{ij}$ is not radical. All sets $\chi(\mathcal{I}\mid_{\alpha_{u}\alpha_{w}}^{ij})$, $1\leq u<w\leq 3$, satisfy (\ref{eq: all 3 terms non-vanishing}) by Lemma \ref{lem: observable sets with 0 are algebraic}, then $\kappa\in\mathbb{C}$ by Lemma \ref{lem: case B^(ij)_(ab) square}, and we can find two proportional terms in $\hat{\chi}(\mathcal{I}\mid_{\alpha_{1}\alpha_{2}}^{ij})$ by Lemma \ref{lem: constant Y-terms in radical triples}.
Recalling the definitions of central and unique terms provided in Section \ref{subsec: Notation}, we can assume, moving to a basis $\mathcal{J}\in\{\mathcal{I},\mathcal{I}_{\alpha_{2}}^{i},\mathcal{I}_{\alpha_{2}}^{j}\}$ if necessary, that the central term in $\hat{\chi}(\mathcal{J}\mid_{\alpha_{1}\alpha_{2}}^{ij})$ is unique while preserving the relations $Y_{\alpha_{1}\alpha_{2}}^{ij}\notin \mathbb{F}$ and $Y_{\alpha_{2}\alpha_{3}}^{ij}\in\mathbb{C}$. With the notation (\ref{eq: associativity for common positions}), 
this means $c_{\alpha_{1}\alpha_{2}}\in\mathbb{C}$, while $c_{\alpha_{2}\alpha_{3}}\in\mathbb{C}$ by Lemma \ref{lem: case B^(ij)_(ab) square} applied to $\chi(\mathcal{I}\mid_{\alpha_{2}\alpha_{3}}^{ij})$; therefore, $c_{\alpha_{1}\alpha_{3}}=c_{\alpha_{1}\alpha_{2}}\cdot c_{\alpha_{2}\alpha_{3}}\in\mathbb{C}$ and, being $Y_{\alpha_{1}\alpha_{3}}^{ij}\notin\mathbb{F}$, we find that the central term in $\hat{\chi}(\mathcal{I}\mid_{\alpha_{1}\alpha_{3}}^{ij})$ is also unique, letting us state  
 \begin{equation}
    \Lambda_{\alpha_{1}\alpha_{2}}=1,\; Y_{\alpha_{2}\alpha_{3}}\in\mathbb{C} \quad \Rightarrow \quad 
    \Lambda_{\alpha_{1}\alpha_{3}}^{ij}=1. 
    \label{eq: proportionality preserves centering}
 \end{equation}
This generates two possible $G$-type configurations, namely the two sets $\mathrm{C}_{1}$ and $\mathrm{C}_{2}$ in (\ref{eq: two configurations for proportional}). The $A$-terms (\ref{eq: expression cross-section quadratic case, a}) corresponding to these two sets, according to the position of the central terms discussed above, are $\mathrm{r}-(\kappa+1)g$ and $(1+\kappa)\cdot\mathrm{r}-(1-\kappa)^{2}\cdot g$: being $\kappa \neq -1$ for $G$-type configurations, these polynomials do not have the same factors in $\mathbb{C}(\mathbf{t})$. Therefore, $\hat{\chi}(\mathcal{I}\mid_{\alpha_{1}\alpha_{2}}^{ij})=\hat{\chi}(\mathcal{I}\mid_{\alpha_{1}\alpha_{3}}^{ij})$ and there cannot be another index $\alpha_{4}\neq \alpha_{3}$ such that $Y_{\alpha_{2}\alpha_{4}}^{ij}\in\mathbb{C}$, as it would entail $Y_{\alpha_{2}\alpha_{3}}^{ij}=Y_{\alpha_{2}\alpha_{4}}^{ij}$  contradicting Remark \ref{rem: no Y=-1}. For $S$-type configurations, the possible values for $Y_{\alpha_{2}\alpha_{3}}^{ij}$ are in $\{-\frac{1}{2},1,-2\}$ by Remark \ref{rem: permutation group representation and scalars}. The existence of two distinct indices $\alpha_{3},\alpha_{4}$ such that $Y_{\alpha_{2}\alpha_{3}}^{ij},Y_{\alpha_{2}\alpha_{4}}^{ij}\in\mathbb{C}$ implies $Y_{\alpha_{3}\alpha_{4}}^{ij}\in\mathbb{C}$, and hence $Y_{\alpha_{2}\alpha_{3}}^{ij},Y_{\alpha_{2}\alpha_{4}}^{ij},Y_{\alpha_{3}\alpha_{4}}^{ij}\in\{-\frac{1}{2},1,-2\}$.
But $Y_{\alpha_{3}\alpha_{4}}^{ij}=-Y_{\alpha_{3}\alpha_{2}}^{ij}Y_{\alpha_{2}\alpha_{4}}^{ij}$ and there exist no elements $a,b,c\in\{-\frac{1}{2},1,-2\}$ such that $ab=-c$. 
\end{proof}

\begin{prop}
\label{prop: bound number of radicals}
If $\max\{k,n-k\}\geq 5$, no observable set generates a radical $Y$-term. 
\end{prop}
\begin{proof}
We can assume $n-k\geq5$ without loss of generality, as noted in Remark \ref{rem: dual Grassmannian representation}. By Proposition \ref{prop: existence proportional but non-equal terms}, all observable sets contain two proportional monomials with ratio $\kappa\in\mathbb{C}\setminus\{0,1\}$, so there are no class-$\mathrm{II}$ terms; taking into account Remark \ref{rem: permutation group representation and scalars}, thus including $S$-type configurations as $\kappa=c_{S}=-1$, we find two non-invertible polynomials $F_{(+1)},F_{(-1)}\in\mathbb{C}[\mathbf{r}]$ defined at $u\in\{+1,-1\}$ and $\Lambda_{\alpha\beta}^{ij}\in\{2,3\}$ by 
\begin{equation}
   \hat{A}_{\alpha\beta}^{ij}=F_{(u)}\Rightarrow \left\{\frac{h(\mathcal{I}_{\alpha\beta}^{ij})\cdot h(\mathcal{I})}{h(\mathcal{I}_{\alpha}^{i})\cdot h(\mathcal{I}_{\beta}^{j})},\frac{h(\mathcal{I}_{\alpha\beta}^{ij})\cdot h(\mathcal{I})}{h(\mathcal{I}_{\alpha}^{j})\cdot h(\mathcal{I}_{\beta}^{i})}\right \}\cap\mathbb{C}=\{\kappa^{u}\}
\end{equation}
such that $\hat{A}_{\alpha\beta}^{ij}\in\{F_{(+1)},F_{(-1)}\}$ whether $\Lambda_{\alpha\beta}^{ij}\notin\{0,1\}$. 
These polynomials determine the behaviour of the corresponding $Y$-terms: for both $u\in\{+1,-1\}$, $\hat{A}_{\alpha\beta}^{ij}=F_{(u)}$ and $Y_{\alpha\gamma}^{ij}\in\mathbb{C}$ imply $Y_{\alpha\gamma}^{ij}=(\kappa^{u}-1)^{\sigma}$ for some $\sigma\in\{1,-1\}$. The coefficients introduced in (\ref{eq: associativity for common positions}) for such polynomials are $c_{\alpha\beta}\in\{c_{(+1)},c_{(-1)}\}$ where 
\begin{eqnarray}
c_{(+1)} & := & \mathrm{r}^{\sigma_{+1,1}}\cdot(1-\kappa)^{-\sigma_{+1,1}\cdot(1+\sigma_{+1,2})},\nonumber \\
c_{(-1)} & := & (\mathrm{r}\cdot\kappa^{\sigma_{-1,2}})^{\sigma_{-1,1}}\cdot(1-\kappa)^{-\sigma_{-1,1}\cdot(1+\sigma_{-1,2})} 
\label{eq: coefficients for non-centered}
\end{eqnarray}
for some $\sigma_{u,1},\sigma_{u,2}\in\{1,-1\}$, $u\in\{+1,-1\}$, that depend on the set configuration ($\mathrm{C}_{1}$ or $\mathrm{C}_{2}$) and $\Lambda_{\alpha\beta}^{ij}$ (in $\{2,3\}$). 

Now, we fix any radical term and, choosing a basis $\mathcal{J}\in\{\mathcal{I},\mathcal{I}_{\alpha_{1}}^{i},\mathcal{I}_{\alpha_{1}}^{j}\}$, we assume $\Lambda_{\alpha_{1}\alpha_{2}}^{ij}=1$. When $Y_{\alpha_{1}\gamma}^{ij}$ and $Y_{\gamma\alpha_{2}}^{ij}$ are both radical, $\gamma\in\mathcal{I}^{\mathtt{C}}$, the conditions $c_{\alpha_{1}\alpha_{2}}\in\{\kappa,\kappa^{-1}\}$ (by Remark \ref{rem: if proportional, then unique constant ratio }), $\kappa\neq 1$ (by Proposition \ref{prop: existence proportional but non-equal terms}), and $c_{\alpha_{2}\gamma}=c_{\alpha_{2}\alpha_{1}}\cdot c_{\alpha_{1}\gamma}$, together with (\ref{eq: radical expression}) and (\ref{eq: associativity, d}) in the case $\kappa^{3}=1$, require at least one $u\in[2]$ such that $c_{\alpha_{u}\gamma}\notin\{\kappa,\kappa^{-1}\}$, hence $\Lambda_{\alpha_{u}\gamma}^{ij}\neq 1$. Being $\Lambda_{\alpha_{1}\gamma}^{ij}=1$ and $\Lambda_{\alpha_{2}\gamma}^{ij}=1$ equivalent when $\Lambda_{\alpha_{1}\alpha_{2}}^{ij}=1$, we conclude $\Lambda_{\alpha_{1}\gamma}^{ij},\Lambda_{\gamma\alpha_{2}}^{ij}\in\{2,3\}$. On the other hand, $A_{\alpha_{1}\gamma}^{ij}$ and $A_{\gamma\alpha_{2}}^{ij}$ do not share the same factors, otherwise the relation $Y_{\alpha_{1}\alpha_{2}}^{ij}=-Y_{\alpha_{1}\gamma}^{ij}\cdot Y_{\gamma\alpha_{2}}^{ij}$ would generate different factors in $B_{\alpha_{1}\alpha_{2}}^{ij}$ and $B_{\alpha_{1}\gamma}^{ij}$. This argument forces 
\begin{equation}
\left\{\hat{A}_{\alpha_{1}\gamma}^{ij},\hat{A}_{\gamma\alpha_{2}}^{ij}\right\}=\left\{F_{(+1)},F_{(-1)}\right\}.
\label{eq: covering F-factors}
\end{equation}  
Considering (\ref{eq: coefficients for non-centered}), this condition allows us to express $c_{\alpha_{1}\alpha_{2}}=c_{\alpha_{1}\gamma}\cdot c_{\gamma\alpha_{2}}=c_{(+1)}\cdot c_{(-1)}$, where  $c_{\alpha_{1}\alpha_{2}}\in\mathbb{C}$ entails   $\sigma_{+1,1}=-\sigma_{-1,1}$, so the previous relation is equivalent to   
\begin{equation}
(1-\kappa)^{\sigma_{+1,2}-\sigma_{-1,2}}=\kappa^{\sigma_{-1,1}-\sigma_{-1,2}}.
\label{eq: mixed class resonances}
\end{equation}
This relation is trivially satisfied only if $\hat{\chi}(\mathcal{I}\mid_{\alpha_{1}\alpha_{2}}^{ij})\neq\hat{\chi}(\mathcal{I}\mid _{\alpha_{u}\gamma}^{ij})$ for both $u\in[2]$, i.e. they do not share the same set configuration (\ref{eq: two configurations for proportional}) and an identity analogous to (\ref{eq: intertwined configurations}) holds. Otherwise, (\ref{eq: mixed class resonances}) restricts the allowed values of $\kappa$, that is, $\kappa\in\{\frac{1}{2},-1,2\}$, noting that $\kappa=1$ 
is excluded by Proposition \ref{prop: existence proportional but non-equal terms}. In particular, this constraint is required by the following implications
\begin{equation}
\gamma_{1},\gamma_{2}\in\mathcal{I}^{\mathtt{C}},\,\hat{A}_{\alpha_{1}\gamma_{1}}^{ij}=\hat{A}_{\alpha_{1}\gamma_{2}}^{ij}\in\left\{F_{(+1)},F_{(-1)}\right\}\; 
\Rightarrow \; Y_{\gamma_{1}\gamma_{2}}^{ij}\in\mathbb{F} \; \Rightarrow \; Y_{\gamma_{1}\gamma_{2}}\in\mathbb{C}
\label{eq: proportionality implications} 
\end{equation}
by Lemmas \ref{lem: observable sets with 0 are algebraic} and \ref{lem: case B^(ij)_(ab) square}; the last condition is achievable only if $\hat{\chi}(\mathcal{I}\mid_{\alpha_{1}\gamma_{1}}^{ij})\neq \hat{\chi}(\mathcal{I}\mid_{\alpha_{1}\gamma_{2}}^{ij})$ at $\Lambda_{\alpha_{1}\gamma_{1}}^{ij}\neq 1$, forcing $\hat{\chi}(\mathcal{I}\mid_{\alpha_{1}\alpha_{2}}^{ij})=\hat{\chi}(\mathcal{I}\mid_{\alpha_{1}\gamma_{u}}^{ij})$ for some $u\in[2]$, and thus requiring $\kappa\in\{\frac{1}{2},-1,2\}$ as above. On the other hand, the proportionality constant $Y_{\gamma_{1}\gamma_{2}}^{ij}$ must belong to $\{\kappa^{u}-1,(\kappa^{u}-1)^{-1}\}$ for both $u\in\{-1,1\}$, as it is accessible from both terms in $\{\hat{A}_{\alpha_{1}\gamma_{1}},\hat{A}_{\alpha_{2}\gamma_{1}}\}=\{F_{(+1)},F_{(-1)}\}$. Therefore, $-\kappa=(1-\kappa)^{\sigma+1}$ for some $\sigma \in \{1,-1\}$, which is not verified at $\kappa\in\{\frac{1}{2},2\}$, while $\kappa=-1$ yields, after a basis exchange, a radical $Y$-term with two equal components in the associated $\chi$, contradicting Proposition \ref{prop: existence proportional but non-equal terms}. 

The previous argument returns at most two indices $\gamma_{1},\gamma_{2}$ such that $Y_{\alpha_{1}\gamma_{1}}^{ij}$ and $Y_{\alpha_{1}\gamma_{2}}^{ij}$ are radical, specifically $\{\hat{A}_{\alpha_{u}\gamma_{1}}^{ij},\hat{A}_{\alpha_{u},\gamma_{2}}^{ij}\}=\left\{F_{(+1)},F_{(-1)}\right\}$ for both $u\in[2]$, and at most two indices $\omega_{1},\omega_{2}$ exist such that $Y_{\alpha_{1}\omega_{1}}^{ij},Y_{\omega_{2},\alpha_{2}}^{ij}\in\mathbb{C}$, according to Corollary \ref{cor: only one constant factor per Y-term}. However, by the same argument, any three such indices are mutually exclusive: for any $\gamma_{1}$ with $\hat{A}_{\alpha_{1}\gamma_{1}}^{ij}=F_{(u)}$, $u\in\{+1,-1\}$, and $\omega_{1},\omega_{2}$ as above, we infer the existence of $\sigma_{1},\sigma_{2}\in\{1,-1\}$ such that
\begin{eqnarray}
Y_{\alpha_{1}\omega_{1}}^{ij} & \in & \{-\kappa,-\kappa^{-1}\}\cap\{\kappa^{u}-1,(\kappa^{u}-1)^{-1}\} \Rightarrow  \kappa=(1-\kappa^{u})^{\sigma_{1}},\nonumber \\
Y_{\omega_{2}\alpha_{2}}^{ij} & \in & \{-\kappa,-\kappa^{-1}\}\cap\{\kappa^{-u}-1,(\kappa^{-u}-1)^{-1}\} \Rightarrow \kappa=(1-\kappa^{-u})^{\sigma_{2}}.
\label{eq: linking chain positions}
\end{eqnarray}
These relations are not compatible with the condition $\kappa\in\{\frac{1}{2},-1,2\}$ obtained from $\hat{\chi}(\mathcal{I}\mid_{\alpha_{1}\alpha_{2}}^{ij})=\hat{\chi}(\mathcal{I}\mid_{\alpha_{1}\gamma}^{ij})$ or $\hat{\chi}(\mathcal{I}\mid_{\omega_{1}\alpha_{2}}^{ij})=\hat{\chi}(\mathcal{I}\mid_{\omega_{1}\gamma}^{ij})$, one of which follows from the conjunction of 
\begin{eqnarray} 
\Lambda_{\alpha_{1}\alpha_{2}}^{ij}=\Lambda_{\omega_{1}\alpha_{2}}^{ij}=1,\; Y_{\alpha_{1}\omega_{1}}^{ij}\in\mathbb{C}  & \Rightarrow & \hat{\chi}(\mathcal{I}\mid_{\alpha_{1}\alpha_{2}}^{ij})=\hat{\chi}(\mathcal{I}\mid_{\omega_{1}\alpha_{2}}^{ij}),\nonumber \\ 
\hat{A}_{\alpha_{1}\gamma}^{ij}=\hat{A}_{\omega_{1}\gamma}^{ij}\in\{F_{(+1)},F_{(-1)}\},\; Y_{\alpha_{1}\omega_{1}}^{ij}\in\mathbb{C} & \Rightarrow &
\hat{\chi}(\mathcal{I}\mid_{\alpha_{1}\gamma}^{ij})\neq\hat{\chi}(\mathcal{I}\mid_{\omega_{1}\gamma}^{ij}).
\label{eq: splitting chain positions}
\end{eqnarray}
The same reasoning applies by considering a single $Y_{\alpha_{1}\omega_{1}}^{ij}\in\mathbb{C}$ but two terms in $\{\hat{A}_{\alpha_{1}\gamma_{1}}^{ij},\hat{A}_{\alpha_{1},\gamma_{2}}^{ij}\}=\{F_{(+1)},F_{(-1)}\}$, since we recover (\ref{eq: linking chain positions}) by connecting all terms $\hat{A}_{\alpha_{1}\alpha_{2}}^{ij}$, $\hat{A}_{\alpha_{1}\gamma_{1}}^{ij}$, and $\hat{A}_{\alpha_{1}\gamma_{2}}^{ij}$ to the same constant $Y_{\alpha_{1}\omega_{1}}^{ij}$, and (\ref{eq: splitting chain positions}) with $\gamma=\gamma_{1}$. Thus, we find at most two indices in $\mathcal{I}^{\mathtt{C}}$ besides $\alpha_{1}$ and $\alpha_{2}$. 
\end{proof}
This bound represents minimal information required to recover the consistency conditions forcing combinatorial reduction, under Assumption (\ref{eq: all 3 terms non-vanishing}), as shown in the following example. 
\begin{example}
\label{exa: counterexample limited information, dimensionality}
When $n-k=4$ we can obtain radical terms through the configuration $\mathbf{L}_{\text{ex}}=\mathbf{E}_{-1}$ and $\mathbf{R}_{\text{ex}}=\mathbf{E}_{1}^{\mathtt{T}}$ with  
\begin{eqnarray*}
\mathbf{E}_{\varepsilon} & := & 
{\small\left(\begin{array}{cccccc}
1 & 0 & 1 & 1 & 1 & 1\\
0 & 1 & 1 & \mathrm{r}+\varepsilon\cdot\sqrt{\mathrm{r}^{2}+1} & -\mathrm{r}-\varepsilon\cdot\sqrt{\mathrm{r}^{2}+1} & -1
\end{array}\right)}.
\end{eqnarray*}
It is easily checked by direct computation that all the minors of $\mathbf{R}_{\text{ex}}$ are non-vanishing, the products
$\Delta_{\mathbf{L}_{\text{ex}}}(\mathcal{I}) \cdot \Delta_{\mathbf{R}_{\text{ex}}}(\mathcal{I})$, $\mathcal{I}\in \wp_{2}[6]$, are monomials 
in $\mathbb{C}(\mathrm{r})$, but (\ref{eq: consistency and curvature}) does not vanish for every choice of indices.
\end{example}

Finally, we can summarise the previous results to prove Theorem \ref{thm: rigidity theorem monomial}. 
\begin{proof}
Call $\alpha_{1},\alpha_{2}$ two generic columns of $\mathbf{L}(\mathbf{1})$, whose existence is stated in Assumption \ref{claim: generic matrices}, and fix $i\in\mathcal{I}$. Take any $(m,\omega)\in\mathcal{I}\times\mathcal{I}^{\mathtt{C}}$: we assume that $\mathbf{L}(\mathbf{1})$ does not have null columns, since they do not enter the deformation (\ref{eq: monomial terms of Cauchy-Binet expansion}) of the determinantal expansion, so we can find $p\in\mathcal{I}$ satisfying $h(\mathcal{I}_{\omega}^{p})\neq0$. We choose a basis $\mathcal{J}\in\{\mathcal{I},\mathcal{I}_{\alpha_{2}}^{p}\}$ such that $h(\mathcal{J}_{\omega}^{m})\neq0$, setting $\mathcal{J}=\mathcal{I}$ at $m=p$; then, define the singletons $\{\overline{p}\}:=\{p,\alpha_{2}\}\cap\mathcal{J}$ and $\{\overline{\alpha}_{2}\}:=\{p,\alpha_{2}\}\cap\mathcal{J}^{\mathtt{C}}$, so that $\overline{\alpha}_{1}:=\alpha_{1}$ and $\overline{\alpha}_{2}$ are generic columns for $\mathcal{J}$, in particular $0\notin\chi(\mathcal{J}\mid_{\overline{\alpha}_{1}\overline{\alpha}_{2}}^{m\overline{p}})$ at $m\neq p$. By Proposition \ref{prop: bound number of radicals} and Lemma \ref{lem: case B^(ij)_(ab) square}, this means $Y(\mathcal{J})_{\overline{\alpha}_{1}\overline{\alpha}_{2}}^{m\overline{p}}\in\mathbb{C}$. We can also find $u\in[2]$ such that $h(\mathcal{I}_{\alpha_{u}\omega}^{mp})\neq 0$ when $m\neq p$, which makes $0\notin\chi(\mathcal{J}\mid_{\overline{\alpha}_{u}\omega}^{m\overline{p}})$; as before, from Proposition \ref{prop: bound number of radicals} and Lemma \ref{lem: case B^(ij)_(ab) square} we can infer $Y(\mathcal{J})_{\overline{\alpha}_{u}\omega}^{m\overline{p}}\in\mathbb{C}$. This returns $Y(\mathcal{J})_{\omega\overline{\alpha}_{2}}^{m\overline{p}}\in\mathbb{C}$ by (\ref{eq: associativity, d}). If $\mathcal{I}\neq\mathcal{J}$, we go back to $\mathcal{I}$ using (\ref{eq: Grassmann-Plucker translation exchange, vertical}) to get $Y_{\alpha_{1}\alpha_{2}}^{mp},Y_{\omega\alpha_{2}}^{mp}\in\mathbb{C}$ and, again using (\ref{eq: associativity, d}), $Y_{\alpha_{1}\omega}^{mp}\in\mathbb{C}$. We can repeat this procedure with $i$ instead of $m$, getting $Y_{\alpha_{1}\omega}^{ip}\in\mathbb{C}$, so we conclude $Y_{\alpha_{1}\omega}^{im}=-Y_{\alpha_{1}\omega}^{ip}\cdot Y_{\alpha_{1}\omega}^{pm}\in\mathbb{C}$. 

In this configuration, we exploit the invariance (\ref{eq: generalized permutation matrix gauge}) to obtain the form (\ref{eq: action of special gauge on R}) with the entries $Y_{\alpha_{1}\omega}^{im}$ in (\ref{eq: entries of R in special gauge}), finding $\tilde{\mathbf{R}}(\mathbf{t}):=\mathbf{D}(\mathbf{t})\cdot \mathbf{R}(\mathbf{t})\cdot \mathbf{d}(\mathbf{t})\in\mathbb{C}^{n\times k}$. Condition (\ref{eq: monomial terms of Cauchy-Binet expansion}) involving $\tilde{\mathbf{L}}(\mathbf{t}):=\mathbf{d}(\mathbf{t})^{-1}\cdot \mathbf{L}(\mathbf{t})\cdot \mathbf{D}(\mathbf{t})^{-1}$ and $\tilde{\mathbf{R}}(\mathbf{t})$ forces all minors $\Delta_{\tilde{\mathbf{L}}(\mathbf{t})}(\mathcal{I})$, with $\mathcal{I}\in\mathfrak{G}(\mathbf{L}(\mathbf{1}))$, to be monomials. Setting 
\begin{equation}
\chi_{\mathcal{J}}(j;\beta):=\Psi\left(\Delta_{\tilde{\mathbf{L}}(\mathbf{t})}(\mathcal{J})^{-1}\cdot \Delta_{\tilde{\mathbf{L}}(\mathbf{t})}(\mathcal{J}_{\beta}^{j})\right),\quad j\in \mathcal{J},\beta\in\mathcal{J}^{\mathtt{C}}
\label{eq: quantification distinguishable columns}
 \end{equation} 
whenever $h(\mathcal{J}_{\beta}^{j})\neq0$, we easily prove that $\chi_{\mathcal{J}}(j;\beta)$ does not depend on $\mathcal{J}$ because, after row reduction, the $2\times2$ minors with $\alpha_{1},\alpha_{2}$ force every column to be constant up to monomial row and column scalings. Alternatively, whenever the following chain preserves $h((\mathcal J_u)_\beta^j)\ne0$, taking two bases $\mathcal{J}_{0},\mathcal{J}_{r}$ satisfying $j\in \mathcal{J}_{u},\beta\in\mathcal{J}_{u}^{\mathtt{C}}$ for both $u\in\{0,r\}$ with $r:=\#(\mathcal{J}_{0}\setminus\mathcal{J}_{r})$, we introduce a finite sequence of bases $\mathcal{J}_{u}$ as in (\ref{eq: chain of non-vanishing}), $u\in[r-1]$, noting that $j,\beta\notin\{m_{1},\dots,m_{r}\}\cup\{\delta_{1},\dots,\delta_{r}\}$. Then, for each $u\in[r]$, the relation $\chi_{\mathcal{J}_{u-1}}(j;\beta)=\chi_{\mathcal{J}_{u}}(j;\beta)$ is entailed by the three-term Grassmann-Pl\"{u}cker relation (\ref{eq: Grassmann-Plucker including permutations}), since a non-trivial linear combination over $\mathbb{C}$ of at most three units in $\mathbb{C}(\mathbf{t})$ can vanish only if they are pairwise proportional. Iterating this argument for each $u\in[r]$, we find that $\chi_{\mathcal{J}}(j;\beta)$ does not depend on $\mathcal{J}$.  

Finally, we fix the basis $\mathcal{I}$ and construct the function $\psi$ in (\ref{eq: set-to-element integrability}) setting $\psi(m):=-\chi_{\mathcal{I}}(m,\alpha_{1})$ for all $m\in\mathcal{I}$ and $\psi(\omega):=\chi_{\mathcal{I}}(m_{\omega};\omega)+\psi(m_{\omega})$ for any $\omega\in\mathcal{I}^{\mathtt{C}}$, where $m_{\omega}\in\mathcal{I}$ satisfies $h(\mathcal{I}_{\omega}^{m_{\omega}})\neq0$. Proceeding as in the proof in \cite[Th.15]{Angelelli2019}, we get the thesis. 
\end{proof}

\section{Combinatorial reduction from set permutations to element permutations}
\label{sec: Reduction of multi-valued maps and applications to permutation coding}

An implication of the previous results concerns the extraction of information from a given combinatorial structure: specifically, we can embed the permutations of $[n]$ into the permutations of $\mathfrak{G}(\mathbf{L}(\mathbf{1}))\subseteq \wp_{k}[n]$, which induce a relabelling of the collection $\{\Delta_{\mathbf{L}}(\mathcal{I}),\,\mathcal{I}\in\mathfrak{G}(\mathbf{L}(\mathbf{1}))\}$. This lets us take advantage of the apparent complexity arising from the embedding, while still allowing us to recover the original permutation and matrices that provide the decomposition.

We adapt each permutation $\widehat{\Psi}:\,\mathfrak{G}(\mathbf{L}(\mathbf{1}))\longrightarrow\mathfrak{G}(\mathbf{L}(\mathbf{1}))$ to our formalism through the composition 
 \begin{equation}
\Psi:=\mathtt{Ind}\circ\widehat{\Psi}:\,\mathfrak{G}(\mathbf{L}(\mathbf{1}))\longrightarrow\{0,1\}^{n}\label{eq: indicator set mapping}
 \end{equation}
where $\mathtt{Ind}$ is the map that associates each set with its indicator function. This defines a set function with values in $\mathbb{Z}^{n}$. With these premises, we have the following consequence of Theorem \ref{thm: rigidity theorem monomial}: 
\begin{cor}
\label{cor: permutation coding and reductions} Let $\widehat{\Psi}:\,\mathfrak{G}\longrightarrow\mathfrak{G}$ be a permutation of a matroid $\mathfrak{G}$ with two generic columns, that is, satisfying Assumption \ref{claim: generic matrices}. Using the correspondence (\ref{eq: indicator set mapping}), for any mapping $g:\,\mathfrak{G}\longrightarrow\mathbb{C}$ there exist two $(k\times n)$-dimensional matrices $\mathbf{A}(\mathbf{t}),\mathbf{Q}^{\mathtt{T}}(\mathbf{t})$ satisfying $\max\{k,n-k\}\geq 5$, $\mathfrak{G}=\mathfrak{G}(\mathbf{A}(\mathbf{1}))$, and 
 \begin{equation}
\Delta_{\mathbf{A}(\mathbf{t})}(\mathcal{I})\cdot\Delta_{\mathbf{Q}(\mathbf{t})}(\mathcal{I})=g(\mathcal{I})\cdot\mathbf{t}^{\Psi(\mathcal{I})},\quad\mathcal{I}\in\mathfrak{G}
\label{eq: terms of Cauchy-Binet expansion for permutation coding}
 \end{equation}
if and only if there exist matrices $\mathbf{a},\mathbf{q}^{\mathtt{T}}\in\mathbb{C}^{k\times n}$ and a permutation $\psi\in\mathcal{S}_{n}$ such that the pairs $(\mathbf{A}(\mathbf{t}),\mathbf{Q}(\mathbf{t}))$ and $(\mathbf{a},\mathrm{diag}(t_{\psi(1)},\dots,t_{\psi(n)})\cdot\mathbf{q})$ have the same Cauchy-Binet expansion, i.e. 
 \begin{equation}
\Delta_{\mathbf{A}(\mathbf{t})}(\mathcal{I})\cdot\Delta_{\mathbf{Q}(\mathbf{t})}(\mathcal{I})=\Delta_{\mathbf{a}}(\mathcal{I})\cdot\Delta_{\mathbf{q}}(\mathcal{I})\cdot\prod_{\alpha\in\mathcal{I}}t_{\psi(\alpha)},\quad\mathcal{I}\in\wp_{k}[n].
\label{eq: integrability for permutation coding}
 \end{equation}
\end{cor}
Thereafter we will use the same symbol for $\Psi$ and $\widehat{\Psi}$ in (\ref{eq: indicator set mapping}) with a slight abuse of notation. 

Analysing the statement of Corollary \ref{cor: permutation coding and reductions} from an information-theoretic perspective, if the labelling $\mathcal{I}\mapsto g(\mathcal{I})$ was known, we could check the existence of $\mathbf{a}$
from $g(\mathcal{I})=\Delta_{\mathbf{a}}(\mathcal{I})\cdot\Delta_{\mathbf{q}}(\mathcal{I})$ and the knowledge about $\mathbf{q}:=\mathbf{Q}(\mathbf{1})$, obtaining $\Delta_{\mathbf{a}}(\mathcal{I})$ and checking the Grassmann-Pl\"{u}cker relations. On the other hand, Corollary \ref{cor: permutation coding and reductions} suggests an approach to check if a (possibly implicitly defined) permutation $\Psi$ of $\mathfrak{G}$ is induced by a permutation $\psi\in\mathcal{S}_{n}$: products $g(\mathcal{I})\cdot\mathbf{t}^{\Psi(\mathcal{I})}$ in (\ref{eq: terms of Cauchy-Binet expansion for permutation coding}) encode accessible information about the permutation $\Psi$ and, based on the previous observation, we have the freedom to choose different evaluation points $\mathbf{t}_{0}$ to infer information about the candidate $\mathbf{a}$ and $\Psi$ from these products. 

We will refer to a given choice of the evaluation point $\mathbf{t}_{0}$ as a \emph{query}, which returns the \emph{unlabeled} multiset of values, that is, a tuple with components $g(\mathcal{I})\cdot\mathbf{t}_{0}^{\Psi(\mathcal{I})}$, $\mathcal{I}\in\mathfrak{G}$, ordered without explicit dependence on $\mathcal{I}$. For a given query $\mathbf{t}_{0}$ and a permutation $\Psi\in\mathcal{S}_{\mathfrak{G}}$, we set 
 \begin{equation}
\mathcal{G}(\Psi;\mathbf{t}_{0}):=\left\{ \left| g(\mathcal{I})\cdot\mathbf{t}_{0}^{\Psi(\mathcal{I})}\right|\right\} \setminus\{0\},
\quad
\Delta(\Psi;\mathbf{t}_{0}):=\sum_{\mathcal{I}\in\mathfrak{G}}g(\mathcal{I})\cdot\mathbf{t}_{0}^{\Psi(\mathcal{I})}.
\label{eq: CB terms and quasi-determinantal coupling permutation}
 \end{equation}
In this framework, an approach to check the hypothesis \texttt{Hp}: ``\textit{$\Psi$ is induced by a permutation $\psi_{\mathtt{Hp}}\in\mathcal{S}_{n}$ of the elements of $[n]$}'' proceeds through the algorithm that we describe after the following proposition, which also applies to the constant matrices $\mathbf{L}=\mathbf{L}(\mathbf{1})$ and $\mathbf{R}=\mathbf{R}(\mathbf{1})$. 
\begin{prop}
\label{prop: sign ambiguity removal}
Let there exist an observable set $\chi(\mathcal{I}|_{\alpha\beta}^{ij})$ whose corresponding polynomial $F_{\alpha\beta}^{ij}$ in (\ref{eq: polynomial h to Y}) has distinct roots. Then, starting from the family of minor products $\{ h(\mathcal{I}):\,\mathcal{I}\in\mathfrak{G}(\mathbf{L}) \}$, each choice of a root of $F_{\alpha\beta}^{ij}$ for $Y_{\alpha\beta}^{ij}$ determines all the $Y$-terms associated with observable sets.
\end{prop}
\begin{proof}
Let us assume that there exists an observable set $\chi(\mathcal{I}|_{\alpha\beta}^{ij})$ such that both the roots of $F_{\alpha\beta}^{ij}$, call them $\ddot{Y}_{\alpha\beta}^{ij}$ and $ \widehat{Y}_{\alpha\beta}^{ij}$ with $\ddot{Y}_{\alpha\beta}^{ij}\neq \widehat{Y}_{\alpha\beta}^{ij}$, are compatible with the family $\{h(\mathcal{I}):\,\mathcal{I}\in\mathfrak{G}(\mathbf{L})\}$. Then, consider any $\gamma\in\mathcal{I}^{\mathtt{C}}$ such that $\chi(\mathcal{I}|_{\alpha\gamma}^{ij})$ and $\chi(\mathcal{I}|_{\beta\gamma}^{ij})$ are observable too: supposing that the same root $Y_{\beta\gamma}^{ij}$ of $F_{\beta\gamma}^{ij}$ is associated with both the choices $\ddot{Y}_{\alpha\beta}^{ij}$ and $\widehat{Y}_{\alpha\beta}^{ij}$ for $Y_{\alpha\beta}^{ij}$, the aforementioned compatibility imposes that  
$
F_{\alpha\gamma}^{ij}(-\ddot{Y}_{\alpha\beta}^{ij}\cdot Y_{\beta\gamma}^{ij}) = 
F_{\alpha\gamma}^{ij}(-\widehat{Y}_{\alpha\beta}^{ij}\cdot Y_{\beta\gamma}^{ij}) = 0 
$ can hold only if the arguments $-\ddot{Y}_{\alpha\beta}^{ij}\cdot Y_{\beta\gamma}^{ij}$ and $-\widehat{Y}_{\alpha\beta}^{ij}\cdot Y_{\beta\gamma}^{ij}$ are the distinct roots of the quadratic polynomial $F_{\alpha\gamma}^{ij}$. Vieta's formula gives
\begin{eqnarray}
    \frac{h(\mathcal{I}_{\alpha}^{i})\cdot h(\mathcal{I}_{\beta}^{j})}{h(\mathcal{I}_{\beta}^{i})\cdot h(\mathcal{I}_{\alpha}^{j})}\cdot(Y_{\beta\gamma}^{ij})^{2} & = & \ddot{Y}_{\alpha\beta}^{ij}\cdot Y_{\beta\gamma}^{ij} \cdot \widehat{Y}_{\alpha\beta}^{ij}\cdot Y_{\beta\gamma}^{ij} 
    = \frac{h(\mathcal{I}_{\alpha}^{i})\cdot h(\mathcal{I}_{\gamma}^{j})}{h(\mathcal{I}_{\gamma}^{i})\cdot h(\mathcal{I}_{\alpha}^{j})} \nonumber \\ 
    & \Rightarrow & 
    (Y_{\beta\gamma}^{ij})^{2} = \frac{h(\mathcal{I}_{\beta}^{i})\cdot h(\mathcal{I}_{\gamma}^{j})}{h(\mathcal{I}_{\gamma}^{i})\cdot h(\mathcal{I}_{\beta}^{j})} = \ddot{Y}_{\beta\gamma}^{ij}\cdot \widehat{Y}_{\beta\gamma}^{ij}
\label{eq: domino conjugation}
\end{eqnarray}
so $\ddot{Y}_{\beta\gamma}^{ij} = \widehat{Y}_{\beta\gamma}^{ij}$ since $Y_{\beta\gamma}^{ij}\in\{\ddot{Y}_{\beta\gamma}^{ij}, \widehat{Y}_{\beta\gamma}^{ij}\}$. 
Therefore, when $\ddot{Y}_{\beta\gamma}^{ij}$ is the $Y$-term associated with the pair $(\beta,\gamma)$ in the configuration that includes $\ddot{Y}_{\alpha\beta}^{ij}$, the second root $\widehat{Y}_{\beta\gamma}^{ij}$ of $F_{\beta\gamma}^{ij}$ appears in the configuration that includes $\widehat{Y}_{\alpha\beta}^{ij}$, possibly entailing $\ddot{Y}_{\beta\gamma}^{ij}=\widehat{Y}_{\beta\gamma}^{ij}$ as in (\ref{eq: domino conjugation}). On the other hand, for any observable set $\chi(\mathcal{I}|_{\alpha\beta}^{ij})$ that does not fulfil (\ref{eq: all 3 terms non-vanishing}) the corresponding $Y$-term is uniquely defined by (\ref{eq: quartic from Grassmann-Plucker, 2}), since the root $0$ is not acceptable for $Y$-terms. 

Now we introduce some notation that will be used in the rest of this proof: we say that $Y_{\alpha\beta}^{ij}$ is \emph{ambiguous} if (\ref{eq: all 3 terms non-vanishing}) holds for $\chi(\mathcal{I}|_{\alpha\beta}^{ij})$ and the two roots of $F_{\alpha\beta}^{ij}$ are distinct. On the other hand, a \emph{non-ambiguous} $Y$-term $Y_{\gamma\delta}^{lm}$ still derives from an observable set $\chi(\mathcal{I}|_{\gamma\delta}^{lm})$, but (\ref{eq: all 3 terms non-vanishing}) is violated or $F_{\gamma\delta}^{lm}$ does not have two distinct roots. The latter condition includes the degenerate cases $\gamma=\delta$ or $l=m$ that return polynomials $F_{\gamma\gamma}^{lm}(X)$ or $F_{\gamma\delta}^{ll}(X)$, respectively, that are proportional to $(X+1)^{2}$. 

Let $\mathcal{Y}$ be the set of ambiguous $Y$-terms $Y_{\alpha\beta}^{ij}$ such that, for any $c\in\mathbb{N}$, $(-1)^{c+1}\cdot Y_{\alpha\beta}^{ij}$ cannot be represented as the product of $c$ non-ambiguous $Y$-terms by iterations of (\ref{eq: associativity, d}) and (\ref{eq: associativity, u}). 
For any $Y_{\alpha\beta}^{ij}\in\mathcal{Y}$ the co-implication $Y_{\alpha\beta}^{ij}\in\mathcal{Y}\Leftrightarrow Y_{\beta\alpha}^{ij}\in\mathcal{Y}$ holds and, given  $\gamma\in\mathcal{I}^{\mathtt{C}}$ such that $\chi(\mathcal{I}|_{\alpha\gamma}^{ij})$ and $\chi(\mathcal{I}|_{\beta\gamma}^{ij})$ are observable, we find $\{ Y_{\alpha\gamma}^{ij},Y_{\beta\gamma}^{ij} \}\cap\mathcal{Y}\neq\emptyset$, say $Y_{\alpha\gamma}^{ij}\in\mathcal{Y}$. From the derivation of (\ref{eq: domino conjugation}), each choice of the root of $F_{\alpha\beta}^{ij}$ for $Y_{\alpha\beta}^{ij}$ determines the choice of the root of $F_{\alpha\gamma}^{ij}$ for $Y_{\alpha\gamma}^{ij}$: we denote this relation as $Y_{\alpha\beta}^{ij}\rightarrowtriangle Y_{\alpha\gamma}^{ij}$ and, similarly, the notation $Y_{\alpha\beta}^{ij}\rightarrowtriangle Y_{\alpha\beta}^{il}$ will be used for analogous variations of upper indices. 

The thesis holds if $\mathcal{Y}=\emptyset$, so fix $Y_{\gamma_{1}\gamma_{2}}^{m_{1}m_{2}}\in\mathcal{Y}$ and take any other $Y_{\alpha\beta}^{ij}\in\mathcal{Y}$. From $Y_{\gamma_{1}\gamma_{2}}^{m_{1}m_{2}}=-Y_{\gamma_{1}\alpha}^{m_{1}m_{2}}\cdot Y_{\alpha\gamma_{2}}^{m_{1}m_{2}}$, by definition there is at least one $w\in[2]$ such that $Y_{\gamma_{w}\alpha}^{m_{1}m_{2}}\in\mathcal{Y}$, say $w=1$. Analogously, 
for at least one $u\in[2]$ we get $Y_{\gamma_{1}\alpha}^{m_{u}i}\in\mathcal{Y}$, say $u=1$. Whether $\left\{ Y_{\gamma_{1}\alpha}^{ij},Y_{\alpha\beta}^{m_{1}i}\right\} \cap\mathcal{Y}\neq\emptyset$, say $Y_{\gamma_{1}\alpha}^{ij}\in\mathcal{Y}$, we get 
$
Y_{\gamma_{1}\gamma_{2}}^{m_{1}m_{2}}\rightarrowtriangle Y_{\gamma_{1}\alpha}^{m_{1}m_{2}}\rightarrowtriangle Y_{\gamma_{1}\alpha}^{m_{1}i}\rightarrowtriangle Y_{\gamma_{1}\alpha}^{ij}\rightarrowtriangle Y_{\alpha\beta}^{ij}
$,
otherwise we infer $Y_{\gamma_{1}\alpha}^{m_{1}j},Y_{\gamma_{1}\beta}^{m_{1}i}\in\mathcal{Y}$ and 
\begin{equation}
Y_{\gamma_{1}\gamma_{2}}^{m_{1}m_{2}}\rightarrowtriangle Y_{\gamma_{1}\alpha}^{m_{1}m_{2}}\rightarrowtriangle Y_{\gamma_{1}\alpha}^{m_{1}i},\quad Y_{\gamma_{1}\alpha}^{m_{1}i}\rightarrowtriangle Y_{\gamma_{1}\alpha}^{m_{1}j},\quad Y_{\gamma_{1}\alpha}^{m_{1}i}\rightarrowtriangle Y_{\gamma_{1}\beta}^{m_{1}i}.
\label{eq: chain of derivations, b1}
\end{equation}
Finally, when $Y_{\gamma_{1}\beta}^{ij}\in\mathcal{Y}$, (\ref{eq: chain of derivations, b1}) gives
\begin{equation}
Y_{\gamma_{1}\gamma_{2}}^{m_{1}m_{2}}\rightarrowtriangle Y_{\gamma_{1}\alpha}^{m_{1}m_{2}}\rightarrowtriangle Y_{\gamma_{1}\alpha}^{m_{1}i}\rightarrowtriangle Y_{\gamma_{1}\beta}^{m_{1}i}\rightarrowtriangle Y_{\gamma_{1}\beta}^{ij}\rightarrowtriangle Y_{\alpha\beta}^{ij}
\label{eq: chain of derivations, b2}
\end{equation}
otherwise we conclude $Y_{\gamma_{1}\beta}^{m_{1}j}\in\mathcal{Y}$ too: therefore, $Y_{\gamma_{1}\beta}^{m_{1}i}\rightarrowtriangle Y_{\gamma_{1}\beta}^{m_{1}j}$ and, together with (\ref{eq: chain of derivations, b1}), $Y_{\gamma_{1}\gamma_{2}}^{m_{1}m_{2}}$ determines each term $Y_{\gamma_{1}\Omega}^{m_{1}A}$ for all $A\in\{i,j\}$ and $\Omega\in\{\alpha,\beta\}$, as well as the product $-Y_{\alpha\gamma_{1}}^{im_{1}}\cdot Y_{\gamma_{1}\beta}^{im_{1}}\cdot Y_{\alpha\gamma_{1}}^{m_{1}j}\cdot Y_{\gamma_{1}\beta}^{m_{1}j}=Y_{\alpha\beta}^{ij}$. 
\end{proof}

In particular, when $\mathfrak{G}(\mathbf{L})=\wp_{k}[n]$ and for any $Y_{\alpha\beta}^{ij}$, both the choices of a root of $F_{\alpha\beta}^{ij}$ are acceptable, as they correspond to the exchange of the roles of $\mathbf{L}$ and $\mathbf{R}$, i.e. the transposition $(\mathbf{L},\mathbf{R})\mapsto (\mathbf{R}^{\mathtt{T}},\mathbf{L}^{\mathtt{T}})$, which preserves the terms of the determinantal expansion.

\begin{proof}{(of Theorem \ref{thm: check and recover permutation reductions})} 
\begin{enumerate} 
\item Using a query $\mathbf{t}_{\star}$, we gain information on the candidates $\mathbf{a}$ and $\mathbf{q}$, specifically, we can choose $\mathbf{t}_{\star}:=\mathbf{1}$ to obtain two bounds 
\begin{equation} 
0<\lambda < \min_{>0}\left(\mathcal{G}(\Psi;\mathbf{1})\cup \{|g(\mathcal{I})-g(\mathcal{J})|:\,\mathcal{I},\mathcal{J}\in \mathfrak{G}\}\right),
\quad \mu > \max\,\mathcal{G}(\Psi;\mathbf{1}).
\label{eq: upper bound for structural matrix} 
\end{equation}
\item After relabelling the reference basis as $[k]$, choose $\mathbf{t}_{0}$ so that 
\begin{equation}
\left|t_{0,s+1}\right|\cdot\left|t_{0,1}^{k-1}\right|\cdot\left|t_{0,s}^{-k}\right| > (2\cdot\#\mathfrak{G}-1)\cdot \mu\cdot\lambda^{-1},\quad s\in [n-1].
\label{eq: ordering of variables for labelling}
\end{equation} 
This provides a strict order $>$ on $\mathcal{G}(\Psi;\mathbf{t}_{0})$
that is compatible with the lexicographic order induced by $|t_{0,n}| > \dots > |t_{0,1}|$. Whether $\Psi_{1}(\mathcal{I})>\Psi_{2}(\mathcal{J})$ according to this lexicographic order, we can define $z:=\max(\Psi_{1}(\mathcal{I})\setminus\Psi_{2}(\mathcal{J}))$ and find 
\begin{eqnarray}
\left|\mathbf{t}_{0}^{\Psi_{1}(\mathcal{I})-\Psi_{2}(\mathcal{J})}\right|& \geq & \left|t_{0,z}\cdot t_{0,1}^{k-1}\right|\cdot\left|t_{0,\max\Psi_{2}(\mathcal{J})\setminus\Psi_{1}(\mathcal{I})}^{-k}\right|\nonumber \\
& \geq & \left|t_{0,z}\cdot t_{0,1}^{k-1}\right|\cdot\left|t_{0,z-1}^{-k}\right|.
\label{eq: bounds for differences in permutations}
\end{eqnarray}
In this way, the sums $\sum_{|g|\in\mathcal{Z}}g$ with $\mathcal{Z}\subseteq\mathcal{G}(\Psi;\mathbf{t}_{0})$ are pairwise different. Given $g\circ\Psi_{1}^{-1}\neq g\circ\Psi_{2}^{-1}$, after matching identical terms, label the two permutations so that $\Psi_{1}(\mathcal{Z})>\Psi_{2}(\mathcal{Z})$, consider $\mathcal{Z}:=\mathrm{argmax}\{\Psi_{1}(\mathcal{I}):\,\mathcal{I}\in\mathfrak{G},\,\Psi_{1}(\mathcal{I})\neq\Psi_{2}(\mathcal{I})\}$, and recall the notation $z:=\max\left(\Psi_{1}(\mathcal{Z})\setminus\Psi_{2}(\mathcal{Z})\right)$; then, setting $g_{1,2}(\mathcal{I}):= g(\mathcal{I}) - g\left(\Psi_{2}^{-1}(\Psi_{1}(\mathcal{I}))\right)$, we have  
\begin{eqnarray}
& & \left|\Delta(\Psi_{1};\mathbf{t}_{0})-\Delta(\Psi_{2};\mathbf{t}_{0})\right| \nonumber \\ 
& \geq & \left|g_{1,2}(\mathcal{Z})\right|\cdot\left(\left|\mathbf{t}_{0}^{\Psi_{1}(\mathcal{Z})}\right|-\left|\mathbf{t}_{0}^{\Psi_{2}(\mathcal{Z})}\right|\right)\nonumber \\ 
& &-\sum_{\mathcal{I}\in\mathfrak{G}, \Psi_{1}(\mathcal{Z})>\Psi_{1}(\mathcal{I})}\left|g_{1,2}(\mathcal{I})\right|\cdot\left(\left|\mathbf{t}_{0}^{\Psi_{1}(\mathcal{I})}\right|+\left|\mathbf{t}_{0}^{\Psi_{2}(\mathcal{I})}\right|\right)
\label{eq: sequence of triangle inequalities}
\end{eqnarray}
by the triangle inequality. Having $|g(\mathcal{Z}) - g(\Psi_{2}^{-1}(\Psi_{1}(\mathcal{Z})))|\geq\lambda$ and $|g(\mathcal{I}) - g(\Psi_{2}^{-1}(\Psi_{1}(\mathcal{I})))|<2\cdot \mu$, we conclude
\begin{eqnarray}
& &\frac{\left|\Delta(\Psi_{1};\mathbf{t}_{0})-\Delta(\Psi_{2};\mathbf{t}_{0})\right|}{\left|g_{1,2}(\mathcal{Z})\right|\cdot\left|\mathbf{t}_{0}^{\Psi_{1}(\mathcal{Z})}\right|}
\overset{{\scriptstyle \text{(by (\ref{eq: bounds for differences in permutations}))}}}{>} 
1-(2\cdot \#\mathfrak{G}-1)\cdot\frac{\mu}{\lambda}\cdot \max_{2\leq z\leq n}\left|t_{0,z}^{-1}\cdot t_{0,1}^{-k+1}\cdot t_{0,z-1}^{k}\right|>0.
\end{eqnarray}

In particular, the order obtained from this query produces the mapping $\Gamma(\mathcal{A}):=|g(\Psi^{-1}(\mathcal{A}))\cdot \mathbf{t}_{0}^{\mathtt{Ind}(\mathcal{A})}|$ with $\mathcal{A}\in\mathfrak{G}$: being the terms in $\mathcal{G}(\Psi;\mathbf{t}_{0})$ distinct, $\Gamma$ is injective.
\item The ordering of $\mathcal{G}(\Psi;\mathbf{t}_{0})$ is total since $\Gamma$ is injective, so we consider two terms $g_{u}$, $g_{u+1}$ whose labels differ by one element. By (\ref{eq: bounds for differences in permutations}), we can identify the sets $\Psi(\mathcal{I})$ and $\Psi(\mathcal{J})$ such that $g_{u}=\Gamma(\Psi(\mathcal{I}))$ and $g_{u+1}=\Gamma(\Psi(\mathcal{J}))$. 
Due to the exchange relation (\ref{eq: exchange relation}) with $\mathcal{A}:=\Psi(\mathcal{J})$, $\mathcal{B}:=\Psi(\mathcal{I})$, and $\alpha:=\max\left(\Psi(\mathcal{J})\setminus\Psi(\mathcal{I})\right)$, there is $\beta\in\Psi(\mathcal{I})\setminus\Psi(\mathcal{J})$ such that $\Psi(\mathcal{J})_{\beta}^{\alpha}\in\mathfrak{G}$. From $\Gamma(\Psi(\mathcal{J}))>\Gamma(\Psi(\mathcal{I}))$ we infer $\alpha>\beta$; this means $\Gamma(\Psi(\mathcal{J}))>\Gamma(\Psi(\mathcal{J})_{\beta}^{\alpha})$, forcing $\Psi(\mathcal{J})_{\beta}^{\alpha}=\Psi(\mathcal{I})$. We can specify
$\mathcal{I}^{(k-1)}:=\Gamma^{-1}(\min\mathcal{G}(\Psi;\mathbf{t}_{0}))$ and, for all $k\leq u\leq n-1$, 
\begin{equation} 
T_{u}:=\mu\cdot\prod_{\beta=u-k+1}^{u}|t_{0,\beta}|,\quad\mathcal{I}^{(u)}:=\Gamma^{-1}\left(\min\left(\mathcal{G}(\Psi;\mathbf{t}_{0})\cap]T_{u},\infty[\right)\right).
\end{equation}
so that $T_{u}$ is an upper bound for $\Gamma(\Psi(\mathcal{I}))$ whether $\max\Psi(\mathcal{I})\leq u$; under the hypothesis \texttt{Hp}, we get 
\begin{equation}
\forall u\in[k+1;n]:\quad \mathcal{I}^{(u-1)}\setminus\bigcup_{w=k-1}^{u-2}\mathcal{I}^{(w)}=\{u\}.
\label{eq: check on out-of-row cardinality}
\end{equation}
Similarly, we can iteratively construct $\psi_{\mathtt{Hp}}(u)$ for all $u\in[k]$ using the recovered labels by taking $(g,\Psi,\psi_{\mathtt{Hp}}):=(g\circ\Psi^{-1},\mathrm{id},\mathrm{id})$ in the following steps.
\item The previous steps hold independently on the determinantal structure we are investigating: now we can check if the query has returned the terms of a determinantal expansion, making use of the previous results. If we attest the condition (\ref{eq: check on out-of-row cardinality}), then we can verify with a single check whether $\Psi$ coincides with the function $\Psi_{\mathtt{Hp}}:=\sum_{u\in\mathcal{I}}\psi_{\mathtt{Hp}}(u)$ generated from $\psi_{\mathtt{Hp}}$ by additivity: indeed, $\Psi_{\mathtt{Hp}}$ gives a candidate determinant expansion by Corollary \ref{cor: permutation coding and reductions}, whose hypotheses can be checked with the knowledge we have recovered in the previous steps; in turn, $\Psi=\Psi_{\mathtt{Hp}}$ is equivalent to $\Delta(\Psi;\mathbf{t}_{0})=\Delta(\Psi_{\mathtt{Hp}};\mathbf{t}_{0})$ by construction of $\mathbf{t}_{0}$ in the previous step. Along with the following steps, we can control this last equality to tell whether $\Psi$ can be represented by a permutation $\psi\in\mathcal{S}_{n}$. Furthermore, recalling the invariance (\ref{eq: action of special gauge on R}) that preserves the terms of the expansion up to a permutation, the same output is returned by $\det(\mathbf{a}(\psi_{\mathtt{Hp}})\cdot \mathrm{diag}(t_{1},\dots,t_{n})\cdot\mathbf{q}(\psi_{\mathtt{Hp}}))$, where 
\begin{equation} 
\mathbf{a}(\psi_{\mathtt{Hp}}):= \mathbf{a}\cdot \mathbf{M}(\psi_{\mathtt{Hp}}^{-1}),\quad \mathbf{q}(\psi_{\mathtt{Hp}}):= \mathbf{M}(\psi_{\mathtt{Hp}})\cdot \mathbf{q}
\label{eq: transposed matrices}
\end{equation}
and $\mathbf{M}(\psi)$ denotes the matrix representation of the permutation $\psi$. Corollary \ref{cor: permutation coding and reductions} states that the factor multiplying $\mathbf{t}_{0}^{\psi_{\mathtt{Hp}}(\mathcal{I})}$ corresponds to the minor product $\Delta_{\mathbf{a}(\psi_{\mathtt{Hp}})}(\mathcal{I})\cdot\Delta_{\mathbf{q}(\psi_{\mathtt{Hp}})}(\mathcal{I})$. This lets us recover the mapping $\mathcal{I}\mapsto g({\mathcal{I}})$ for the transformed matrices (\ref{eq: transposed matrices}). 
\item Taking into account the position of the sets $\mathcal{I}^{(u)}$ in the output list, from the labelling $\mathcal{A}\mapsto\Gamma(\mathcal{A})\mapsto g(\Psi_{\mathtt{Hp}}^{-1}(\mathcal{A}))$ and the injectivity of $\Gamma$, we obtain the association $\alpha\mapsto \psi_{\mathtt{Hp}}(\alpha)$, $\alpha\in[n]$, using (\ref{eq: check on out-of-row cardinality}).
\item Finally, we can check the existence and recover a decomposition given by the matrices $\mathbf{a}$ and $\mathbf{q}$: if we know $\mathbf{q}$, then we easily get $\Delta_{\mathbf{a}}(\mathcal{I}):=g(\mathcal{I})\cdot \Delta_{\mathbf{q}}(\mathcal{I})^{-1}$, which can be tested to verify whether they are the maximal minors of a matrix $\mathbf{a}$, using Grassmann-Pl\"{u}cker relations, and Assumption \ref{claim: generic matrices}. Even without knowledge of $\mathbf{q}$, we can use explicit information on $\mathfrak{G}$ provided by the previous steps: first, we check assumption (\ref{eq: two generic columns condition}); if it is verified, then we construct the terms $Y_{\alpha_{1}\omega}^{im}$ for each $(m,\omega)\in\mathcal{I}\times\mathcal{I}^{\mathtt{C}}$ by using Proposition \ref{prop: sign ambiguity removal} for observable sets and getting the remaining ones through the procedure defined in the proof of Theorem \ref{thm: rigidity theorem monomial}. 
In this way, we recover $\mathbf{q}$ in the form (\ref{eq: action of special gauge on R}) and, from this information, we retrieve $\mathbf{a}$ using $\Delta_{\mathbf{a}}(\mathcal{I})=g(\mathcal{I})\cdot \Delta_{\mathbf{q}}(\mathcal{I})^{-1}$ as before. 
\end{enumerate}
\end{proof}
The method described above lets us distinguish a permutation $\psi$ of $[n]$, if it exists, even when the explicit form of the permutation $\Psi$ of $\mathfrak{G}$ is unknown, that is, when we do not have information on $\Psi$ beyond $\mathcal{G}(\Psi;\mathbf{t}_{0})$ in (\ref{eq: CB terms and quasi-determinantal coupling permutation}). In fact, the input, the intermediate steps, and the output produced by a query involve only unlabelled lists of minor products. 

The method uses two queries, namely, $\mathbf{1}$ (introduced in Step 1) and $\mathbf{t}_{0}$ (introduced in Step 2): even if the two permutations $\Psi_{\star}$ and $\Psi$ involved in these steps differ, the method still provides an answer on the existence and form of a combinatorial reduction for the second query $\Psi$, as $\mathbf{1}$ returns information on $\mathbf{a}$ that is independent of $\Psi$. Even Assumption \ref{claim: generic matrices}, which entails Theorem \ref{thm: rigidity theorem monomial}, holds independently of the exponents $\Psi(\mathcal{I})$ and the matrix $\mathbf{Q}(\mathbf{t})$, since the existence of two generic columns only refers to the matroid of $\mathbf{A}(\mathbf{1})$. This observation further stresses the robustness of the method against the uncertainty about $\Psi$, which fits into unlabelled sensing related to homogeneous coordinates of $k$-dimensional subspaces of $\mathbb{C}^{n}$, as noted in the Introduction. 

Finally, we specify Theorem \ref{thm: check and recover permutation reductions} for rational determinantal expansions, i.e. $g_{\mathcal{I}}\in \mathbb{Q}$ for all $\mathcal{I}\in\mathfrak{G}$. 
\begin{cor}
\label{cor: combinatorial reduction from two scalars when integer}
Under the hypotheses of Theorem \ref{thm: check and recover permutation reductions} and, in addition, the assumption $g(\mathcal{I})\in\mathbb{Q}$, $\mathcal{I}\in\mathfrak{G}$, we can check the occurrence of a combinatorial reduction and retrieve two matrices $\mathbf{a}$ and $\mathbf{q}$ as in (\ref{eq: transposed matrices}) only based on two scalar outputs, namely, the maximum $\max\mathcal{G}(\Psi;\mathbf{1})$ for a first query $\mathbf{1}$ and the candidate determinant $\Delta(\Psi;\mathbf{t}_{0})$ for a second query $\mathbf{t}_{0}$.
\end{cor}
\begin{proof}
We can reduce to $g(\mathcal{I})\in\mathbb{Z}$, $\mathcal{I}\in\mathfrak{G}$, by scaling each term through a common integer factor, which corresponds to the left action of $GL_{k}(\mathbb{Q})$ on $\mathbb{C}^{k\times n}$. Then, adapting the proof of Theorem~\ref{thm: check and recover permutation reductions}, we can set $\lambda=1$ and use $L:=\log (2+\max\mathcal{G}(\Psi;\mathbf{1}))>1$ to choose $\mu:=\exp(L) >\max\mathcal{G}(\Psi;\mathbf{1})$. Based on the same argument as in Step 2 in the previous proof, we further specify $\mathbf{t}_{0}$ setting 
\begin{equation}
t_{0,s}:=\exp\left(\frac{L}{k-1}\left\lceil \log(2\cdot\#\mathfrak{G})+L\right\rceil \cdot(k^{s-1}-1)\right),\quad s\in[n]
\end{equation}
where $\left\lceil \cdot\right\rceil$ denotes the ceiling function, and the exponent is $L\cdot \left\lceil\log(2\cdot\#\mathfrak{G}) + L\right\rceil\cdot (s-1)$ at $k=1$. Indeed, we obtain $\mathbf{t}_{0}\in\mathbb{N}^{n}$ and 
\begin{equation}
\log t_{0,s+1}-k\cdot \log t_{0,s}=L\cdot\left\lceil \log(2\cdot\#\mathfrak{G})+L\right\rceil >\log(2\cdot\#\mathfrak{G}-1)+ L 
\end{equation} 
which is equivalent to (\ref{eq: ordering of variables for labelling}). We split the base-$\mathrm{e}^{L}$ representation of $\Delta(\Psi;\mathbf{t}_{0})$ into blocks of length $\lceil \log(2\cdot\#\mathfrak{G})+L\rceil$, numbering their position starting from the least significant block: each non-vanishing block encodes a term $g(\Psi^{-1}(\mathcal{I}))$, and the label $\mathcal{I}$ of the block at position $P$ is obtained by the expansion of $(k-1)\cdot P+k$ in base $k$. Adapting Steps 4 and 6 of the previous proof, we get the thesis.
\end{proof}






\end{document}